\documentclass[12pt,a4paper,twoside]{book}

\usepackage{epsf}
\usepackage{epsfig}
\usepackage{graphicx}
\usepackage{graphics}

\usepackage{amsmath}
\usepackage{amssymb}
\usepackage{amsfonts}
\usepackage{pifont}
\usepackage{gensymb}

\newtheorem{teo}{Theorem}[chapter]
\newtheorem{coro}{Corollary}[chapter]
\newtheorem{ex}{Example}[chapter]
\newtheorem{obs}{Remark}[chapter]
\newtheorem{lema}{Lemma}[chapter]
\newtheorem{defi}{Definition}[chapter]
\newtheorem{prop}{Proposition}[chapter]
\newenvironment{proof}[1][Proof]{\noindent\textbf{#1.} }{\
\rule{0.5em}{0.5em}}
\newcommand{\C}{{\mathbb C}}
\newcommand{\R}{{\mathbb R}}
\newcommand{\Q}{{\mathbb Q}}
\newcommand{\Z}{{\mathbb Z}}
\newcommand{\N}{{\mathbb N}}
\newcommand{\D}{{\mathbb D}}
\newcommand{\Pp}{{\mathbb P}}

\newcommand{\fol}{{\mathcal F}}

\newcommand{\tXX}{{\widetilde{X}}}

\newcommand{\wdc}{{\widetilde{\mathbb C}}}

\newcommand{\cpd}{{{\mathbb C}{\mathbb P}\, (2)}}
\newcommand{\spfl}{{\rm Fol}\, (\mathbb{C}\mathbb{P}\, (2), d)}
\newcommand{\llz}{\overline{\overline{\zeta}}}
\newcommand{\cpum}{{{\mathbb C} {\mathbb P}\, (1)}}
\newcommand{\fracx}{{\frac{\partial}{\partial x}}}
\newcommand{\fracy}{{\frac{\partial}{\partial y}}}
\newcommand{\fracz}{{\frac{\partial}{\partial z}}}
\newcommand{\calp}{{\mathcal P}}

\newcommand{\da}{{\alpha}}
\newcommand{\db}{{\beta}}

\newcommand{\dl}{{\lambda}}
\newcommand{\de}{{\varepsilon}}

\newcommand{\partx}{{\partial /\partial x}}
\newcommand{\party}{{\partial /\partial y}}

\newcommand{\qed}{\hfill \fbox{}}

\begin{document}

\title{Local Theory of Holomorphic Foliations and Vector Fields}
\author{Julio C. Rebelo \hspace{0.4cm} \& \hspace{0.4cm} Helena Reis}
\date{2010.12}
\maketitle

\cleardoublepage

\centerline{\textsc{Foreword}}

\vspace{0.8cm}

\noindent These notes constitutes a slightly informal introduction to what is often called the {\it theory of (singular) holomorphic
foliations}\, with emphasis on its local aspects that may also be called {\it singularity theory of holomorphic foliations and/or vector fields}.
Though the text is unpolished, we believe to have provided detailed proofs for all the statements given here. The material is roughly split into
two parts, namely Chapters~1 and~2. Chapter~1 is primarily devoted to motivation for the theory of ``(singular) holomorphic foliations'' and
in this direction it contains several examples of these foliations appearing in situations of interest. It also contains some general background
from analytic geometry and topology that are often used in their study. This chapter is essentially of global nature so its purpose is
indeed to motivate (and to provide some for) the study of globally defined foliations/differential equations rather than the study of their
singularities that is the object of Chapter~2. Explanations for this discrepancy begin by saying that these notes are planned to be continued
in the future with the inclusion of new chapters. Besides it is interesting to point out that the (local) study of singularities already entails a
global perspective once the dimension of the ambient is at least~$3$. The simplest way to notice the existence of global foliations (possibly with
rather complicate dynamics) hidden into the structure of a singularity of a vector field defined on, say, $(\C^3, 0)$, is to perform the blow-up of
the singularity in question (cf. Section~1.4.1). In general, the blown-up foliation will induce a globally defined foliation on the resulting exceptional divisor (isomorphic to
the complex projective plane in this case). Actually this foliation depends only on the first non-zero homogeneous component of the Taylor series
of the mentioned vector field. Besides all foliations on the projective plane are recovered in this way, see ``Example 3'' in Section~1.3.

Chapter~2 is then devoted to the singularity theory for vector fields/foliations. Most of the chapter is taken up by the case of singularities defined on
$(\C^2, 0)$ where the theory has reached a high level of development, with landmark progresses being represented by the papers of Mattei-Moussu
and of Martinet-Ramis, \cite{M-M}, \cite{Ma-R}. Much of this chapter is indeed devoted to present the results of these papers. Yet we have also included
some classical and more recent material concerning singularities in higher dimensions.

In a sense these notes are vaguely reminiscent from a course the first author lectured at Stony Brook several years ago. He would like to thank the
audience of his course and specially Andre Carvalho and Misha Lyubich for their interest. The project of writing these notes, however, would never
be undertaken if it were not by the interest of Gisela Marino to whom we are most grateful. Gisela was the first to put into typing and fill in details of definitions and proofs
for most of what is now ``Chapter~2'', cf. \cite{gisela}. Unfortunately she did not wish to continue her study and left to us to complete the present material
(as well as the planned subsequent chapters).

\vspace{0.3cm}

\noindent \hspace{9.0cm} The authors

\tableofcontents{}

\chapter{Foundational Material}

We are interested in understanding the behavior of first order
ordinary differential equations in $\mathbb{C}^2$, where the
``time'' parameter is in $\mathbb{C}$.

To begin with, let us recall the main aspects of real ordinary
differential equations. Complex Ordinary Differential Equations
(ODEs) can then be obtained from real ones from a natural
``complexification'' procedure. It will be seen that complex ODEs
are closely related to singular holomorphic foliations.


\section{Real Ordinary Differential Equations}

Let $X:\mathbb{R}^2 \rightarrow \mathbb{R}^2$ be a vector field
given by $X(x,y)=(P(x,y),Q(x,y))$, where $P$ and $Q$ are
polynomials. The ordinary differential equation associated to this
field is:
\[
\frac{d}{dt}(x(t),y(t))=X(x(t),y(t)) \, .
\]
It can also be regarded as the following system of ODEs
\begin{equation}\label{eq1}
\left \{ \begin{array}{l}
          \frac{d}{dt}x(t)=P(x(t),y(t)) \, ,\\
          \\
          \frac{d}{dt}y(t)=Q(x(t),y(t)) \, .
          \end{array}
\right.
\end{equation}

Since the vector field is sufficiently smooth ($C^\infty$, in
fact) we may assert that given a $t_0 \in \mathbb{R}$ and
$(x_0,y_0) \in \mathbb{R}^2$, there exists a unique solution
$\phi(t)=(\phi_1(t),\phi_2(t))$ of (\ref{eq1}) defined on a
neighborhood of $t_0$, and verifying
\[
\phi(t_0)=(\phi_1(t_0),\phi_2(t_0))=(x_0,y_0) \, .
\]
Furthermore, there exists the concept of extending a local
solution to obtain a solution defined on a maximal domain, i.e, on
the ``largest possible interval''. This means that for each
$(x_0,y_0) \in \mathbb{R}^2$ there exists an interval $I(x_0,y_0)$
and a solution $\varphi$ of (\ref{eq1}) defined on $I(x_0,y_0)$
which satisfies the following conditions:
\begin{enumerate}
\item $\varphi(t_0)=(x_0,y_0)\, .$

\item If $\psi$ defined on $I\subseteq\mathbb{R}$, $t_0\in I$, is
another solution of (\ref{eq1}) verifying $\psi(t_0)=(x_0,y_0)$,
then, $I\subset I(x_0,y_0)$ and, furthermore, $\varphi |_I=\psi$.
\end{enumerate}

Let $\Omega\subseteq\mathbb{R}\times\mathbb{R}^2$ be the open set
$\{(t,x_0,y_0)\in\mathbb{R}^3; t\in I(x_0,y_0)\}$. The {\it flow}
associated to (\ref{eq1}) is defined to be:
\begin{eqnarray*}
\Phi:\Omega &\rightarrow& \mathbb{R}^2\\
(t,x_0,y_0)&\mapsto& \phi(t) \, ,
\end{eqnarray*}
where $\phi(t)$ is the solution of (\ref{eq1}) such that
$\phi(t_0)=(x_0,y_0)$.

Notice that the flow may not be {\it complete}. That is, it may
not be defined for all $t\in \mathbb{R}$, since $I(x_0,y_0)$ is
not necessarily the whole real line. On the other hand, it is easy
to check that, when $I(x_0,y_0)\neq\mathbb{R}$ then
$\Phi(t,x_0,y_0)$ tends to infinity as $t$ approaches the
extremities of $I(x_0,y_0)$, for a fixed $(x_0,y_0)$. Here we say
that $\Phi(t,x_0,y_0)$ ``tends to infinity'' in the sense that it
leaves every compact set contained in the domain of definition of
$X$.

\begin{obs}\label{obs1}
{\rm The above discussion actually holds for every
regular (say $C^1$) vector field on arbitrary manifolds. Also it
is well-known that, if the orbits of a regular vector field are
contained on a compact set, then the maximal interval of
definition for the corresponding solutions is, indeed,
$\mathbb{R}$. In other words the flow generated by this vector
field is complete. In particular every regular vector field
defined on a compact manifold (without boundary) is complete.}
\end{obs}

Let us now return to our polynomial vector field $X=(P,Q)$. What
precedes implies that the image of $\Phi$ decomposes
$\mathbb{R}^2$ into a set of curves (orbits of $X$) along with the
singular points of $X$. To develop this remark further, let us
recall the definition of (regular, real) foliations.

\begin{defi}
\label{foliationdef} Consider a manifold $M$ of real dimension
$n$. A foliation $\fol$ of class $C^r$ and of dimension $k$ $(1 \leq
k < n)$ on $M$ consists of a distinguished coordinate covering $\{
U_i ,\psi_i \}$, $i \in I$, of $M$ satisfying the conditions
below:
\begin{enumerate}

\item If $i \in I$, then $\psi_i (U_i) = U_i^1 \times U_i^2
\subset \R^k \times \R^{n-k}$, where $U_i^1, U_i^2$ are open discs
of $\R^k$ and $\R^{n-k}$ respectively.

\item If $i,j \in I$ and $U_i \cap U_j \neq \emptyset$, then the
change of coordinates $\psi_i \circ \psi_j^{-1} : \psi_j (U_i \cap
U_j) \rightarrow \psi_i (U_i \cap U_j)$ has the form $\psi_i \circ
\psi_j^{-1} (x,y) = (h_1(x,y) , h_2(y))$ where $x \in \R^k$ and $y
\in \R^{n-k}$.
\end{enumerate}
\end{defi}

Naturally a distinguished coordinate covering for a foliation
$\fol$ can automatically be enlarged to a maximal {\it foliated
atlas}. A distinguished coordinate $\psi_i : U_i \rightarrow \R^k
\times \R^{n-k}$ is sometimes called {\it a foliated chart}, {\it
a foliated coordinate} or even {\it a trivializing coordinate for
$\fol$}. Given a foliated chart $\psi$ as above, a set of the form
$\psi_i^{-1} ( U_i^1 \times {\rm cte})$ is called a {\it plaque}.
A {\it plaque chain}\, is a sequence of plaques $\alpha_1, \ldots
,\alpha_l$ such that $\alpha_i \cap \alpha_{i+1} \neq \emptyset$
for every $i \in \{ 1, \ldots ,l-1\}$. We then introduce an
equivalence relation between points of $M$ by stating that $p \in
M$ is equivalent to $q \in M$ if there is a plaque chain
$\alpha_1, \ldots ,\alpha_l$ such that $p \in \alpha_1$ and $q \in
\alpha_l$. The classes of equivalence of this relation are said to
be the {\it leaves}\, of $\fol$.

\begin{obs}
\label{obsfolholo} {\rm If $M$ is a complex manifold and the
changes of coordinates for a foliated atlas are, in fact,
holomorphic diffeomorphisms then we have a holomorphic foliation.}
\end{obs}

With this terminology, we return to the vector field $X$. A
standard fact about Ordinary Differential Equations is the
so-called Flow Box Theorem. It states the existence of a
diffeomorphism $R$ defined on a neighborhood $V$ of a non-singular
point of $X$, such that $R_* X=e_1$, where $(e_1,\ldots,e_n)$
denotes the canonical basis of $\mathbb{R}^n$. Once again, this
result holds for $C^r$ vector fields ($r\geq 1$) and the resulting
diffeomorphism $R$ has the same regularity $C^r$ of the vector
field.

In particular, away from the singular points of $X$, the Flow Box
Theorem implies that solutions of ODEs are the leaves of a
foliation of dimension~$1$. In this sense, we say that the image
of $\Phi$ defines a {\it singular foliation} on $\mathbb{R}^2$. We
shall make this definition more formal in the sequel.


\section{Complex Ordinary Differential Equations}

We now wish to extend these notions to the complex case. The main
idea is to identify $\mathbb{C}^{n}$ with $\mathbb{R}^{2n}$ by
considering a {\it complex structure} on $\mathbb{R}^{2n}$. To do
this, we must define an automorphism of $\mathbb{R}^{2n}$ that
plays the role of the multiplication by $\sqrt{-1}$ in a complex
vector space. More precisely, a {\it complex structure}\, on
$\mathbb{R}^{2n}$ consists of an automorphism
$J:\mathbb{R}^{2n}\rightarrow\mathbb{R}^{2n}$ satisfying
$J^2=-Id$. In the sequel we are going to use the ``standard''
complexification, where $J$ is given by
$J(x_1,y_1,\ldots,x_n,y_n)=(-y_1,x_1,\ldots,-y_n,x_n)$.

Now, let $X:\mathbb{C}^2 \rightarrow \mathbb{C}^2$ given by
$X(z_1,z_2)=(P(z_1,z_2),Q(z_1,z_2))$ be a polynomial vector field.
The complex ODE associated to this field is
\begin{equation}\label{eq2}
\left \{ \begin{array}{l}
          \frac{d}{dT}z_1(T)=P(z_1(T),z_2(T)) \\
          \\
          \frac{d}{dT}z_2(T)=Q(z_1(T),z_2(T)) \, ,
          \end{array}
\right.
\end{equation}
where the time parameter $T$ is complex. Once again, $X$ being a
holomorphic vector field, the complex version of the Theorem of
Existence and Uniqueness for regular ODEs guarantees that given
$T_0\in \mathbb{C}$ and $(a,b)\in \mathbb{C}^2$ there exists a
unique holomorphic solution $\phi(T)=(\phi_1(T),\phi_2(T))$ of
(\ref{eq2}) defined on a neighborhood $B$ of $T_0$, and
satisfying:
\[
\phi(T_0)=(\phi_1(T_0),\phi_2(T_0))=(a,b) \, .
\]

The next step would be to try to glue together these local
solutions so as to obtain a ``maximal domain of definition''.
However we notice that, in general, this is not possible since the
time parameter $T$ is in $\mathbb{C}$. Indeed, the problem is
that, as we try to glue together the neighborhoods, their union
$V$ \emph{may not} be simply connected (see figure below).
Consequently, the solution $\phi(T)$ may not be well-defined on
all of $V$, i.e, it may be multi-valued. This is an important
difference between real and complex ODEs. This phenomenon is
illustrated by Figure~(\ref{figmono}), since the intersection of
$V_1$ and $V_2$ is not connected, solutions defined on $V_1$ and
$V_2$ cannot, in general, be ``adjusted'' to coincide in both
connected components.

\begin{figure}[h]
\label{figmono}
\begin{center}
\includegraphics[scale=1]{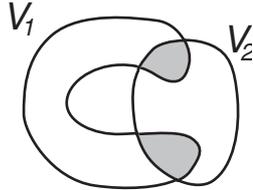}
\caption{$V=V_1\cup V_2$ is not simply connected.}
\end{center}
\end{figure}

Let us now introduce a more geometric point of view for these
topics. Using the above mentioned identification of $\mathbb{C}$ with
$\mathbb{R}^2$ (and $\mathbb{C}^2$ with $\mathbb{R}^4$), a
solution of (\ref{eq2}) starting at $(a,b)=(a_1+ia_2,b_1+ib_2)$
may be regarded locally as a ``piece'' of real $2$-dimensional
surface $L_0$ in $\mathbb{R}^4$ passing through the point
$(a_1,a_2,b_1,b_2)$. In addition, at $(a_1,a_2,b_1,b_2)$, $L_0$ is
tangent to the vector space spanned by the vectors
\begin{equation}
D_{(0,0)} (\phi_1 , \phi_2) . \left(
\begin{array}{c}
                                      1 \\ 0
                                     \end{array} \right)
\; \; ,\; \; D_{(0,0)} (\phi_1 , \phi_2) . \left( \begin{array}{c}
                                      0 \\ 1
                                     \end{array} \right) \, ,
\label{eq3}
\end{equation}
where $\phi=(\phi_1,\phi_2)$ is regarded as a real map from
$B\subseteq\mathbb{R}^2$ to $\mathbb{R}^4$, and $\phi_i$ ($i=1,2$)
satisfies the Cauchy-Riemann equations. Notice that $T_{(a,b)}L_0$
is invariant under the automorphism $J$, due to the Cauchy-Riemann
equations, so that $T_{(a,b)}L_0$ is indeed a {\it complex line},
i.e the image of a one-dimensional subspace (over $\mathbb{C}$) of
$\mathbb{C}^2\simeq\mathbb{R}^4$ under the preceding
identification.

As the initial conditions $(a,b)$ vary, from (\ref{eq3}) we obtain
a distribution of $2$-real dimensional real planes (or complex lines)
that can be integrated in the sense of Frobenius to yield
$2$-dimensional surfaces (or complex curves). In particular, away
from its singular set, the vector field defines a foliation of
real dimension equal to~$2$. Furthermore this foliation is
holomorphic as it follows from the complex version of the Flow Box
Theorem.

The leaves of the foliation in question inherit a natural
structure of Riemann surfaces. Actually an atlas for this
structure is provided by the local solutions $\phi$
of~(\ref{eq2}). More precisely, $\phi$ is a holomorphic
diffeomorphism from $B\subseteq \mathbb{C}$ to its image in the
leaf (Riemann Surface) $L_0$.

Summarizing what precedes, a holomorphic vector field on
$\mathbb{C}^2$ immediately yields a holomorphic foliation on
$\mathbb{C}^2$ away from its singular set. Again we also say that
the vector field defines a singular foliation on $\C^2$ which is
said to be its associated foliation  (or underlying foliation).
Conversely, given a (singular) holomorphic foliation $\cal{F}$, in
order to obtain the vector field whose non-constant orbits are the
leaves of $\cal{F}$, we need an extra data. More precisely, we
must associate a complex number (or vector in $\mathbb{R}^2$) to
the tangent space of each leaf, so as to recover the
parametrization of the leaves of $\cal{F}$, which in the above
situation was given by $\phi$. This complex number will be playing
the role of the ``speed'' of the flow of $X$. It allows us to
recover the local parametrizations $\phi$ for the leaves of the
foliation which are given as local solutions of~(\ref{eq2}).

Clearly the notion of singular holomorphic foliation is a
convenient geometric way to think of a complex ODE (equivalently a
holomorphic vector field). Nonetheless the above remark shows that
it does not capture all the information contained in a vector
field. As already mentioned, the local solutions $\phi$ in general
cannot be glued together and this makes the problem of extending
them, as far as possible, more subtle. We shall return to this
problem later.


\section{Basic Definitions and Examples}

After the relatively informal discussion of the previous section,
we shall now begin to provide precise definitions and detailed
statements.

\begin{defi}
A complex manifold $M^n$ of dimension $n$ is a differential
manifold equipped with an atlas $\{U_i,\psi_i\}$ such that:
\[
\psi_j\circ\psi_i^{-1}:\psi_i(U_i\cap U_j)\rightarrow \psi_j((U_i\cap
U_j)
\]
is holomorphic whenever $U_i\cap U_j\neq \varnothing$, and
$\psi_i:U_i\rightarrow V_i\subseteq\mathbb{C}^n$,
$\psi_j:U_j\rightarrow V_j\subseteq\mathbb{C}^n$.
\end{defi}

By virtue of the Cauchy-Riemann equations, the Jacobian
determinant of a holomorphic diffeomorphism is always positive. It
then follows that every complex manifold is orientable.

The Cauchy-Riemann equations also imply that a map
$F:U\subset\mathbb{R}^{2n}\rightarrow\mathbb{R}^{2n}$ is
holomorphic if and only if $DF(Jv)=J(DFv)$, for $v\in U$. So that
vectors in $\mathbb{R}^{2n}$ invariant under $J$ are sent by $DF$
to vectors that are invariant under $J$ as well. In other words,
if $F$ is holomorphic then $DF$ preserves $n$-dimensional complex
planes.

Next, we shall give a working definition of singular holomorphic
foliation on a complex manifold which is going to be used
throughout these notes.

\begin{defi}
\label{defncfoli}
A singular holomorphic foliation $\cal{F}$ defined on a complex
manifold $M^n$ consists of the following data:
\begin{enumerate}

\item There exists an atlas $\{U_i, \psi_i\}$ compatible with the
complex structure on $M^n$, where $\psi_i:U_i\rightarrow
V_i\subseteq\mathbb{C}^n$.

\item There exist holomorphic vector fields $X_i$ defined on each
$V_i$, given by $P_{1,V_i}\frac{\partial}{\partial z_1}+\ldots+
P_{n,V_i}\frac{\partial}{\partial z_n}$.

\item If $U_i\cap U_j \neq\varnothing$ then there exist functions
$h_{ij}:\psi_i(U_i\cap U_j)\rightarrow \mathbb{C}$ such that:
\[
(\psi_j\circ\psi_i^{-1})_*X_i(\psi_i(U_i\cap U_j))=h_{ij}(z_1,\ldots,z_n)X_j(\psi_j(U_i\cap 
U_j)) \, .
\]\label{d1}
\end{enumerate}
\end{defi}

When $M$ is a complex surface (i.e. $M$ has complex
dimension~$2$), we have initially thought of singular foliation as
being a regular foliation defined away from finitely many points
(the corresponding singularities). In this regard the definition
above seems to be more restrictive. Yet this is not the case.

If the functions $h_{ij}$ are actually all constant and equal to
$1$, then we have, indeed, a holomorphic vector field on $M$.

\begin{defi}
A holomorphic vector field $X$ defined on a manifold $M^n$ is such
that, given an atlas of $M^n$, $\{U_i,\psi_i\}$, the following
equation is satisfied:
\[
(\psi_j\circ\psi_i^{-1})_*X_i(\psi_i(U_i\cap U_j))=X_j(\psi_j(U_i\cap U_j)) \, ,
\]
where $\psi_i:U_i\rightarrow V_i\subset\mathbb{C}^n$ and $X_i$ are
as in Definition~\ref{defncfoli}.
\end{defi}

In the sequel we present a list of vector field and foliations
of varied nature. Our main purpose is to convince the reader of
the richness and importance of the subject. The first elementary
examples will also give us a hint that the condition required
to define a vector field on a complex manifold is much stronger
than the conditions that allow us to define a singular holomorphic
foliation.


\bigskip
\textbf{Example 1: Complex Tori}

Let $\Lambda$ be a lattice on $\mathbb{C}^n$. The $n$-dimensional
complex torus is the quotient space $\mathbb{C}^n/\Lambda$. Notice
that a constant vector field $Y$ on $\mathbb{C}^n$ induces a
holomorphic vector field on the torus. In fact, constant vector
fields are obviously preserved by the translations of $\C^n$
associated to the elements of $\Lambda$. Thus $Y$ descends as a
holomorphic vector field to the torus given by the quotient
$\mathbb{C}^n/\Lambda$.


\bigskip
\textbf{Example 2: Hopf Surfaces}

Consider $\lambda_1$, $\lambda_2$ in $\mathbb{C}^{\ast}$ such that
$\vert \, \lambda_1 \vert <1$ and $\vert \, \lambda_2 \vert<1$.
Let $\sigma(z_1,z_2)=(\lambda_1 z_1, \lambda_2 z_2)$. The Hopf
surface $M$ associated to $\sigma$ is the quotient
$(\mathbb{C}^2\setminus \{(0,0)\})/ \sigma$. It is immediate to
check that $M$ is, in fact, a complex $2$-dimensional manifold.

Let $X(z_1,z_2)=P(z_1,z_2)\frac{\partial}{\partial z_1} +
Q(z_1,z_2)\frac{\partial}{\partial z_2}$ be a polynomial vector
field on $\mathbb{C}^2\setminus\{(0,0)\}$ such that
\[
X(\sigma(z_1,z_2))=\alpha P(z_1,z_2)\frac{\partial}{\partial
z_1}+\beta P(z_1,z_2)\frac{\partial}{\partial z_2} \, ,
\]
where $\alpha / \lambda_1 = \beta / \lambda_2$. Then $X$ defines a
singular holomorphic foliation on $(\mathbb{C}^2\setminus
\{0,0\})/ \sigma$. The reader will notice, however, that in order
to obtain a holomorphic vector field in the Hopf surface, one
should have the above ratio equal to $1$. This indicates that in
Hopf surfaces, there are many more holomorphic foliations than
holomorphic vector fields.

The following concrete example illustrates this.
Let $\lambda_1=e^{-2}$ and $\lambda_2=e^{-4}$. Consider the
polynomial vector field given by $X=P(z_1, z_2) \ \partial/\partial z_1 + Q(z_1,z_2) \
\partial/\partial z_2$ where:
\begin{eqnarray*}
P(z_1,z_2)&=& z_1^3+z_1z_2 \, , \\
Q(z_1,z_2)&=& z_2^2+2z_1^2z_2 \, .
\end{eqnarray*}
Notice that
\[
X(\sigma(z_1,x_2))=e^{-6}P(z_1,z_2)\frac{\partial}{\partial
z_1}+e^{-8}Q(z_1,z_2)\frac{\partial}{\partial z_2} \, .
\]

On the other hand,
\[
D\sigma.X(z_1,z_2)= e^{-2}P(z_1,z_2)\frac{\partial}{\partial
z_1}+e^{-4}Q(z_1,z_2)\frac{\partial}{\partial z_2} \, ,
\]
so that $D\sigma.X(z_1,z_2)=e^{4}X(\sigma(z_1,z_2))$. In view of
the above discussion,  the vector field $X$ induces a holomorphic
foliation on $M$ but not a holomorphic vector field.


\bigskip
\textbf{Example 3: Complex Projective Plane (Space)}

Foliations on complex projective spaces constitute the main source
of examples, in the sense that they are easy to describe and, in
addition, they usually already encode the essential difficulties
of more general cases. For this reason, they are going to be
treated here with a significative amount of details. Let us begin
by considering the following relation of equivalence:
\[
z\sim z' \ \Leftrightarrow \
\exists \ \lambda\in\mathbb{C}^{\ast} \; \; \; \; \; \; \; ;\ \ \
\ \ \ z=\lambda z'; \ \ \ \ \ \ z, z'
\in\mathbb{C}^3\setminus\{(0,0,0)\}\, .
\]
The equivalence classes $(\mathbb{C}^3\setminus\{(0,0,0)\})/\sim$
form the $2$-dimensional complex projective space, denoted by
$\mathbb{CP} \, (2)$.

Two points $(a,b,c), (a', b' ,c') \in \mathbb{C}^3\setminus
\{(0,0,0)\}$ define the same point in $\cpd$ if and only if
\[
a/a'= b / b' = c / c'  \, ,
\]
so that the projection $\pi:\mathbb{C}^3\setminus
\{(0,0,0)\}\rightarrow \cpd$ is determined by the ratios between
the coordinates of $(a,b,c)$. Traditionally $\pi(a,b,c)$ is
denoted by $(a:b:c)$ and $a,b,c$ are called homogeneous
coordinates for $(a:b:c)$.

Let us consider the following open sets that cover $\cpd$:
\begin{eqnarray*}
U_a&=&\{(a:b:c)\in \cpd \; \; ;\ \ a\neq 0\} \, ; \\
U_b&=&\{(a:b:c)\in \cpd \; \; ;\ \ b\neq 0\} \, ; \\
U_c&=&\{(a:b:c)\in \cpd \; \; ;\ \ c\neq 0\} \, .
\end{eqnarray*}
Along with these open sets, we have the following coordinate
charts:
\begin{eqnarray*}
\varphi_a: U_a  &\rightarrow & \mathbb{C}^2 \\
(a:b:c)  &\mapsto &  (b/a,c/a)=(x,y) \, ;
\end{eqnarray*}
\begin{eqnarray*}
\varphi_b: U_b  &\rightarrow & \mathbb{C}^2 \\
(a:b:c)  &\mapsto &  (a/b,c/b)=(u,v) \, ;
\end{eqnarray*}
\begin{eqnarray*}
\varphi_c: U_c  &\rightarrow & \mathbb{C}^2 \\
(a:b:c)  &\mapsto &  (a/c,b/c)=(z,w) \, .
\end{eqnarray*}

Now let $L_{\infty}= \cpd \setminus E_a= \{(0:b:c)\in \cpd ; \ \
(b,c)\in \mathbb{C}^2\}$. A direct inspection using the coordinate
charts introduced above shows that $L_{\infty}$ is isomorphic to
$\mathbb{CP}\, (1)$. Thus $\cpd=\mathbb{C}^2\cup\mathbb{CP} \,
(1)$ i.e. $\cpd$ may be regarded as $\mathbb{C}^2$ being added to
the Riemann sphere. In this sense $\cpd$ is a compactification of
$\C^2$. Moreover in the affine coordinates $(x,y)$, $L_{\infty}$
corresponds to ``infinity'' and $L_{\infty}$ is linearly embedded
in $\cpd$. For these reasons it is called the {\it line at
infinity}. Finally, we note that this construction applies to any
affine coordinates in $\cpd$, that is, any affine $\C^2 \subset
\cpd$ gives rise to a ``line at infinity''.

We will construct singular holomorphic foliation on $\cpd$ in a
similar way to what was done in the case of Hopf surfaces.

Let us consider a homogeneous polynomial vector field
$X=P\frac{\partial}{\partial z_1}+Q\frac{\partial}{\partial
z_2}+R\frac{\partial}{\partial z_3}$ on
$\mathbb{C}^3\setminus\{(0,0,0)\}$ of degree~$d$. Consider the
action of $\C^{\ast}$ on $\C^3\setminus \{0,0,0\}$ given by the
homotheties $\sigma(z_1,z_2,z_3)=(\lambda z_1, \lambda z_2,
\lambda z_3)$, $\lambda\in\mathbb{C}^{\ast}$. As already pointed
out, the quotient of this action is precisely $\cpd$. On the other
hand, note that
\[
X(\sigma(z_1,z_2,z_3))=\lambda^d
P(z_1,z_2,z_3)\frac{\partial}{\partial z_1}+\lambda^d
Q(z_1,z_2,z_3)\frac{\partial}{\partial z_2}+\lambda^d
R(z_1,z_2,z_3)\frac{\partial}{\partial z_3} \, .
\]
Besides,
\[
D\sigma.X(z_1,z_2,z_3)=\lambda
P(z_1,z_2,z_3)\frac{\partial}{\partial z_1}+\lambda
Q(z_1,z_2,z_3)\frac{\partial}{\partial z_2}+\lambda
R(z_1,z_2,z_3)\frac{\partial}{\partial z_3} \, .
\]
Therefore $D\sigma.X=\lambda^{d-1}X$ so that ``the directions''
associated to $X$ are invariant under homotheties. Hence $X$
indeed defines a singular foliation on $\cpd$.

Another equivalent way to define a singular holomorphic foliation
in $\cpd$ is to consider a polynomial vector field in
$\mathbb{C}^2$ and see that it can be extended to an holomorphic
foliation on all of $\cpd$. Details are provided below.

In fact, it can be shown that every
holomorphic foliation in $\mathbb{CP}(2)$ is induced by a
polynomial vector field in $\mathbb{C}^2$.

Let $X=P\frac{\partial}{\partial x}+Q\frac{\partial}{\partial y}$
be a polynomial vector field in the affine coordinates $(x,y)$.
Naturally, it induces a rational vector field $Y$ (resp.$Z$)
defined on the affine coordinates $(u,v)$ (resp. $(z,w)$). Then we
only need to multiply the vector fields $Y,Z$ by their
denominators so as to have holomorphic ones (indeed polynomial
ones).

For example, $Y$ is given by:
\begin{eqnarray*}
Y(u,v)&=&(\varphi_a\circ\varphi_b)^{\ast}(X(x,y)) \\
&=& D(\varphi_a\circ\varphi_b)^{-1}.X(\varphi_a\circ\varphi_b(u,v))\\
&=&\left(%
\begin{array}{cc}
  -u^2 & 0 \\
  -uv & u \\
\end{array}%
\right)
\left(%
\begin{array}{c}
  P(1/u,v/u) \\
  Q(1/u,v/u) \\
\end{array}%
\right)\, .
\end{eqnarray*}
As already mentioned, this vector field is not holomorphic in the
domain of the coordinates $(u,v)$ since it has poles over $\{ u=0
\}$. However, multiplying $Y$ by $u^d$, the new vector field
$u^dY$ is obviously holomorphic (and with isolated singularities)
in the domain of $(u,v)$. Now changing from the coordinate system
$(u,v)$ to $(x,y)$ one obtains
\begin{eqnarray*}
D(\varphi_a\circ\varphi_b).(u^dY)&=& u^d
D(\varphi_a\circ\varphi_b).Y(u,v)\\
&=&u^dD(\varphi_a\circ\varphi_b).D(\varphi_a\circ\varphi_b)^{-1}.X(x,y)\\
&=&u^dX(x,y)\, .
\end{eqnarray*}
Repeating this procedure with the other coordinate charts, we
obtain a holomorphic foliation on $\cpd$. Finally note also that
the original polynomial vector field on $\mathbb{C}^2$ does not
induce a holomorphic vector field on $\cpd$. Instead it induces a
{\it meromorphic}\, vector field whose poles are contained in the
corresponding line at infinity.

This has an obvious generalization to higher dimensional complex
projective spaces which is left to the reader.

We have constructed singular holomorphic foliations on $\cpd$
following two {\it a priori}\, different methods:
\begin{itemize}

\item By means of a homogeneous polynomial vector field on $\C^3$.

\item By means of a polynomial vector field on $\C^2$.
\end{itemize}

It is easy to check that both constructions are equivalent in the
sense that they produce the same set of foliations. This
verification is implicitly carried out in the proof of
Lemma~(\ref{l2}). It is however much harder to show that these
constructions give rise to {\it all singular holomorphic
foliations on}\, $\cpd$. This is the contents of
Theorem~(\ref{spfoliation}) below.

\begin{teo}
\label{spfoliation} Let $\fol$ be a singular holomorphic foliation
on $\cpd$. Then there is $\fol$ is induced by a homogeneous
polynomial vector field $X$ on $\C^3$ which, in addition, has
singular set of codimension at least~$2$.
\end{teo}


In view of Theorem~(\ref{spfoliation}), it is natural to try to
define a notion of {\it degree}\, for a foliation on $\cpd$. At
first, one might feel tempted to define it as the degree of a
polynomial vector field representing the foliation in affine
coordinates. However, as the reader can easily verify, this degree
may, in fact, vary depending on the affine coordinates chosen.

The following lemma will motivate the correct definition of the
degree for a foliation $\fol$ on $\cpd$ induced by a polynomial
vector field $X=P\frac{\partial}{\partial
x}+Q\frac{\partial}{\partial y}$ on $\C^2$. Let $P=\sum_{i=0}^d
P_i(x,y)$ and $Q=\sum_{i=0}^d Q_i(x,y)$ where $P_i$, $Q_i$ are
homogeneous polynomials of degree $i$.

\begin{lema}
\label{l1} The ``line at infinity'', $L_\infty$, of
$\mathbb{CP}(2)$ is invariant under the foliation $\fol$, induced
by $X$ as above, if and only if the top-degree homogeneous
component $P_d\frac{\partial}{\partial
x}+Q_d\frac{\partial}{\partial y}$ is not of the form
$h(x,y)(x\frac{\partial}{\partial x}+y\frac{\partial}{\partial
y})$, for some polynomial $h$ of degree $d-1$.
\end{lema}

\begin{proof}
To understand the behavior of
$X=P\frac{\partial}{\partial x}+Q\frac{\partial}{\partial y}$ near
infinity in the coordinate system $(x,y)$, we use the following
change of coordinates: $u=\frac{1}{x}$, $v=\frac{y}{x}$. So that
this vector field in the coordinate chart $(u,v)$ is given by:
\begin{eqnarray*}
&X(u,v)&=\left(%
\begin{array}{cc}
  -u^2 & 0 \\
  -uv & u \\
\end{array}%
\right)
\left(%
\begin{array}{c}
  P(1/u,v/u) \\
  Q(1/u,v/u) \\
\end{array}%
\right) \\
&=&-u^2(P(1/u,v/u)) \frac{\partial}{\partial u}+u(Q(1/u,v/u)-v
P(1/u,v/u)) \frac{\partial}{\partial v}
\end{eqnarray*}
and now we only need to analyze the corresponding foliation on a
neighborhood of $\{u=0\}$.

Let us denote $\sum_{i=0}^{d-1}P_i$ by $\widetilde{P}(1/u,v/u)$,
and $\sum_{i=0}^{d-1}Q_i$ by $\widetilde{Q}(1/u,v/u)$. Notice that
$u^d\widetilde{P}(1/u,v/u)=uh_1(u,v)$,
$u^d\widetilde{Q}(1/u,v/u)=uh_2(u,v)$ for appropriate polynomials
$h_1$ and $h_2$ of degree $d-1$. Also, $u^dP_d(1/u,v/u)=P_d(1,v)$,
naturally there is an analogous expression for $Q_d$. Multiplying
the vector field $X(u,v)$ by $u^{d-1}$ we obtain a holomorphic
vector field in the $(u,v)$-coordinate which is given by
\[
Y(u,v)=-u(P_d(1,v)+uh_1(u,v))\frac{\partial}{\partial
u}+(Q_d(1,v)-vP_d(1,v)+ug(u,v))\frac{\partial}{\partial v} \, ,
\]
where $g(u,v)=h_2(u,v)-vh_1(u,v)$.

Now, if $Q_d(1,v)-vP_d(1,v)\equiv 0$, the components of $Y(u,v)$
are both divisible for $u$. By eliminating this common factor, it
becomes clear that the line at infinity is not preserved by the
foliation. Conversely, if $Q_d(1,v)-vP_d(1,v)$ is not identically
zero, $L_\infty$ is preserved, since the component $\partial
/\partial u$ of $Y$ vanishes identically over $L_{\infty} \simeq
\{ u=0\}$.

Finally it is clear that $Q_d(1,v)-vP_d(1,v)$ vanishes identically
if and only if the top-degree homogeneous component is radial.
\end{proof}

\begin{defi}
The degree of a foliation $\fol$ on $\cpd$ given by the
compactification of the polynomial vector field
$X=P\frac{\partial}{\partial x}+Q\frac{\partial}{\partial y}$ of
degree $d$ and having only isolated zeros is equal to:
\begin{enumerate}

\item $d-1$, if there exists a polynomial $h(x,y)$ of degree $d-1$
such that $P_d\frac{\partial}{\partial
x}+Q_d\frac{\partial}{\partial y}=h(x,y)(x\frac{\partial}{\partial
x}+y\frac{\partial}{\partial y})$. In other words, $d-1$ if the
top-degree homogeneous component of $X$ is a multiple of the
radial vector field.

\item $d$, otherwise.
\end{enumerate}
\label{d2}
\end{defi}

It should be verified that this is indeed well-defined.

Here we shall give a more geometric interpretation of the degree of
a foliation as defined above. In fact, the contents of this lemma
can be used as an equivalent definition of degree.

\begin{lema}
\label{l2} Let $\fol$ be a singular holomorphic foliation on
$\cpd$ of degree $d$. Then the following holds:
\begin{enumerate}
\item There is a homogeneous polynomial vector field $Z$ of degree
$d$ on $\C^3$, with singular set of codimension at least $2$,
which induces $\fol$ by radial projection of its orbits;

\item The number of tangencies of $\fol$ with a generic projective
line is $d$.
\end{enumerate}
\end{lema}

\begin{proof}
Let us first show that the projection of
the foliation associated to a homogeneous polynomial vector field
of degree $d$ on $\C^3$ is, indeed, a foliation $\fol$ on $\cpd$
having degree $d$. Let $Z=\sum_{i=0}^2 H_i (z_0,z_1,z_2)
\frac{\partial}{\partial z_i}$, where $(z_0,z_1,z_2)$ stands for
the coordinates of $\C^3$. In the chart $(x,y)=(z_1/z_0,z_2/z_0)$,
the vector field $Z$ becomes
\[
X=P(x,y)\frac{\partial}{\partial x}+Q(x,y)\frac{\partial}{\partial
y} \;\; ,\;\; {\rm where}
\]
\begin{eqnarray*}
P(x,y)&=& H_1(1,x,y)-xH_0(1,x,y)\\
Q(x,y)&=& H_2(1,x,y)-yH_0(1,x,y)\, .
\end{eqnarray*}

If $H_0(1,x,y)$ has degree $d$ (i.e, $H_0$ is not divisible by
$z_0$), then $X$ has degree $d+1$. Furthermore, the top-degree
component of $X$ is given by
$-H_0^d(1,x,y)(x\frac{\partial}{\partial
x}+y\frac{\partial}{\partial y})$, where $H_0^d$ stands for the
homogeneous component of degree $d$ of $H_0(1,x,y)$. So, according
to Definition~(\ref{d2}), $\fol$ has degree~$d$.

If $H_0(1,x,y)$ has degree less than $d$ (i.e, $H_0$ is divisible
by $z_0$), then at least one between $H_1(1,x,y)$ and $H_2(1,x,y)$
must have degree $d$, otherwise all of the three polynomials would
be divisible by $z_0$. This would mean that the singular set of
$Z$ has codimension $1$, which contradicts our assumption.
Therefore, $X$ has degree $d$, and once again
Definition~(\ref{d2}) implies that $\fol$ has degree $d$. The
converse is analogous and establishes the first part of the
statement.

Let us now consider the tangencies between $\fol$ and a generic
line in $\cpd$. Modulo performing a projective change of
coordinate, we may suppose that the tangencies with a generic
projective line $y=\lambda x$ are all contained in the main affine
chart of $\cpd$ so that they are given by:
\[
\lambda P(x,\lambda x)=Q(x,\lambda x)\, ,
\]
that is, by the zeros of the polynomial $\lambda P(x,\lambda
x)-Q(x,\lambda x)$. Hence the number of tangencies (counted with
multiplicity) is the degree of $\lambda P(x,\lambda x)-Q(x,\lambda
x)$. However, if $d$ is the degree of the foliation, then either
$P_d(x,y)\frac{\partial}{\partial x} +
Q_d(x,y)\frac{\partial}{\partial y}$ is radial (and consequently
$X$ has degree $d+1$) or it is not (and $X$ has degree $d$). The
first case is equivalent to have $\lambda P_d(x,\lambda
x)-Q_d(x,\lambda x)=0$, which means that $\lambda P(x,\lambda
x)-Q(x,\lambda x)$ has degree $d$. The other case only happens
when the top-degree component of $\lambda P(x,\lambda
x)-Q(x,\lambda x)$ is not zero, implying that the degree of this
polynomial is $d$. The converse is again analogous.
\end{proof}

According to Theorem~(\ref{spfoliation}) and to Lemma~(\ref{l2}),
the space $\spfl$ consisting of singular holomorphic foliation of
degree~$d$ is naturally contained in the space of homogeneous
polynomial vector fields of degree~$d$ on three variables. Besides
two such vector fields having a singular set of codimension at
least~$2$ define the same foliation if and only if they differ by
a multiplicative constant. Thus a simple counting of coefficients
yields the following corollary:

\begin{coro}
\label{coro3} The space $\spfl$ is naturally identified with a
Zariski-open set of the complex projective space of dimension
$$
(d+1)(d+3) -1 \, .
$$
\mbox{ }\qed
\end{coro}

It should also be noted that the group of automorphims of $\cpd$,
${\rm PSL}\, (3,\C)$, has a natural action on $\spfl$ through
projective changes of coordinates.


\bigskip
\textbf{Example 4: Foliations on weighted projective spaces}

This is a very natural generalization of the previous example
that often appears in higher dimensional questions related
to resolution of singularities, cf. for example \cite{arnold}.
They are also similar to foliations on Hopf surfaces.

A polynomial $P$ on $n$ variables $(x_1 , \ldots ,x_n)$ is said to
be {\it quasi-homogeneous}\, with {\it weights}\, $(k_1, \ldots ,k_n)$
and degree $d$ if and only if for
every $\lambda \in \C^{\ast}$ one has
$$
P(\lambda^{k_1} z_1, \ldots , \lambda^{k_n} z_n) =\lambda^d
P  (z_1 ,\ldots ,z_n) \, .
$$
For example $P(x,y,z) = xz + y^2$ is not only quadratic (ie
homogeneous of degree $2$) but also quasi-homogeneous of degree~$4$
relative to the weights $(1,2,3)$.

Chosen a set of weights $(k_1, \ldots ,k_n)$, we have a natural action
of $\C^{\ast}$ on $\C^n \setminus \{ (0, \ldots ,0)\}$ given by
$$
\lambda . (z_1 ,\ldots ,z_n) = (\lambda^{k_1} z_1, \ldots , \lambda^{k_n} z_n) \, .
$$
Consider the quotient space of $\C^n \setminus \{ (0, \ldots ,0)\}$
where two points are identified if and only if they belong to the
same orbit of $\C^{\ast}$. The resulting space is a compact manifold
with singularities called a {\it weighted projective space}\, whose dimension
is obviously equal to~$n-1$. Whether or not it actually has singularities,
this type of manifold can be given an algebraic structure since it can be
realized as a Zariski-closed set of a complex projective space with sufficiently
high dimension. The existence of this imbedding can easily be shown by
means of Pl\"ucker coordinates.

Alternatively the quotient of this $\C^{\ast}$-action can also
be realized as the quotient of the
projective space of dimension $n-1$ by some {\it finite group of automorphism}.

Next a polynomial vector field $P_1 \partial /\partial x_1 +
\cdots + P_n \partial /\partial x_n$ is said to be {\it quasi-homogeneous}\,
with weights $(k_1, \ldots ,k_n)$ and degree $d$ if and only if for
every $\lambda \in \C^{\ast}$ one has
$$
\Lambda^{\ast}X = \lambda^{d-1} X \, ,
$$
where $\Lambda$ stands for the map $(z_1 ,\ldots ,z_n) \mapsto
(\lambda^{k_1} z_1, \ldots , \lambda^{k_n} z_n)$. An example of
quasi-homogeneous vector field with weights $(1,2,3)$ and degree~$4$
is
$$
(xz +y^2) \fracx + (2zy + 3x^5) \fracy + (x^3z -y^3 +2z^2) \fracz \, .
$$

If $X$ is quasi-homogeneous with weights $(k_1, \ldots ,k_n)$ then the
definition above implies in particular
that $X$ unequivocally defines a complex direction at each point of the
projective space with the same weights $(k_1, \ldots ,k_n)$. Being of
complex dimension~$1$, these directions can naturally be integrated
to form a singular holomorphic foliation.
Therefore quasi-homogeneous vector fields give
rise to singular holomorphic foliations on weighted projective spaces.


\bigskip
\textbf{Example 5: Riccati Foliation}
The classical Riccati equation is given by
$$
\frac{dy}{dx} = a(x) y^2 + b(x) y + c(x) \, ,
$$
where $x,y$ are complex variables. We are interested in the
case where $a,b,c$ are rational functions of $x$. In this case, if
$P$ denotes the least common multiple of the denominators of $a,b,c,$,
the preceding equation is equivalent to the vector field
$$
X = P(x) \partial /\partial x + (a^{\ast} (x) y^2 + b^{\ast} (x) y +
c^{\ast} (x)) \partial /\partial y \, ,
$$
with $a^{\ast}, b^{\ast}, c^{\ast}$ polynomials. Although it is possible
to compactify the associated foliation on $\cpd$, it is more natural to
compactify it in $\C {\mathbb P} (1) \times \C {\mathbb P} (1)$. Consider the
``vertical'' projection $\pi_1$ onto the first factor. The change o coordinate
$(x,y) \mapsto (x, 1/y)$ allows one to check that the vector field $X$ has
a holomorphic extension to the ``infinite'' of the fibers (or to the section
at infinity). However $X$ has in general poles on the ``fiber over infinity''
of affine coordinates $(1/x,y)$. In the initial affine coordinates $(x,y)$,
we see that the set $\{ P=0\}$ is consituted by {\it invariant fibers}\, and
contains all the (affine) singularities of the underlying foliation $\fol$. A similar
calculation applies to the fiber over infinity. Thus we conclude that the
foliation associated to $X$ has the following properties:
\begin{enumerate}

\item It admits a {\it nonzero} number of invariant fibers given in the
affine coordinates $(x,y)$ by $\{ P =0\}$. The fiber over infinity may
or may not be invariant by this foliation.

\item The union of these invariant fibers contains all the singularities
of $\fol$.

\item Away from the invariant fibers, $\fol$ is transverse to the fibers
of $\pi_1$.

\end{enumerate}

Since the fibers of $\pi_1$ are compact, a simple remark due to Ehresmann
guarantees that the restriction of $\pi_1$ to the regular leaves of $\fol$
(different from the invariant fibers) defines a covering map onto
$\C {\mathbb P} (1) \setminus \{ p_1, \ldots ,p_k\}$ where $p_1 , \ldots
,p_k \in \C {\mathbb P} (1)$ are precisely the projection of the fibers
invariant under $\fol$.

Fix a fiber $\pi_1^{-1} (p)$, with $p \not\in \{ p_1, \ldots ,p_k\}$.
Thanks to the remark above, paths contained in $\C {\mathbb P} (1) \setminus \{ p_1, \ldots ,p_k\}$
can be lifted in the leaves of $\fol$. Therefore we can consider the {\it global
holonomy}\, of $\fol$ (w.r.t. $\pi_1$). This is given by a homomorphism
$$
\rho \Pi_1 (\C {\mathbb P} (1) \setminus \{ p_1, \ldots ,p_k\}) \longrightarrow
{\rm Aut}\, (\pi_1^{-1} (p)) \simeq {\rm Aut}\, (\C {\mathbb P} (1)) \, .
$$
Since ${\rm Aut}\, (\C {\mathbb P} (1))$ is isomorphic to ${\rm PSL}\, (2,\C)$,
the homomorphism $\rho$ can also be viewed as taking values in this latter
group. In particular the theory of the Riccati equations is naturally connected to
the theory of finitely generated subgroups of ${\rm PSL}\, (2,\C)$. In particular,
in the cases where the group is in addition discrete, to the theory of Kleinian
groups and consequently, at least to some extent, to hyperbolic geometry in dimension~$3$.
We also note that, conversely, a classical result due to Birkhoff states that every
finitely generated subgroup of ${\rm PSL}\, (2,\C)$ can be relized as the monodromy group
of some Riccati equation. As a matter of fact, this correspondence is not unique: there
are several Riccati equations possessing the same monodromy group. This raises the question
of trying to find the ``simplest'' Riccati equation realizing a given monodromy group.
As far as we know, this question has not been addressed to in the literature.


\bigskip
\textbf{Example 6: Linear Equations}

Linear equations are among the most classical topics in complex analysis.
They generalize Riccati equations as well as many other equations such as
Gauss hypergeometric equation, Fuchsian equations and so on. Furthermore
they play a significative role in the theory of vector bundles over
Riemann surfaces.

In a classical setting we consider a meromorphic function from
$\C$ with values on ${\rm GL}\, (n ,\C)$, for some $n \geq 2$. These
are simply a family of invertible matrices of rank~$n$ whose coefficients
are meromorphic functions defined from $\C$ to $\C$. The notion of invertible
matrices of course refer to those points away from the poles of these
coefficients.




\bigskip
\textbf{Example 7: Jouanolou Foliation}

This is a very special example of foliation on the complex projective plane.
Fixed $n \in \N^{\ast}$, the {\it Jouanolou foliation $J_n$ of degree $n$}\, is the
foliation induced on $\cpd$ by the homogeneous $1$-form
$$
\Omega = (yx^n -z^{n+1}) \, dx \; + \; (zy^n - x^{n+1}) \, dy \; + \;
(xz^n - y^{n+1}) \, dz \; .
$$
The reader will notice that the kernel of $\Omega$ always contains the radial
direction so that it naturally induces a line field, and hence a foliation, on
$\cpd$. In terms of homogeneous vector field, the foliation $J_n$ is given by
$$
y^n \fracx + z^n \fracy + x^n \fracz \, .
$$
In standard affine coordinates, $\{ z=1 \}$, the foliation is induced by the
vector field
$$
(y^n -x^{n+1}) \fracx + (1-yx^n) \fracy
$$
The main result of Jouanolou states that $J_n$ leaves {\it no algebraic
curve}\, invariant. Foliations leaving algebraic curves invariant will be discussed
in detail in Chapter~3 and the foliations $J_n$ will play a significative role
in the discussion. For the time being, let us only mention some elementary
properties of $J_n$. The proofs of the assertions below amount to direct inspection
in the spirit of the preceding example, they are therefore left to the reader.

We begin$\{x=y=z\}$ by noticing that $J_n$ is left invariant by a nontrivial group of
automorphisms of $\cpd$. To identify this group, let $\zeta$ be a primitive
$N^{\rm th}$-root of the unit with $N =n^2 +n+1$. The cyclic group $H$ generated
by the automorphism $T(x,y,z) = (\zeta z , \zeta^{-n}y, z)$ preserves $J_n$ as
well as it does the group $K$ generated by the cyclic permutation $S(x,y,z)=
(y,z,x)$. The order of $H$ is precisely $n^2+n+1$ while the order of $K$ is
obviously equal to~$3$. It can be checked that $T$ and $S$ generate the maximal
group of automorphisms of $\cpd$ preserving $J_n$.

It is a remarkable fact that all the foliations $J_n$, $n \in \N^{\ast}$, are
obtained by means of a family of quadratic foliations on $\cpd$. This can be
done as follows. Consider that map $\Upsilon_n : \, \cpd \rightarrow \cpd$
given in homogeneous coordinates by $\Upsilon_n (x:y:z) = (y^{n+1} z: z^{n+1} x :
x^{n+1} y)$. This induces a ramified covering of $\cpd$ with degree $N=n^2+n+1$.
It can be checked that $J_n$ is the pull-back $\Upsilon_n {\mathcal Q}_n$ where
${\mathcal Q}_n$ is the quadratic foliation on $\cpd$ induced by the homogeneous
vector field
$$
X_n = x(x-ny) \fracx + y(y-nz) \fracy + z(z-nx) \fracz \; .
$$
The reader will immediately check that ${\mathcal Q}_n$ has exactly~$7$
singularities. These correspond to the radial lines invariant by $X_n$, namely
the coordinates axes, $\{ x=y=z\}$, $t. (0,n+1,1)$, $t.(n+1,1,0)$ and $t.(1,0,n+1)$
where $t\in \C$. ${\mathcal Q}_n$ also leaves $3$ projective lines $l_1,l_2,l_3$
invariant, namely those induced by projection of the coordinates $2$-planes. It is
easy to see that the lines $l_1,l_2,l_3$ form a ``triangle'' containing the singularities
of ${\mathcal Q}_n$ except for the singularity $\{ x=y=z\}$. This ``triangle'' is
the ``maximal'' algebraic curve invariant by ${\mathcal Q}_n$.

Naturally $X_n$ is still invariant by the permutation $S$, and hence by the group $K$.
It is therefore possible to further ``simplify'' ${\mathcal Q}_n$ by exploting these
symmetries. This would lead us to a foliation of higher degree leaving a single irreducible
algebraic curve invariant. Further information on Jouanolou foliations can be found
in Chapter~$3$ along with some specific references.

\bigskip
\textbf{Example 9: Ramanujan Differential Equation}

This is a system of differential equations introduced by
Ramanujan that plays a significative role in the part of Number
Theory dealing with transcendent numbers. The system closed related to
the following vector field defined on $\C^3$.
$$
X = \frac{1}{12} (x^2 - y) \fracx + \frac{1}{3} (xy-z) \fracy
+ \frac{1}{2} (xz-y^2) \fracz \, .
$$
This vector field is, indeed, a particular example of class of
differential equations known as Halphen vector fields. These were
detailed studied by A. Guillot in \cite{Guillot}. In later
chapters, we shall have occasion to talk more about Halphen vector
fields. For the time being, we would like to indicate very briefly the
connection between the above vector field and the theory of transcendent numbers,
for further information we refer the reader to \cite{LNM}.
This begins with the so-called Ramanujan $P,Q,R$ functions that can be
defined by explicit formulas as follows.
\begin{eqnarray*}
P(t) & = & 1-24\sum_{n=1}^{\infty} \sigma_1 (n) t^n \, , \\
Q(t) & = & 1 + 240 \sigma_{n=1}^{\infty} \sigma_3 (n) t^n \, , \\
R(t) & = & 1-504 \sigma \sum_{n=1}^{\infty} \sigma_5 (n) t^n \, ,
\end{eqnarray*}
where $t \in \C$ and $\sigma_k (n) = \sum_{d \vert n} d^k$, ie
$\sigma_k (n)$ is the sum of the $k^{\rm th}$-powers of the
positive integers dividing $n$. Since these formulas may
look especially weird, let us mention that they can be obtained
from rather natural procedures with automorphic functions. For our
present purpose however the above definition will suffice. For example,
it makes easy to see that these
functions are defined for $\Vert t \Vert <1$. It is less easy to
see that the unit circle constitutes an ``essential boundary for
them'' in the sense that the have no holomorphic extension to a
neighborhood  of any point in the circle. The interest of these
functions for number theorists is partially due to the fact that
they assume ``specially remarkable values'' for specific choices
of $t$.

According to Ramanujan, the functions $P,Q,R$ verify the following
{\it nonautonomous}\, system of differential equations:
\begin{eqnarray*}
tP'(t) & = & \frac{1}{12} (P^2 -Q) \, , \\
tQ'(t) & = & \frac{1}{3} (PQ -R) \, , \\
tR'(t) & = & \frac{1}{2} (PR-Q^2) \, .
\end{eqnarray*}
Although this system is not autonomous, it ensures that the
``velocity vector'' of the parametrized curve $t \mapsto
(P(t), Q(t) ,R(t))$ is parallel to $X$ at the point
$(P(t), Q(t) ,R(t))$. In other words, this curve parametrizes
a leaf of the foliation $\fol$ associated to $X$. Actually, the
image of this curve also contains the point $(1,1,1)$ that technically
speaking does not belong to the leaf in question since
$(1,1,1)$ is a singular point of $X$.

If we choose $t_0 \ne 0$ (for example $t_0 = $)
and the corresponding point $(P(t_0),
Q(t_0), R(t_0))$, we can compare the functions $P,Q,R$ with the
actual solutions $(x(t),y(t),z(t))$ of the vector field
$X$ satisfying the initial value condition $(x(t_0), y(t_0),
z(t_0)) = (P(t_0), Q(t_0), R(t_0))$. This would lead us
to the simple ($1$-dimensional) equations $x'=tP'$, $y' = tQ'$
and $z' = tR'$. Therefore the solutions are
\begin{eqnarray*}
 & & x(t) = tP(t) + (1-t_0) P(t_0) - \int_{t_0}^t P(s)ds \, .\\
 & & y(t) = tQ(t) + (1-t_0) Q(t_0) - \int_{t_0}^t Q(s)ds \, .\\
 & & z(t) = tR(t) + (1-t_0) R(t_0) - \int_{t_0}^t R(s)ds \, .
\end{eqnarray*}
Naturally the effect of changing the initial data $t_0 \neq 0$ and
$(P(t_0), Q(t_0), R(t_0))$ essentially amounts to performing
a translation for the functions $(x,y,z)$ and thus is of little
importance.

Summarizing the preceding discussion, we conclude that information
on the properties of $P,Q,R$
can be extract from the study of the solutions of the vector field $X$.


\bigskip
\textbf{Example 10: The Lorenz Attractor}

In the course of the past few years,
the dynamical system associated to Lorenz attractor has become a
paradigmatic example of ``chaotic dynamics''. The vector field
defining the dynamical system was introduced by Lorenz in \cite{lorenz}
in connection with a evolution problem for atmospheric conditions.
In general, the domain belongs to fluid dynamics and the phenomenon
is governed by the KdV equation. Since these equations are notoriously
hard to be analysed, Lorenz thought of his vector field as a simplified
model to describe this evolution. The vector field in question is
simply given by
$$
X = 10(y-x) \fracx + (28x -y - xz) \fracy + (xy -8z/3) \fracz \, .
$$
It is therefore a quadratic vector field (obviously $X$ is not homogeneous
so that ``quadratic'' here is a abuse of language) defined on $\R^3$. Despite
its innocent
appearance, $S$ exhibits a remarkably complicated dynamics. A traditional
picture of the plotting of orbits of $X$ looks like figure~\ref{l-attractor}.

\begin{figure}[ht!]
\centering
\includegraphics[scale=0.5]{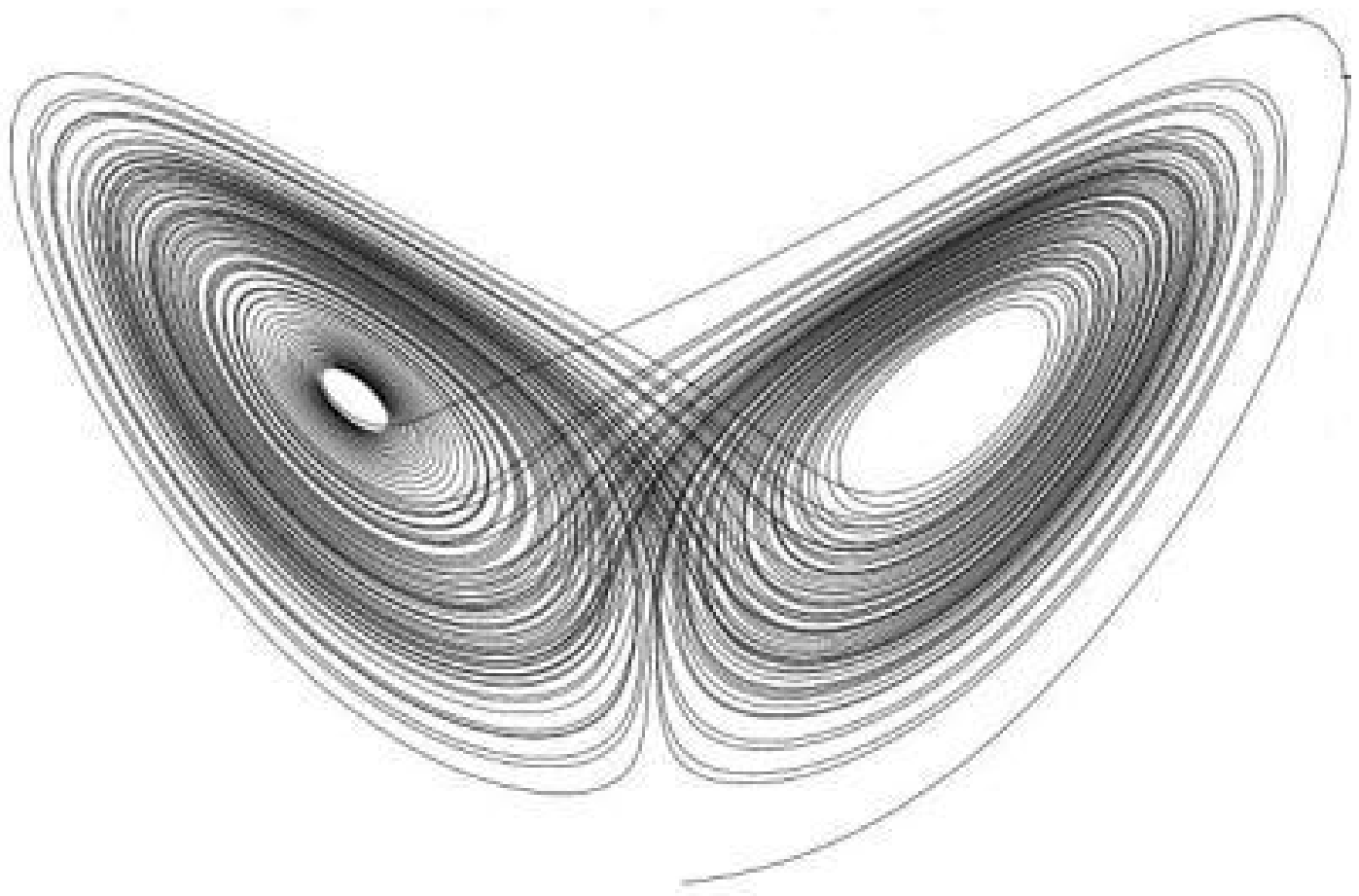}
\caption{}\label{l-attractor}
\end{figure}

Figure~\ref{l-attractor} seems to indicate the presence of an ``attractor''
for $X$. The definition of the term ``attractor'' may vary depending
on the context and we shall give none in this discussion. In any case the guiding
picture to bear in mind is that of a ``compact
invariant set'' attracting
nearby orbits. In particular, orbits of points that actually belong to the
attractor remain contained in it forever (both in ``past'' and
in ``future''). Since this attractor is not hyperbolic, it was called
``strange''. Here the reader may consider that hyperbolicity is the best known mechanism to produce this type of behavior. Besides, it has been established
that the attractor, if it existed, was {\it robust}, ie. it could not be
destroyed by performing
a small perturbation on $X$. Curiously enough, the very existence of Lorenz
attractor was settled only recently by W. Tucker through rigorous numerics
\cite{tucker}. An interesting survey on the Lorenz attractor
containing in particular precise definitions and notions is \cite{viana}.
For a beautiful geometric discussion of the properties
of this attractor, including some wonderful graphics, the reader is referred
to \cite{ghys-leys}.

Although the interest on Lorenz vector field has primarily to do with its
real dynamics, this vector field can equally well be thought of as
being a complex polynomial (and thus holomorphic) vector field defined
on $\C^3$. It is natural to consider this complexification not only for
it may provide new insight in the real dynamics, but also because it
is likely to be interesting in itself. Besides, in view of Example~3,
it may be useful to consider the holomorphic extension to $\C \Pp (3)$
of the foliation $\fol$ associated to $X$.

Elementary calculations yield some basic facts, such as singularities
and existence of invariant curves, about
$\fol$ viewed as a singular holomorphic foliation on $\C \Pp (3)$.
In the affine $\C^3$, the foliation has exactly three singularities
corresponding to the points $(0,0,0)$, $(6\sqrt{2}, 6\sqrt{2}, 27)$
and $(-6\sqrt{2}, -6\sqrt{2}, 27)$ which in addition have non-degenerate
linear part. Also the axis $\{ x=y=0\}$ is
obviously invariant by $\fol$. The plane at infinity $\Delta =
\C \Pp (3) \setminus \C^3$ is also invariant by $\fol$ and contains
$3$ other singularities of it. In coordinates $x=1/u$, $y=v/u$
and $z=w/u$, the plane at infinity is given by $\{ u=0\}$ the
restriction of $\fol$ to this plane coincides with the one induced
by the vector field $-28w \partial /\partial v + v \partial /\partial w$.
In particular, it has a singularity whose eigenvalues vanish at $u=
v=w=0$. The complement of the coordinates $v,w$ in $\Delta$ is a projective
line that is, in addition, invariant by $\fol$. This line contains
the last two singularities of $\fol$.

The picture of the existence of the Lorenz attractor and the
consequently existence of real orbits lying entirely in some compact set
of $\R^3$ raises the question of knowing if a similar phenomenon will
happen with the complex leaves of $\fol$. This is however not the case.
Still the question serves as motivation for us to state and proof
Proposition~\ref{unboundedorbits} below. This proposition will play a
significative role
in Chapter~4 as well as it does in many aspects of the study of holomorphic
foliations on complex projective spaces.

\begin{prop}
\label{unboundedorbits}
Let $\fol$ be a holomorphic foliation of $\C \Pp (n)$ and consider
the foliation $\fol^0$ obtained as the restriction of $\fol$ to
an affine $\C^n \subset \C \Pp (n)$. Then no regular leaf of
$\fol^0$ can entirely be contained in a compact set $K \subset \C^n$.
\end{prop}

It is to be noted that this proposition cannot be derived as an immediate
consequence of the maximum principle since the orbits of $\fol^0$ need not
be compact. Indeed, in the compact case, we can consider the projection
of the compact leaf on a chosen coordinate and argument that the image of
this projection must be a single point by virtue of the maximum principle.
In fact, there must be a point in the leaf corresponding to a projection of
``maximal modulus''. It then follows that the projection in question is constant.
In the case of open leaves, it is not obvious that the maximum of a
projection as above is attained so that we cannot conclude that the
projection is constant.

\begin{proof}
Suppose for a contradiction that $L$ is a regular
leaf of $\fol^0$ that is wholly contained in a compact set $K \subset
\C^3$. Consider a polynomial vector field $X$ tangent to $\fol$. The
leaf $L$ of $\fol$ can then be considered as an orbit of $X$. Since $X$
is uniformly bounded in $K$, there exists $\epsilon >0$ such that
for every $p \in K$ the solution $\varphi$ of $X$ satisfying the initial
condition $\varphi (t) = p$ is defined on a disc of radius
$r$ about $t \in \C$. Now we proceed by choosing a point $p_1 \in
L \cap K$ and considering the solution $\varphi_1$ satisfying
$\varphi_1 (t_1) =p_1$, for some $t_1 \in \C$.
Naturally $\varphi_1$ is defined on the disc
$B_{\epsilon} (t_1)$
of radius $\epsilon$ about $t_1 \in \C$. Next consider a point
$t_2$ in the boundary of $B_{\epsilon} (t_1)$. Modulo reducing
$\epsilon$ from the beginning, we can without loss of generality
suppose that $\varphi_1 (t_2) =p_2$. However $p_2$ must belong to
$K$ so that the solution $\varphi_2$ of $X$ satisfying $\varphi_2
(t_2) = p_2$ is defined on the disc $B_{\epsilon} (t_2)$ of radius
$\epsilon$ about $t_2$. Since the union $B_{\epsilon} (t_1) \cup
B_{\epsilon} (t_2)$ is simply connected, it follows from the monodromy
theorem that the solutions
$\varphi_1, \, \varphi_2$ can be patched together into a solution of
$X$ defined on $B_{\epsilon} (t_1) \cup B_{\epsilon} (t_2)$. Actually,
since $t_2$ is arbitrary in the boundary of $B_{\epsilon} (t_1)$, it
follows that this boundary can be covered by finitely many discs
as indicated above. By repeatedly using the monodromy theorem, we
conclude that the initial solution $\varphi_1$,
$\varphi_1 (t_1) =p_1$, can be extended to a disc of radius
$\epsilon + \delta$ about $t_1$ with $\delta >0$ uniformly chosen.
An obvious induction shows that this solution is defined on the
disc of radius $\epsilon + k\delta$ about $t_1$ for every $k \in \N$.
In other words, $\varphi_1$ is defined on all of $\C$.

The desired contradiction follows now from Liouville Theorem: since
$\varphi_1$ is a bounded holomorphic map from $\C$ to $\C^n$ it
must be constant. The proposition is proved.
\end{proof}


\section{Basic tools}


\subsection{The Blow-up of manifolds and of foliations}

At first sight, the ``blow-up'' may be thought as a device that
simply creates new manifolds out of previous ones. However, it will be
seen that it is particularly useful to understand the behavior of
foliations or vector fields at singular points.

\begin{defi}
The blow-up of $\C^2$ at $(0,0)$ is a complex manifold $\wdc^2$
obtained by identifying two copies of $\C^2$ in the following way:
\[
(x,t)\simeq (s,y) \Leftrightarrow \; s=\frac{1}{t};\;\;y=tx\;\; (t\neq0\; ,s\neq 0)\; ,
\]
where $(x,t)$ and $(s,y)$ are the coordinates of the two above
mentioned copies.
\end{defi}

By definition, the exceptional divisor of $\wdc^2$ is $E \subset
\wdc^2$ given by $\{x=0\}$ (resp. $\{y=0\}$) in the coordinates
$(x,t)$ (resp. $(s,y)$). Therefore $E$ is well-defined and it is
isomorphic to ${\mathbb C}{\mathbb P}\, (1)$.

The {\it blow-up mapping}\, $\pi: \wdc^2 \rightarrow \C^2$ is given
by $\pi(x,t)=(x,tx)$ and $\pi(s,y)=(sy,y)$. Moreover, it verifies the following:
\begin{itemize}
\item $\pi^{-1}(0,0)=E$;

\item $\pi:\wdc^2 \setminus E \rightarrow \C^2 \setminus
\{(0,0)\}$ is a holomorphic diffeomorphism;

\item $\pi$ is proper (i.e. the preimage of a compact set is also
compact).
\end{itemize}

Now we shall define the blow-up of a complex $2$-dimensional
manifold $M$ at a point $p\in M$. Consider a local coordinate
chart $\psi:U\rightarrow W\subset\mathbb{C}^2$ defined on a neighborhood
$U$ of $p$ and such that $\psi(p)=(0,0)$.

Let $\widetilde{W}=\pi^{-1}(W)$, where $\pi$ is the blow-up
mapping. Let $M'$ be the disjoint union of $M \setminus \{p\}$
with $\widetilde{W}$, and consider the following equivalence
relation:
\[
q_0 \simeq q_1 \Longleftrightarrow q_0\in U \setminus \{p\} \; , \;\;\;
q_1\in \widetilde{W} \setminus E \;\; \rm and \;\;
q_1=\pi^{-1}(\psi(q_0))\; .
\]
The blow-up $\widetilde{M}$ of $M$ at $p$ is the quotient
$M'/\simeq$. Notice that $\widetilde{M}$ is indeed a smooth
complex manifold for $\widetilde{W}$ is a manifold and
$\pi^{-1}\circ\psi:U\setminus \{p\}\rightarrow
\widetilde{W}\setminus E$ is a holomorphic diffeomorphism.

Similarly, there is a blow-up mapping from $\widetilde{M}$ to $M$
(which will also be denoted by $\pi$) that is proper and takes $E$
to $p$, i.e. $\pi (E)=p$. Moreover, $\pi:\widetilde{M}\setminus
E\rightarrow M\setminus \{p\}$ is a holomorphic diffeomorphism.

The blow-up of a foliation or vector field can also be defined in
a natural way. Let $X=F \frac{\partial}{\partial x}+G
\frac{\partial}{\partial y}$ be a vector field on a neighborhood
$U$ of the origin in $\mathbb{C}^2$,
where $F$ and $G$ are holomorphic functions.
Suppose that $(0,0)$ is a singularity of $X$ of
order $k$ ($k$ is the minimum between the orders of $F$ and $G$ at
$(0,0)$).

Let $F=\sum_{n=k}^\infty F_n$, and $G=\sum_{n=k}^\infty G_n$,
where $F_n$ and $G_n$ are the homogeneous components of degree $n$
of the Taylor series of $F$ and $G$, respectively.

By using the blow-up mapping $\pi:\widetilde{\mathbb{C}}^2 \setminus
E \rightarrow\mathbb{C}^2 \setminus \{(0,0)\}$ in the $(x,t)$
coordinates, namely $\pi(x,t)=(x, tx)$ ($x\neq 0$),
$\pi^\ast X$ has a natural meaning. In fact, one has
\[
\ \pi^\ast X=\left(%
\begin{array}{cc}
  1 & 0 \\
  -t/x & 1/x \\
\end{array}%
\right)
\left(%
\begin{array}{c}
  x^k F_k(1,t)+ x^{k+1}(F_{k+1}(1,t)+xF_{k+2}(1,t)+\ldots) \\
  x^k G_k(1,t)+ x^{k+1}(G_{k+1}(1,t)+xG_{k+2}(1,t)+\ldots) \\
\end{array}%
\right) \; .\]
Setting $f(x,t) = (F_{k+1}(1,t)+xF_{k+2}(1,t)+\ldots)$ and
$g(x,t) = (G_{k+1}(1,t)+xG_{k+2}(1,t)+\ldots)$, the above
equation becomes
\begin{equation}
\pi^\ast X=x^k [F_k(1,t)+ xf(x,t)] \; \frac{\partial}{\partial
x}+x^{k-1}[-tF_k(1,t)-xtf(x,t)+G_k(1,t)+xg(x,t)] \;
\frac{\partial}{\partial t} \; . \label{eq5}
\end{equation}
It is to be noted that
this vector field admits a holomorphic extension
$\tXX$ to $E$ ($\{x=0\}$). Thus, $\tXX$ is defined to be the
{\it blow-up}\, of $X$ at the singular point $(0,0)$. Moreover, if
$k\geq 2$ then $\tXX$ is singular at every point of $E$.
Similarly, the {\it blow-up} of the foliation $\fol$ associated
to $X$ is the foliation $\widetilde{\fol}$ associated to $\tXX$.

The behavior of $\widetilde{\fol}$ (and of $\tXX$) on a
neighborhood of $E$ is significantly different, depending on
whether or not $G_k(1,t)-tF_k(1,t)$ is identically zero. Let us
analyze each case separately.

\begin{itemize}
\item If $G_k(1,t)-tF_k(1,t)$ is {\it not}\, identically
{\it zero}.

Dividing Equation~(\ref{eq5}) by $x^{k-1}$, the foliation remains
unchanged and $\widetilde{\fol}|_{\wdc^2 \setminus E}$ is given by
\[
x [F_k(1,t)+ xf(x,t)] \; \frac{\partial}{\partial
x}+[-tF_k(1,t)-xtf(x,t)+G_k(1,t)+xg(x,t)] \;
\frac{\partial}{\partial t} \; .
\]
Thus, the singularities of $\widetilde{\fol}$ on $E$ are determined by
the equation $G_k(1,t)-tF_k(1,t)=0$. Furthermore, we see that
$E$ ($\{x=0\}$) is invariant under the foliation in question.

\item If $G_k(1,t)-tF_k(1,t) \equiv 0$ (equivalently, if
$F_k(x,y)\frac{\partial}{\partial
x}+G_k(x,y)\frac{\partial}{\partial y}$ is a multiple of the
radial vector field $x\frac{\partial}{\partial
x}+y\frac{\partial}{\partial y}$)

In this case, we may divide Equation~(\ref{eq5}) by $x^k$ so
as to obtain:
\[
[F_k(1,t)+ xf(x,t)] \; \frac{\partial}{\partial
x}+[-tf(x,t)+g(x,t)] \; \frac{\partial}{\partial t} \; .
\]
Hence
\[
\widetilde{\fol}|_E=F_k(1,t) \; \frac{\partial}{\partial
x} + [-tF_{k+1}(1,t)+G_{k+1}(1,t)] \; \frac{\partial}{\partial t}
\; .
\]

Notice that $F_k(1,t)$ is {\it not}\, identically {\it zero}, for
if it were, then $G_k(1,t)\equiv 0$ and $(0,0)$ would not be a
singularity of order $k$ as assumed from the beginning. Hence, the exceptional
divisor $E$ is not preserved by $\widetilde{\fol}$ and
$\widetilde{\fol}$ may have no singularity contained in $E$.
\end{itemize}

The leaves of $\widetilde{\fol}$ are transverse to $E$ and are
projected by $\pi$ onto curves passing through $(0,0)$ which
are obviously invariant by $\fol$. Due to the fact that $\pi$ is proper, the
projection is an analytic set (the common zeros of a finite number
of holomorphic functions). A local analytic curve invariant by a
foliation and containing the singularity of $\fol$ is called a
{\it separatrix} of $\fol$. A singularity possessing infinitely many
separatrizes is called {\it dicritical}. These singularities will
be characterized in Chapter~2 in connection with Seidenberg's theorem.


\subsection{Intersection numbers:}

Let $M$ be a real $4$-dimensional compact oriented
manifold. Consider two submanifolds $S_1, S_2$ of
real dimension~$2$ with the induced orientation.

Recall that the (set-theoretic) intersection of $S_1,
S_2$ consists of a finite number of points provided
that $S_1$ is transverse to $S_2$ (we also say that $S_1,
S_2$ are in general position). If $p \in M$ is a
point of $S_1 \cap S_2$, which is necessarily isolated,
then the corresponding tangent space of $M$ splits into
a direct sum
$$
T_p M = T_p S_1 \oplus T_p S_2 \, ,
$$
where $T_p S_1$ (resp. $T_p S_2$) stands for the tangent space
to $S_1$ (resp. $S_2$) at $p$.
In particular the orientations of $T_p S_1$ and $T_p S_2$
taken together yield an orientation for $T_p M$ which
may or may not coincide with the original orientation
of $T_p M$. We let $\nu (p) =1$ if these two orientations coincide
and we let $\nu (p) = -1$ otherwise. With these notations
we define the {\it intersection number}\, $\sharp (S_1 \cap
S_2)$ as
$$
\sharp (S_1 \cap S_2 ) = \sum_{p_i \in S_1 \cap S_2}
\nu (p_i) \,
$$
provided that $S_1 ,S_2$ are in general position. In the general
case, we perturb $S_1$ into $S_1'$ so that $S_1' ,S_2$ are
in general position and then set $\sharp (S_1 \cap S_2 )
= \sharp (S_1' \cap S_2 )$. The existence of the perturbation
$S_1'$ and the fact that $\sharp (S_1' \cap S_2 )$ does not
depend on the perturbation chosen are standard basic facts
of Differential Topology (see for instance \cite{top}).
When $S_1= S_2 = S$, the number $\sharp (S_1 \cap S_2 )$
is called the {\it self-intersection}\, of $S$.

\begin{obs}\label{seila}
{\rm Assume for a moment that $M, S_1, S_2$ are all
complex manifolds (and that $S_1, S_2$ are submanifolds). If $S_1,
S_2$ are transverse and $p \in S_1 \cap S_2$, then the fact that
complex manifolds have a preferred orientation implies that the
orientations of $M , S_1, S_2$ at $p$ are all compatible in the
sense that $\nu (p) =1$. Thus, if $S_1, \, S_2$ are complex
submanifolds in general position, one automatically has $\sharp (S_1
\cap S_2 ) \geq 0$.

Nonetheless if $S_1, S_2$ are not in general position, then it is
possible to have $\sharp (S_1 \cap S_2 ) < 0$. In fact, to obtain a
perturbation $S_1'$ of $S_1$ in general position with $S_2$, it may
be necessary to work in the differential category, i.e. the perturbed
submanifold $S_1'$ need no longer be a complex submanifold.}
\end{obs}

A standard topological argument ensures
that $\sharp (S_1 \cap S_2 )$ depends only
on the {\it homology class}\, of $S_1, S_2$. In the present case,
this gives rise to a symmetric {\it pairing}
$$
\langle \; \, , \, \; \rangle \, : \; H_2 (M ,\Z) \times
H_2 (M, \Z) \longrightarrow \Z
$$
which can be extended by {\it linearity}\, to an bilinear form
on $H_2 (M, \R)$ called the {\it intersection form}\, of $M$.
In particular the {\it signature}\, of the intersection form is
an important invariant of a differentiable $4$-manifold.

Let us close this short discussion with two simple facts whose
verification is left to the reader. They are going to be freely used
in the course of this text.

\vspace{0.2cm}

\noindent {\sc Fact}:  The self-intersection of
the exceptional divisor $\pi^{-1} (0)$ of $\widetilde{\C}^2$
is equal to~$-1$.

\vspace{0.2cm}

\noindent {\sc Fact}: Suppose that $S$ is a compact
Riemann surface $S \subset M$ realized as a submanifold on a complex
surface $M$. Suppose that $\widetilde{M}$ is the blow-up of $M$ at
a point $p \in S \subset M$. Then, if $\widetilde{S}$ denotes the
proper transform of $S$, one has
$$
\sharp (\widetilde{S} \cap \widetilde{S}) = \sharp (S \cap S) - 1 \, .
$$


\subsection{Complex and holomorphic vector bundles:}

Assume that $M$ is a compact differentiable manifold. A
{\it $C^{\infty}$-complex vector bundle on $M$}\, consists
of a family of {\it complex vector spaces}\, $\{ E_x \}_{x \in M}$
parametrized by $M$ along with a $C^{\infty}$-manifold structure
on $E = \bigcup_{x \in M} E_x$ such that:
\begin{enumerate}

\item The projection map $\pi : E \longrightarrow M$ taking $E_x$
to $x \in M$ is a $C^{\infty}$-map.

\item For every $x_0 \in M$, there exists an open
neighborhood $U \subset M$ of $x_0$ and a diffeomorphism
$$
\varphi_U : \pi^{-1} (U) \longrightarrow U \times C^k
$$
commuting with $\pi_1$, namely $\pi_{\mid E_x} = \pi_1 \circ
\varphi_{U \mid E_x}$ or $\pi_1 \circ \varphi_U (E_x) ={\rm const}$.

\item The diffeomorphism $\varphi_U$ considered above takes
$E_x$ isomorphically onto $\{ x \} \times \C^k$.

\end{enumerate}
The integer $k$ is said to be the {\it rank}\, of the vector
bundle and the diffeomorphisms $\varphi_U$ considered above are called its
{\it local trivializations}. Given two
trivializations $\varphi_U , \varphi_V$ of the same vector
bundle, the map
$$
g_{UV} \, : \; U \cap V \longrightarrow {\rm GL} \, (k, \C)
$$
given by $g_{UV} (x) = (\varphi_U \circ \varphi_V^{-1})_{\mid \{x
\} \times \C^k}$ is differentiable. In addition these maps
satisfy the identities
\begin{eqnarray}
 & & g_{UV} (x) \, . \, g_{VU} (x) = {\rm Id}  \label{ident1}\\
 & & g_{UV} (x) \, . \, g_{VW} (x) \, . \, g_{WU} (x)
= {\rm Id} \, \label{ident2} ,
\end{eqnarray}
where ${\rm Id}$ stands for the Identity map and the dot ``$.$''
for the multiplication of matrices in ${\rm GL} \, (k, \C)$.
The functions $g_{UV}$ are called the {\it transition functions}\,
of the complex vector bundle.

Conversely, to {\it specify}\, a structure of complex vector bundle
on $M$, all we need is an open covering ${\mathcal U} = \{ U_{\alpha}
 \}$ of $M$ together with $C^{\infty}$-maps $g_{\alpha \beta}
: U_{\alpha} \cap U_{\beta} \rightarrow {\rm GL} \, (k, \C)$
satisfying the identities~(\ref{ident1}) and~(\ref{ident2}).
The description of complex vector bundles by means of transition
functions is well-suited to define operations with them.
Precisely let $E \rightarrow M$ and $F \rightarrow M$ be complex
vector bundles having transition functions
$\{ g_{\alpha \beta} \}$ and $\{ h_{\alpha \beta} \}$ with respect to
the same covering ${\mathcal U} = \{ U_{\alpha} \}$ of $M$.
Denote by~$k$ (resp.~$l$) the rank of $E$ (resp. $F$). We give
below some standard examples of new vector bundles constructed
from the previous ones.

\noindent 1. $E \oplus F$ (the direct sum). This vector bundle
is determined by the transition functions $j_{\alpha \beta}$
given by
$$
j_{\alpha \beta} (x) = \left( \begin{array}{cc}
                               g_{\alpha \beta} (x) & 0 \\
                                0 & h_{\alpha \beta} (x)
                               \end{array} \right)
\in {\rm GL} \, (\C^k \oplus \C^l) \subset {\rm GL} \, (\C^{k+l}) \, .
$$

\noindent 2. $E \otimes F$ (the tensor product). The transition
functions are $j_{\alpha \beta} (x) = g_{\alpha \beta} (x)
\otimes h_{\alpha \beta} (x) \in {\rm GL} \, (\C^k \otimes \C^l)$.

\noindent 3. $\Lambda^r E$ (the exterior power). One sets
$j_{\alpha \beta} (x) = \Lambda^r g_{\alpha \beta} (x) \in
{\rm GL} \, ( \Lambda^r \C^k)$.

\noindent In particular $\Lambda^k E$ is a {\it line bundle}\, given by
$j_{\alpha \beta} (x) = {\rm Det} \, g_{\alpha \beta} (x) \in
{\rm GL} \, (\C ) \simeq \C^{\ast}$. This line bundle is referred
to as the {\it determinant bundle of $E$}\,.

\begin{defi}
A subbundle $F \subset E$ of a bundle $E$ is a collection
$\{ F_x \subset E_x \}_{x \in M}$ of subspaces of the fibers
$E_x$ such that $F = \bigcup_{x \in M} F_x$ is a submanifold
of $E$.
\end{defi}

The condition that $F\subset E$ is a submanifold is equivalent
to saying that, for every $x \in M$, there exists a neighborhood
$U \subset M$ of $x$ and a trivialization $\varphi_U : E_U
\rightarrow U \times \C^k$ such that
$$
\varphi_{U \mid F_U} : F_u \rightarrow U \times \C^l \subset
U \times \C^k \, .
$$
Relative to these trivializations, the transition functions
$g_{UV}$ of $E$ have the form
$$
g_{UV} (x) = \left( \begin{array}{cc}
                    h_{UV} (x) & i_{UV} (x) \\
                      0 & j_{UV} (x)
                     \end{array} \right) \, .
$$
The bundle $F$ will have transition functions $h_{UV}$. Notice
that the maps $j_{UV}$ verify the identities~(\ref{ident1}),
(\ref{ident2}) so that they define themselves a vector bundle
on $M$. This vector bundle, denoted by $E/F$, is called the
quotient bundle of $E, F$. As vector spaces we have $(E/F)_x
= E_x /F_x$, nonetheless there is no natural notion of orthogonality.

It is natural to work out the condition on two sets of
transition functions defined w.r.t. the same covering of $M$ so that
they define the same vector bundle.

\begin{defi}
\begin{itemize}
\item[1)] A {\it map between vector bundles} $E, \, F$ is a $C^{\infty}$ map $f: E \rightarrow F$ such that $f(E_x) \subseteq E_y$ and $f_x: E_x \rightarrow F_x$ is linear. The vector bundle $E$ is said isomorphic to $F$ if there if $f$ such that $f_x$ is an isomorphism for all $x \in M$. A vector bundle isomorphic to the product $M \times \C^k$ is called trivial.

\item[2)] A {\it section} $\delta$ of the vector bundle $E \rightarrow^{\pi} M$ over $U \subseteq M$ is a $C^{\infty}$-map $\sigma: U \rightarrow E$ such that $\sigma(x) \in E_x$ for all $x \in U$. In other words, is a map $\sigma: U \rightarrow E$ such that $\pi \circ \sigma = Id$. The set of sections (resp. local sections) is denoted by $\Gamma(E)$ (resp. $\Gamma_U(E)$).

\item[3)] A {\it frame} for $E$ over $U \subseteq M$ is a collection of $\sigma_1, \, \ldots, \, \sigma_k$ sections of $M$ over $U$ such that $\{\sigma_1(x), \, \ldots, \, \sigma_k(x)\}$ is a basis of $E_x$ for all $x \in U$. A frame is essentially the same thing as a trivialization over $U$.
\end{itemize}
\end{defi}

Now let us assume that $M$ is a {\it complex manifold}. A
{\it holomorphic vector bundle}\, $E \rightarrow M$ is a complex
vector bundle together with a structure of complex manifold
on $E$ such that the trivializations
$$
\varphi_U \, : \; E_U \longrightarrow U \times \C^k
$$
are holomorphic diffeomorphisms (we say that they are holomorphic
trivializations). Transition maps $g_{UV} : U \cap V \rightarrow
{\rm GL} \, (k ,\C)$ are now {\it holomorphic}. Conversely given
holomorphic maps $g_{UV} : U \cap V \rightarrow
{\rm GL} \, (k ,\C)$ satisfying identities~(\ref{ident1}),
(\ref{ident2}), we can construct a holomorphic vector bundle
realizing these maps as its transition functions.

\noindent {\bf Example 1.2.D}: The complex tangent bundle.

Let $M$ be a complex manifold and $T_xM$ the complex
tangent space of $M$ at $x\in M$ whose dimension over $\C$
is~$2N$. Consider a neighborhood
$U$ of $x$ with a coordinate chart $\psi_U : U \rightarrow \C^n$.
This coordinate induces a map
$$
\psi_U^{\ast} : T_x M \rightarrow T_{\psi_U (x)} \C^n \simeq
\C \{ \partial /\partial x_j \; , \; \partial /\partial y_j \} \simeq
\C^{2n}
$$
for each $x$ in $U$. Hence we have a map
$$
\psi_U^{\ast} : \bigcup_{x \in U} T_xM \longrightarrow
U \times \C^{2n}
$$
which gives $TM =\bigcup_{x \in U} T_xM$ the structure of
a complex vector bundle called the {\it complex tangent
bundle}\, of $M$. Transition functions for $TM$ are
$$
j_{UV} = {\rm Jac}_{\R} (\psi_U \circ \psi_V^{-1}) \, .
$$
Now $T_xM$ splits into $T_xM = T_xM' \oplus T_xM''$
where $T_xM' = \C \{ \partial /\partial z_j \}$ with
$\partial /\partial z_j = \partial /\partial x_j
- \imath \partial /\partial y_j$. The collection of
$T_xM'$ forms a subbundle of $T_xM$ which is called the
{\it holomorphic tangent bundle}\, of $M$. The last bundle
has transition functions given by
$$
j_{UV} = {\rm Jac}_{\C} (\psi_U \circ \psi_V^{-1}) \, .
$$

\noindent {\bf Example 1.2.E}: $\widetilde{\C}^2 \rightarrow \pi^{-1}
(0)$.

Recall that the blow-up $\widetilde{\C}^2$ of $\C^2$ can be
viewed as a holomorphic line bundle (i.e. a holomorphic
vector of rank~$1$) over the exceptional divisor $\pi^{-1} (0)
\simeq \cpum$. Consider the coordinates $(x,t)$, $(s,y)$
for $\widetilde{\C}^2$ introduced in paragraph~1.2 and set
$$
U = \{ (x,t) \in \C^2 \; ; \; t \neq 0 \} \; \; \; \;
{\rm and} \; \; \; \; V = \{ (s,y) \in \C^2 \; ; \;
s \neq 0 \} \, .
$$
We also have the projection $\pi$ given in coordinates by
$\psi_U (x,t) = (x, tx)$ and $\psi_V (s,y) = (sy,y)$.
The transition function for the holomorphic
tangent space of $\widetilde{\C}^2$ is the Jacobian matrix
of $\psi_V \circ \psi_U^{-1}$ for $ts \neq 0$. Thus we
obtain
$$
{\rm Jac} \, (\psi_V \circ \psi_U^{-1}) = \left( \begin{array}{cc}
                                               -1/t^2 & 0 \\
                                                x & t
                                               \end{array} \right)
\, .
$$
Letting $x=0$ we obtain in particular the transition function
for the tangent and normal bundles of $\pi^{-1} (0)$ which
are respectively $-1/t^2$ and $t$.

Among vector bundles, line bundles (i.e. complex bundles of
rank~$1$) play a prominent role in the study of Complex ODEs as well
as in Algebraic Geometry as a whole. For this reason, we are going
to say a few words about their classification, further information
can be obtained in Chapter~4. To begin with, we note that the set of
holomorphic line bundles over a compact complex manifold $M$ form a
group under the tensor product. The group structure can be made
explicit by considering two line bunles $L,\, L'$ and a common open
covering $\{ U_i \}$ for $M$ such that both $L, \, L'$ are trivial
over the $U_i$'s. In particular, for $i\neq j$ and $U_i \cap U_j
\neq \emptyset$, we have transition functions $g_{ij}, \, g_{ij}'$
respectively for $L, \, L'$. By definition, the functions $g_{ij},
\, g_{ij}'$ take values on $\C^{\ast}$. Setting $h_{ij} = g_{ij} .
g_{ij}'$, it is immediate to check that the functions $h_{ij}$
defined on $U_i \cap U_j$ verify the natural cocycle relations so that they define
a new line bundle over $M$. The line bundle associated to these transition
functions $h_{ij}$ is $L \otimes L'$. The resulting group is called
the {\it Picard group}\, of $M$ and it is denoted by ${\rm Pic}\, (M)$.

In general the structure of ${\rm Pic}\, (M)$ is rather subtle and varies with
the holomorphic structure on $M$ (for a fixed underlying differentiable structure).
In particular, if $M$ is in addition {\it algebraic}\, then ${\rm Pic}\, (M)$ is
isomorphic to the group of divisors of $M$ (cf. Chapter~4). On the other hand,
the classification of complex line bundles in class $C^{\infty}$, i.e. in differentiable
category, depends only on certain characteristic classes associated to $M$. These
classes are called {\it Chern classes}. The first Chern class, $c_1(L)$, admits a
particularly simple geometric interpretation. We consider a $C^{\infty}$ section
$\sigma$ of $L$ such that $M$ and $\sigma (M)$ are in general position. The
intersection $M \cap \sigma (M)$ is then a codimension $2$ compact submanifold
of $M$. Thus it defines an element of the homology of $M$ in codimension~2 with
integral coefficients. Finally Poincar\'e duality yields an element in $H^2(M, \Z)$
which is exactly $c_1 (L)$.

When $M$ has dimension~$1$ i.e. it is a Riemann surface, then we have
some precise statements. Already, it follows from the preceding that the
$C^{\infty}$ classification of complex line bundles is given by the
self-intersection of the null section. In particular they are topologically
trivial if and only if the self-intersection equals {\it zero}.
As to the holomorphic classification, we have (cf. \cite{beau})

\begin{lema}
\label{linecp1}
Two holomorphic line bundles over $\cpum$ are isomorphic
(holomorphically diffeomorphic) if and only if their null-section
has the same self-intersection.
\end{lema}

By virtue of Lemma~(\ref{linecp1}) above, it is easy to construct
models for all line bundles over $\cpum$. Indeed to construct
a line bundle whose null-section has self-intersection $n \in \Z$
we take two copies $(x,y)$, $(u,v)$ of $\C^2$ and identify
the points satisfying $u=1/x$ and $v= x^{-n}y$.

Consider now holomorphic line bundles over an elliptic curve $S$. More precisely let
us restrict our attention to the case of topologically trivial line bundles. Unlike
rational curves, these holomorphic line bundles have ``moduli''. The following is also
very well known (cf. \cite{arnold}).

\begin{lema}
\label{lineelliptic}
Topologically trivial holomorphic line bundles over an elliptic
curve $S$ are in one-to-one correspondence with points in $\C^{\ast}$.
\end{lema}


\section{Tubular Neighborhoods}

When a (complex) manifold $M$ contains a (complex) submanifold $S$,
it is interesting to analyse the structure of a neighborhood of $S$
in $M$. This is of much interest for Complex EDOs since they occasionally
possess a {\it compact leaf}. It is then natural to investigate the structure
of our equation on a neighborhood of this leaf since it makes a lot of
information on the global behavior of the equation accessible to us. In
particular, it is often necessary to know the structure of the neighborhood
of this leaf as a submanifold.
In the real setting, this structure is determined by the
normal bundle of $S$ in $M$ as follows from the classical tubular
neighborhood theorem whose statement we are going to recall.

\begin{defi}
\label{tubuneigh}
Let $S$ be a (complex) submanifold of a (complex) manifold
$M$. A tubular neighborhood of $S$ in $M$ is an open set
$V \subset M$ together with a (holomorphic) submersion $\Pi:
V \rightarrow S$ such that
\begin{itemize}

\item $V \stackrel{\Pi}{\longrightarrow} S$ is a (holomorphic)
vector bundle.

\item $S \subset V$ is naturally identified with the
{\it zero} section of this vector bundle.
\end{itemize}
\end{defi}

\noindent According to the well-known ``tubular neighborhood'' theorem,
a tubular neighborhood always exist in the differentiable category.
Indeed a tubular neighborhood $V$ is isomorphic (as real vector
space) to the normal bundle $T_N S$ where the fiber $N_x$ over
$x \in S$ is isomorphic to $T_x M /T_x S$. Actually the tubular
neighborhood should be thought of as a ``realization'' of the normal
bundle as an open neighborhood of $S$ in $M$.

In the holomorphic setting a holomorphic tubular neighborhood need
not exist in general. This means that there may exist {\it no
holomorphic}\, diffeomorphism between an neighborhood $V$ of $S
\subset M$ and a neighborhood of the zero-section in the normal
bundle of $S$.

\begin{obs}
{\rm In principle this may cause some confusion
since the expression ``tubular neighborhood'' is often used to refer to
a neighborhood of a submanifold without taking in consideration the
existence of any diffeomorphism between the neighborhood in question
and a neighborhood of the null section of the corresponding normal bundle.
Maybe we should use another word, for example collar neighborhood, to tell
apart one situation from the other. Unfortunately, it seems that this
``abuse of language'' is already well established. Hopefully it will lead
to no misunderstanding.}
\end{obs}

To provide an example where a holomorphic
diffeomorphism as above does not exist, it suffices to consider a trivial
fibration over the unit disc $D \subset \C$ whose fibers are
elliptic curves with different complex structures. Clearly the
normal bundle of the fiber $L_0$ over $0 \in D$ is trivial and, if a
holomorphic tubular neighborhood existed, then the restriction of
$\Pi$ to a fiber $L_{\varepsilon}$ would be a holomorphic
diffeomorphism between $L{\varepsilon}$ and $L_0$. This is
impossible provided that we vary the complex structures of the
fibers. Here is an explicit construction.

We consider an elliptic curve $\Gamma$ of periods $2\pi$ and $\tau$. We also
consider $\C \times D$ equipped with coordinates $(x,y)$ and impose the
identifications
$$
(x,y) \simeq (x + 2\pi , y ) \simeq (x + \tau ,y) \, .
$$
Next let us consider the identifications on $\C \times D$ given
by $(x,y) \simeq (x + 2\pi , y) \simeq (x + \tau , \lambda y)$
where $\lambda \in \C^{\ast}$. With these identifications
$\C \times D$ becomes a surface $\Sigma$ and $\{ y =0 \}$ can be thought
of as an embedding of an elliptic curve $\Gamma$ in $\Sigma$.
Furthermore a neighborhood of the null section in
the normal bundle of $\Gamma$ in $\Sigma$ is equivalent to
a neighborhood of $\Gamma$ in the original $\C \times D$. However
neighborhoods of $\Gamma$ in $\C \times D$ and in $\Sigma$
are not equivalent.
Actually even more is true: $\Sigma$ is a topologically trivial
line bundle over $\Gamma$. However, if $\lambda \neq e^{2\pi i k}$,
Fourier series shows that $\Gamma$ is a {\it rigid curve} in
$\Sigma$ (i.e. $\Gamma$ does not admit a holomorphic deformation)
which contrasts with the topological triviality of $\Sigma$
as bundle over $\Gamma$.

It is an equally remarkable fact that, as far as fibrations are
concerned, the preceding example is universal in a precise sense. To
explain it, let us consider a $C^{\infty}$-trivial fibration whose
fibers are complex manifolds (as well as the basis and the
projection). More generally the reader may consider a smooth family
of compact complex manifolds, i.e. a triplet consisting of a pair of
complex spaces ${\mathcal X}$ and ${\mathcal S}$ along with a proper
holomorphic map $\Pi$ form ${\mathcal X}$ onto ${\mathcal S}$ with
connected fibers and {\it flat}. The condition of {\it flatness}\,
of $\Pi$ can be expressed by saying the ring of germs of holomorphic
functions at every point $x \in {\mathcal X}$ is a module over the
ring of germs of holomorphic functions on $\mathcal X$ at $\Pi (x)$.
If the spaces ${\mathcal X}, \, {\mathcal S}$ are smooth, this
condition simply means that $\Pi$ is a submersion. The ``universal''
character of our examples results from the following theorem.

\begin{teo}
{\rm ({\bf Fisher \& Grauert})}
A fibration as above is locally {\it holomorphically
trivial}\, if and only if all fibers are isomorphic as complex
manifolds.\qed
\end{teo}

As a matter of fact, it appears that the existence of tubular neighborhoods
for complex submanifolds was not very intensively studied. For the case of
Riemann surfaces, a remarkable theorem due
to Grauert gives a sufficient condition for the existence of
a holomorphic tubular neighborhood.

\begin{teo}\label{graue}
{\rm ({\bf Grauert})} Assume that $S$ is a Riemann surface embedded
as complex submanifold of a complex surface $M$. If the self-intersection
of $S$ is strictly negative, then there exists a holomorphic
tubular neighborhood for $S$ in $M$.
\end{teo}

Before discussing further Theorem~(\ref{graue}), let us state a useful
corollary which was already known to the Italian geometers. If $S$ is a Riemann
surface in a complex surface $M$, we say that $S$ is
{\it contractible} if there exists another complex surface
$N$ and a proper holomorphic map $\pi : M \rightarrow N$
which collapses $S$ into a point $p \in M$ and induces
a holomorphic diffeomorphism from $M\setminus S$
to $N \setminus \{ p \}$.

\begin{coro}
Let $S$ be a rational curve in $M$ and assume that $S$ has
self-intersection~$-1$. Then $S$ is contractible.
\end{coro}

\begin{proof}
Because of Grauert's theorem, a neighborhood
of $S$ in $M$ is equivalent to a neighborhood of $\pi^{-1} (0)$
in $\widetilde{\C}^2$. It is then obvious that $S$ can be collapsed.
\end{proof}

Let us now make some comments concerning Grauert's theorem or, more generally,
the existence of tubular neighborhoods for compact Riemann surfaces $S$ embedded
in complex surface $M$.

To begin with it is necessary to point out the natural role played by
the assumption on the negativity of the self-intersection of $S$. This lies
in the fact that a Riemann surface $S$ with negative self-intersection in $M$
is {\it isolated}\, in the sense that there is a neighborhood $U$ of $S \subset M$
containing no embedded compact Riemann surface other than $S$ itself. Otherwise,
we may consider a Riemann surface $S'$ contained in a sufficiently small
neighborhood $U$ as a {\it deformation of} $S$, i.e. they belong to the same
class in the second homology group of $M$. Hence we would have $S.S = S.S'$. However
$S.S'$ is nonnegative as the intersection number of two complex submanifolds
and the resulting contradiction shows that $S$ is isolated in the mentioned sense.
Clearly for an isolated Riemann surface $S$, the mechanism we have used before to exclude the
existence of a tubular neighborhood (as in Definition~\ref{tubuneigh}) for $S$
can no longer apply: we cannot find in $U$ other Riemann surfaces with complex
structure different from that of $S$.

A second comment involving Theorem~(\ref{graue}) is that it suggests a rather
naive way to study the case of nonnegative self-intersection. This consists of
blowing up a number of points in $S$ so as to bring the self-intersection of
$S$ to, say, $-1$ (where $S$ is naturally identified with its proper transform).
The effect of these blow-ups amounts to adding a finite number of exceptional
rational curves having normal crossings with $S$. In the new manifold, $S$ has
negative self-intersection and thus a neighborhood of it is isomorphic to a
neighborhood of the null section in the normal bundle. To establish the existence
of a tubular neighborhood for the ``initial'' $S$ is then equivalent to show that
the isomorphism in question can be extended to a neighborhood of the exceptional
curves in a natural way. Although naive, this remark is useful in some cases,
especially if one wants to prove index theorems.



\section{Miscellaneous}

This section is intended to collect several results, some of them very
non-trivial, that are needed for a better understanding of the material
discussed in the subsequent chapters. The discussion below is by no means
self-contained. In fact, at some points it is necessary to use notions
that are introduced only in Chapter~3, this is of course compounded with
the several results presented without proof. To remedy for this situation,
we try to keep the following two botton lines:
\begin{enumerate}

\item Precise references for the theorems stated without proofs.

\item Whenever we state a theorem using notions and/or terminology that
will be introduced only later, we also state the most used especial versions
of it in a simpler language that hopefully can immediately be ``decoded''.
\end{enumerate}

The section is organized in three main families of results: we begin
with results belonging to Complex Analysis more directly. The second type of
results discussed revolves around Serre's GAGA theorems. Finally we shall discuss
the nature of automorphisms of compact complex manifolds.


\subsection{Local results in Complex Analysis}

We are going to begin with the well-known Weierstrass preparation and division theorems for
germs of holomorphic functions. The corresponding proofs can be found in any standard book
about several complex variables.

\begin{teo}
\label{WPTHM}
{\rm ({\bf Weierstrass Preparation Theorem})}
Let $f$ be a holomorphic function defined on a neighborhood $U$
of the origin of $\C^n$ having the form $U =U' \times \{ \vert z_n \vert
<r\}$ for some $r >0$ and where $U'$ is a neighborhood of the origin
in $\C^{n-1}$. Suppose also that $f(0, z_n)$ is not identically zero on
the disc $\vert z_n \vert <r$. Suppose also that $f(0,z_n)$ does not
have zeros on the circle of radius $r$ and let $k$ denote the number
of its zeros, counted with multiplicities, in the disc of radius~$r$.
Then, in a suitable neighborhood $V = V' \times \{ \vert z_n \vert < r\}
\subset U$, we have
$$
f(z',z_n) = u(z',z_n) . (z_n^k + c_1(z') z_n^{k-1} + \cdots + c_k (z'))
$$
where $z' = (z_1, \ldots ,z_{n-1})$, $u(0,\ldots ,0) \neq 0$ and $c_i$
is holomorphic for $i=1 ,\ldots ,k$.
\end{teo}

It is convenient to say that a {\it Weierstrass polynomial}\, of degree
$k$ in $z_n$ is a monic polynomial of degree $k$ in $z_n$ having coefficients
that are holomorphic functions of $z'=(z_1, \ldots ,z_{n-1})$. Here these
coefficients, other than the leading one, are supposed to vanish at the
origin. Summarizing, a Weierstrass polynomial as above is a function of
the form
$$
z_n^k + c_1(z') z_n^{k-1} + \cdots + c_k (z') \, ,
$$
with $c_i$ holomorphic satisfying $c_i (0) =0$, $i=1, \ldots ,k$.

The companion result of Theorem~(\ref{WPTHM}) is the so-called
Weierstrass Division Theorem stated below:

\begin{teo}
\label{WDTHM}
{\rm ({\bf Weierstrass Division Theorem})}
Let $f$ be as before. Consider a Weierstrass polynomial $h$ of degree
$k$ in $z_n$. Then $f$ can uniquely be written in the form $f = gh + q$
where $g$ is holomorphic on a neighborhood of the origin in $\C^n$ and
$q$ is a polynomial in $z_n$ of degree strictly less than~$k$.
\end{teo}

Recall that an {\it analytic subvariety $V$}\, of a complex manifold $M$
is a subset which is locally given as the intersection of the zero-set
of finitely many analytic functions. This means that if $p \in V$, there
is a neighborhood $U \subset M$ containing $p$ together with holomorphic
functions $g_1 ,\ldots ,g_r$ defined on $U$ such that $V \cap U =
\bigcap_{i=1}^r g_i^{-1} (0)$. It is clear that if we have an holomorphic
map $f$ between complex manifolds $M,\, N$, the preimage by $f$ of an analytic
subvariety $V \subset N$ is itself an analytic subvariety of $M$. This is however
not true for {\it direct images}\, unless the map in question is, in addition, proper.
Indeed this is the contents of the
{\it Proper Mapping Theorem}, also called Remmert's mapping theorem,
whose statement is as follows.

\begin{teo}
\label{PMT}
Consider complex manifolds $M, \, N$ along with a proper holomorphic map
$f: M \rightarrow N$. If $V \subset M$ is a closed analytic subset of $M$, the
its image $f(V) \subset N$ is closed as well.\qed
\end{teo}


It is also relevant to have some feeling concerning the local nature
of an analytic subvariety (also called ``an analytic set''). In the sequel
we present
a rather elementary and yet useful way to ``parametrize'' one given set.
The procedure is very standard and is detailed in several textbooks.
Since we are dealing with local parametrizations, it suffices
to consider analytic sets defined on a neighborhood of the origin in $\C^n$
which will still be denoted by $V$.

Let us begin with the case of {\it curves}, i.e. analytic sets of dimension~$1$
(the dimension of a singular set is the maximal dimensional of the complement
of the singular set). Note also that the singular set must have codimension at least
one inside the analytic set in question. In particular for the case of (germs
of) analytic curves, we can suppose that the origin is the unique singularity of
it. Germs of curves are known to admit a {\it Puiseaux
parametrization} (cf. for example \cite{fultoncurves}). Precisely
there exists a holomorphic map ${\sc P}: \D \rightarrow V$,
where $\D \subset \C$ stands for the unit disc, satisfying the conditions
below.
\begin{itemize}

\item ${\sc P} (0) =(0, \ldots ,0)$ and ${\sc P}$ is one-to-one from
$\D$ to $V$.

\item ${\sc P}$ is a diffeomorphism from $\D^{\ast}$ to $V \setminus \{ (0,
\ldots ,0)\}$.
\end{itemize}
In terms of a coordinate $t \in \D$, we can set
$$
{\sc P} (t) = (t^k , {\sc P}_2 (t), \ldots , {\sc P}_n (t)) \, .
$$

Unfortunately, the existence of the Puiseaux parametrization does
not generalize to higher dimensional analytic sets. Parametrizing
these sets is therefore a less precise procedure that essentially amounts to
noticing that they (locally) are ramified coverings of a plane of suitable dimension.
We describe below only the associated geometric picture. Proofs for our
statements can easily be produced with the help of Weierstrass Preparation Theorem.


\subsection{Introduction to GAGA theorems and to Hartogs principle}

Let us now move to the second topic to be discussed in the section, namely
the GAGA theorems \cite{gaga}. In vague but meaningful terms, one states the
{\it GAGA principle}\, by saying that {\it any analytic object defined over an
algebraic manifold is indeed algebraic}.

To begin the term {\it algebraic variety} will be used in what follows
to refer to an irreducible Zariski-closed subset of some complex projective
space  $\C \Pp (n)$. Recall that a Zariski-closed subset of
$\C \Pp (n)$ is nothing but the common zero-set of a finite
collection of homogeneous
polynomials in the variables $(x_0, x_1 ,\ldots ,x_n)$. It should be
pointed out here that Hilbert Basis Theorem tells us that one does not
obtain a ``larger class'' of sets by allowing infinite collections of
these polynomials. Finally the term ``irreducible'' means that the ideal
of regular functions vanishing identically on this set forms a prime
ideal of the ring of regular functions on $\C \Pp (n)$.
As usual, by an {\it algebraic
manifold} it is meant a {\it smooth}\, algebraic variety. The first
instance of ``GAGA principle'' is manifested by the celebrated Chow Lemma.

\begin{teo}
\label{chowlemma}
{\rm ({\bf Chow Lemma})}
An analytic subvariety of a projective space is algebraic.
\end{teo}

On algebraic varieties one can consider the ring of regular
functions as well as its fraction field called the field of rational
functions on $M$. It is well-known that a rational function on $V$
is given in (say inhomogeneous) coordinates as the quotient of two
polynomials. A rational function is said to be regular at $p \in V$,
if it can be represented as a quotient of polynomials with denominator
different from {\it zero} at $p$. A regular function is nothing but
a rational function that is regular at every point of $V$. Finally
a rational (resp. regular) map is simply a map whose coordinates are
rational (resp. regular) functions. In particular, a rational map between
algebraic varieties $V_1 \subset \C \Pp (n)$ and $V_2 \subset \C \Pp (m)$
is given by the restriction of a rational map from $\C \Pp (n)$
to $\C \Pp (m)$. All this is fairly basic material in algebraic geometry
that can be found in any textbook, for example \cite{shafarevich1,shafarevich2}.

The second fundamental instance of GAGA principle can then be stated
as

\begin{teo}
\label{merofunctions}
Every meromorphic function defined on an algebraic variety is rational.
\end{teo}

Nice treatments of the above statements are given in \cite{gh} and in
\cite{mumford}. The proof in \cite{mumford} has the advantage
of being rather self-contained while the proof in \cite{gh} is shorter,
however it relies on the Proper Mapping Theorem.




\subsection{The automorphism group of complex manifolds}

Finally let us briefly review some basic facts about automorphism
groups of a compact complex manifold. The first very general and
important result is:

\begin{teo}
\label{liegroup}
The automorphism group of a compact complex manifold $M$ is a complex Lie
group of finite dimension acting holomorphically on $M$.
\end{teo}

A proof of this result can be found in the nice survey article
of B. Deroin \cite{bertrand}. It is based on a couple of very
important theorems whose interest goes beyond Theorem~\ref{liegroup}.
These are as follows.

\begin{teo}
\label{cartanserre}
{\rm ({\bf Cartan \& Serre}, \cite{cartanS})}
Let $E$ be a (finite dimensional)
holomorphic vector bundle over a compact complex manifold $M$.
Then the vector space $\Gamma \, (M, E)$ consisting of holomorphic sections
of $E$ has finite dimension.
\end{teo}

\begin{teo}
\label{bochnermontgomery}
{\rm ({\bf Bochner \& Montgomery}, \cite{bochnerM})}
Let $M$ be as before. Then every automorphism of $M$ sufficiently
$C^0$-close to the identity is the time-one map induced by a
holomorphic vector field on $M$.
\end{teo}

Suppose we want to determine the connected component containing
the identity of the group of automorphism of a
complex manifold $M$. An equivalent formulation for this problem
is to describe the Lie algebra of this automorphism group
denoted by ${\rm Aut}\, (M)$. We can starting thinking of
the $n$-dimensional complex projective space $\C \Pp (n)$. It is well-known
that the corresponding
automorphism group is ${\rm PSL}\, (n+1, \C) = {\rm PGL}\, (n+1, \C)$
and the associated action is nothing but the projectivization
along radial lines of the linear action of the invertible matrices
of ${\rm GL}\, (n+1, \C)$ on $\C^{n+1}$. As a motivation for our
discussion, and in particular for the classical Blanchard Lemma, let us
sketch a proof of this fact. First of all, it is clear that
${\rm PSL}\, (n+1, \C)$ is contained in ${\rm Aut}\, (\C \Pp (n))$.
If $g$ is an algebraic automorphism automorphism of $\C \Pp (n)$
then it is represented in affine coordinates by algebraic maps, ie
quotient of polynomials. Imposing that these expressions must glue
together by change of coordinates is a very strong condition and allows
us to show that $g$ coincides with the action of an element of
${\rm PSL}\, (n+1, \C)$.

The above discussion did not totally solve the problem since there
might exist {\it nonalgebraic}\, automorphisms. Naturally the
GAGA theorems allows to rule out this possibility and thus to actually
establish that ${\rm Aut}\, (\C \Pp (n))$ coincides with
${\rm PSL}\, (n+1, \C)$. However Blanchard Lemma provides with a
more elementary alternative that does not rely on GAGA Principle.
First we state Blanchard Lemma. Here is convenient to recall that
a {\it fibration}\, on a compact complex manifold $M$ is nothing but a
holomorphic map $\calp : M \rightarrow N$ onto another compact
complex manifold of smaller dimension that is a submersion at
generic points. If $S \subset M$ stands for the singular values
of $\calp$, then  the well-known theorem of Ereshemann
implies the existence of a {\it locally trivial differentiable
fibration}\, of $M \setminus \calp^{-1} (S)$ over $N \setminus S$.
Now we have:

\begin{teo}
\label{blancahrd}
{\rm ({\bf Blanchard Lemma})}
Consider a fibration as above in a complex compact manifold
$M$. Then every holomorphic vector field globally defined on $M$
is such that its flow sends fibers of $\calp$ to fibers of $\calp$.
\end{teo}

\begin{proof}
The theorem is of local nature so that
we consider a open set $U$ of $M$ fibering over unit polydisc of some
$\C^k$. Note that every vector field as in the statement is complete
since $M$ is compact. Next let $\varphi^t$ denote the induced flow
for $t \in \C$. Denote by $F_0$ the fiber sitting over the origin of
$\C^k$. If $t$ is sufficiently small in absolute value, it follows
that the image $\varphi^t (F_0)$ is still contained in $U$. It can
therefore be composed with the projection from $U$ to the unit
polydisc. The resulting map is holomorphic from $F_0$ to the
polydisc. Because $F_0$ is compact, this map must be constant. In
other words, $\varphi^t$ is contained in a fiber of $\calp$. The
statement promptly follows from this observation.
\end{proof}

Since ${\rm Aut}\, (M)$ is a Lie group of finite dimension, its
connected component ${\rm Aut}_0 (M)$ containing the identity
is obtained through the corresponding Lie algebra. Precisely, the
elements of ${\rm Aut}_0 (M)$ are induced by the holomorphic
vector fields carried by $M$. In this context, the
upshot of Blanchard Lemma is that, by construting fibrations in
our manifolds, the problem of identifying holomorphic vector fields
on $M$ can be reduced to smaller dimensional cases. In particular,
if $M$ is algebraic, then fibrations can always be constructed by
considering the ``level hypersurfaces'' of a nonconstant meromorphic
function on $M$. Strictly speaking these hypersurfaces define a pencil,
rather than a fibration, on $M$. Indeed, a meromorphic function is
in general not defined on a subvariety of codimension~$2$. Desingularization
results however allow us to easily turn this pencil in an actual
fibration.

In the special case of $\C \Pp (n)$, we can consider the
meromorphic function given in affine coordinates by
$x_2/x_1$. Clearly this function is not defined on the subvariety
$\{x_2=x_1=0\}$. Nonetheless, we blow this submanifold up to produce
a new manifold $M$ where it is defined a fibration whose basis is
$\C \Pp (1)$ and whose typical fiber is $\C \Pp (n-1)$. Thanks to
Blanchard Lemma, holomorphic vector fields on $M$ can be projected
over the basis to define a vector field on $\C \Pp (1)$.
When this projection is trivial, the vector field must be tangent
to the fibers and, in particular, it has to induce a holomorphic
vector field on $\C \Pp (n-1)$.

To start out the recurrence procedure, it is therefore enough to
determine the automorphism group of $\C \Pp (1)$. This case can
easily be handed: it suffices to consider the stereographic projection
of $\C \Pp (1)$ to $\C$ and apply, for example, Picard theorem to
conclude that an automorphism of $\C$ must have an ``algebraic nature''
at infinity. With this information in hand, the recurrence procedure
described above becomes effective. In particular, for $\C \Pp (2)$,
where the blown-up manifold $M$ coincides with the first Hirzebruch
surface $F_1$, the reader can check \cite{akh} for explicit
calculations.

It remains the problem of passing from ${\rm Aut}_0 (M)$ to ${\rm Aut}\, (M)$.
This is more elaborate. In the case of $\C \Pp (n)$ we can exploit the
size of ${\rm Aut}_0 (\C \Pp (n))$ to show that ${\rm Aut}\,
(\C \Pp (n))$ actually coincides
with ${\rm Aut}_0 (\C \Pp (n))$. Let us briefly sketch the argument.
Suppose that $f$
is an element of ${\rm Aut}\, (\C \Pp (n)) \setminus {\rm Aut}_0 (\C \Pp (n))$.
Then $f$ naturally acts on the Lie algebra
of holomorphic vector fields of $\C \Pp (n)$. We claim the existence
of an eigenvector for this action, ie. there is a holomorphic vector
field $X$ on $\C \Pp (n)$ such that $f^{\ast} X = cX$ where $c \in \C$.
The proof amounts to observe that ${\rm Aut}_0 (M) \simeq
{\rm PSL}\, (n+1, \C)$ acts transitively on $n+2$-tuples of points.
Hence we can suppose that $f$ has, say, $n+1$ fixed points. Besides the
Lie subalgebra of vector fields having singularities at the $n+1$ fixed
points of $f$ is generated by a single vector field $X$. Thus we must
have $f^{\ast} X = cX$ as desired. To finish the proof, we observe that
$X$ has a first integral, i.e. its orbits can naturally be compactified
into rational curves in $\C \Pp (n)$. By construction $f$ must preserve
the corresponding pencil of rational curves so that we can now apply a
recurrence procedure to obtain a contradiction with the assumption
that $f$ belongs to
${\rm Aut}\, (\C \Pp (n)) \setminus {\rm Aut}_0 (\C \Pp (n))$.

Even in dimension~$2$, an automorphism of an algebraic surface can have
a ``rich dynamics'', in particular, they need not preserve any fibration
on $M$. Clearly these diffeomorphisms cannot belong to ${\rm Aut}_0 (M)$.
There is however a simple way to measure the chances of an automorphism
of an algebraic surface $M$ be strictly more complicated than those
belonging to ${\rm Aut}_0 (M)$. In fact, if the corresponding
action on the second homology group $H^2 (M)$ of $M$ is trivial, then
the automorphism must preserve every fibration defined on $M$. The
proof of this claim is given below; we left to the reader the straightforward
generalizations of this lemma to higher dimensions.

\begin{lema}
\label{exploitingintersection}
Suppose that $M$ is a compact complex surface supporting a fibration
$\calp : M \rightarrow S$ over a Riemann surface $S$. Let $f$
be an automorphism of $M$ acting trivially on $H^2 (M)$. Then $f$
preserves $\calp$.
\end{lema}

\begin{proof}
We begin by noticing that all the fibers of
$\calp$ are homologous one to the others. Thus they are associate
to a well defined class $[L]$ in $H^2 (M)$. Since $f$ acts trivially
in $H^2 (M)$, we have $f^{\ast} [L] =[L]$. On the other hand,
since $L$ is a fiber of a fibration, one has $[L].[L] =0$. However,
if $f$ did not preserve $\calp$, then the image $f(L)$ of a generic chosen fiber
$L$ would intersect other fibers non trivially another fiber, say $L'$.
Since $f(L)$ and $L'$ are both complex submanifolds of $M$, it would
follows that $\sharp (f(L) \cap L' ) >0$ (strictly). However the intersection
number depends only on the homology class of the submanifold, so we
have
$$
0 < [f(L)][L'] = [L][L'] = [L][L] = 0 \, .
$$
The resulting contradiction establishes the lemma.
\end{proof}

\chapter{Singularities of holomorphic vector fields}


\section{Introduction}

The basic topic of the present chapter is the study of singularities
of complex vector fields in dimension $2$. In the preceding chapter
we had already traced a parallel between real and complex ordinary
differential equations and point out the main aspects in which they
differ. For instance, the analogue to the maximal interval of
definition for solutions of {\it real} ODEs does not exist, in general,
in the complex case. On the other hand, the geometric idea of viewing
the solutions of real ODEs (away from the singular set) as
$2$-dimensional leaves of a foliation can easily be transported to
the complex scenario. In the case of complex ODEs, we can say that
the orbits are Riemann Surfaces and that, away from the
singularities, they are leaves of a holomorphic foliation.

In Section~$2$ we give some basic results related to singularities
of holomorphic foliations. We begin with the foliation $\fol$
associated to a vector field $X$, having an isolated simple
singularity at $(0,0)$, with non-zero eigenvalues $\lambda_1$ and
$\lambda_2$. In other words,
\begin{equation}\label{eq38}
X(x_1,x_2)=[\lambda_1
x_1+\varphi_1(x_1,x_2)]\frac{\partial}{\partial x_1}+[\lambda_2
x_2+\varphi_2(x_1,x_2)]\frac{\partial}{\partial x_2}\,.
\end{equation}
We discuss the problem of linearizing such vector fields. More
precisely, understand under which conditions there exists a
holomorphic change of coordinates that linearizes the system. In
fact, a {\it formal} change of coordinates can easily be found,
except for very specific resonant cases. What is more challenging
is to prove its convergence. It becomes clear that the existence
of a holomorphic change of charts linearizing $X$ depends
entirely on the values of the eigenvalues. For instance, if
$\lambda_1/\lambda_2$ do not belong to $\R_-$ and if neither
$\lambda_1/\lambda_2$, nor $\lambda_2/\lambda_1$ belong to $\N$
then, in appropriate coordinates, it is indeed linear. This is the
contents of the well-known Poincar\'{e} Linearization Theorem. Now
if the singularity belongs to the Siegel domain, i.e.,
$\lambda_1/\lambda_2\in\R_-$, we cannot find a linearizing
holomorphic change of coordinates. Nevertheless, in local
coordinates $(y_1,y_2)$, Equation~\ref{eq38} may be expressed as
\[
X = \lambda_1 y_1[1 +
(h.o.t.)] \partial /\partial y_1 + \lambda_2 y_2 [1 + (h.o.t.)]
\partial /\partial y_2 \, .
\]

Next we shall investigate the case in which one of the
eigenvalues, say $\lambda_2$ is {\it zero} and $\lambda_1\neq 0$.
Such singularities are called {\it saddle-nodes}. We obtain a
normalization for this type of singularity, known as Dulac's
Normal Form. This result simply states that vector fields
containing saddle-node singularities may be given in local
coordinates $(y_1,y_2)$ by
\begin{equation}\label{eq39}
X(y_1,y_2) = [y_1 (1 + \lambda y_2^p) + y_2
R(y_1,y_2)] \, \frac{\partial}{\partial y_1} \; + \; y_2^{p+1} \,
\frac{\partial}{\partial y_2} \, ,
\end{equation}
up to an invertible factor.

Subsection~$3.2$ is devoted to a brief study of singularities in
higher dimensions. In particular we give a generalization of
Poncar\'{e} Linearization Theorem and some results related to
saddle-node singularities in dimension $3$.

Section~$3$ is the core of the present text and undoubtedly
contains the most important results of this text. It is strongly
inspired in the work of J.-F. Mattei and R. Moussu (cf.
\cite{M-M}) and J. Martinet and J.-P. Ramis (cf. \cite{Ma-R}). We
begin by revisiting saddle-node singularities, but this time from
a more advanced standpoint. We are interested in understanding
whether there exists a holomorphic change of coordinates where the
term $R(y_1,y_2)$ in Equation~\ref{eq39} becomes {\it identically}
zero. Even though it is possible to obtain a formal conjugacy
between the two normal forms, in general, this application does
not converge on a neighborhood of the singularity. However, in
certain sectors of the neighborhood, the formal conjugacy is
indeed {\it summable}. This is basically the contents of the
Hukuara-Kimura-Matuda Theorem. So the next step is to study the
functions that ``glue'' together these sectors; i.e., the sector
changing diffeomorphisms. In the simplest case, there are two
diffeomorphisms that realize the two possible changes of sectors,
depending on the connected component of the domain of intersection
that is being considered. One of them is a translation and the
other one happens to be a diffeomorphism tangent to the identity.
The interesting issue here, is that these diffeomorphisms are not
unique. Only their conjugacy classes is canonic and can be used to
parameterize the moduli space of the saddle-nodes. This leads us
to a related topic, namely, the classification of diffeomorphisms
of the type $f(z)=z+z^2+\cdots$, following the work of S. Voronin
\cite{Vo}. The procedure to be employed is again based on
sectorial normalizations, whereas the normalizing maps will now be
constructed by means of the Measurable Riemann Theorem.

Up to this point we have only been dealing with simple
singularities. One might ask what is to be done about more
degenerate singularities. Indeed, in dimension~$2$ we do not
really have to worry about them, since there is an effective
method to reduce any singularity to a ``superposition'' of simple
ones. This is precisely what is done in Subsection~$4.2$, following
\cite {M-M}. Essentially, the idea is that by composing a finite
number of blow-up applications we reduce a singularity of higher
order to an arrangement of curves containing only simple
singularities. This is the contents of the Seidenberg Theorem.
Another reduction may yet be done in the case of simple
singularities: either it is reduced to a saddle-node singularity
or both $\lambda_1/\lambda_2$ and $\lambda_2/\lambda_1$ do not
belong to $\N$.

Finally Subsection~$4.3$ is devoted to Mattei-Moussu Theorem,
regarding the existence of holomorphic first integrals for
foliations. In their joint work (\cite{M-M}), J.-F. Mattei
and R. Moussu find necessary and sufficient conditions for
the existence of non-constant holomorphic functions
that are constant along the leaves of a foliation $\fol$ with an
isolated singularity. Moreover these conditions are of topological
nature. Their theorem is as follows.
\begin{teo}[Mattei-Moussu \cite{M-M}]
Consider the holomorphic foliation $\fol$ defined on $U\subset
\C^2$ with an isolated singularity at $(0,0)$. Suppose that it
satisfies the following conditions:
\begin{enumerate}
\item Only a {\it finite} number of leaves of $\fol$ accumulates
on $(0,0)$;
\item The leaves of $\fol$ are closed on $U\setminus \{(0,0)\}$.
\end{enumerate}
Then $\fol$ has holomorphic non-constant first integral
$f:U\rightarrow \C$.
\end{teo}

We have divided the proof of Mattei-Moussu Theorem into~$3$ parts
so as to make the exposition more transparent. The first step is
to show that under these assumptions the singularity can be
reduced to a superposition of singularities belonging to the
Siegel domain. This is done by studying the local holonomy of a
leaf with respect to a loop encircling each type of singularity
that might be contained in the Seidenberg tree of $\fol$. We reach
the conclusion that the only way we do not violate Conditions~$1$
and~$2$ is if the singularities are in the Siegel domain.

The next step is to show that the singularities in the Seidenberg
tree admit local first integrals. This is done by showing that
$\lambda_1/\lambda_2$ belong to $\Q_-$ and using the fact (also
due to J.-F. Mattei and R. Moussu) that if the holonomy associated
to a leaf of $\fol$ is linearizable, then the vector field related
to this foliation is linearizable as well.

Finally we extend the local first integrals in order to obtain a
global one. We show this in the case where all of the
singularities are reduced by a single blow-up and the general case
follows easily by induction.


\section{Normal Forms for Singularities of Holomorphic Foliations}

This section is devoted to the study of isolated singularities of
holomorphic foliations in dimension $2$.

Consider a holomorphic foliation $\fol$ with an isolated
singularity at $(0,0)$. As previously seen, there is a
holomorphic vector field $X$ associated to this foliation, which
is uniquely defined up to a multiplicative function $f$ such that
$f(0,0)\neq 0$.

The {\it eigenvalues} for the foliation $\fol$ at $(0,0)$ are
defined to be the eigenvalues $\lambda_1$, $\lambda_2$ associated
to the vector field $X$ at $(0,0)$, i.e. the eigenvalues of ${\rm Jac}X(0,0)$.
Notice however that the precise values of $\lambda_1, \, \lambda_2$ are not well
defined (except for the case where both are equal to zero) since the representative
of the foliation is defined up to a multiplicative unitary function. What is, in
fact, well defined for the foliation is the ratio $\lambda_1 / \lambda_2$ (assuming
$\lambda_2 \ne 0$). The ration is clearly invariant by the choice of vector field
associated to $\fol$. We have therefore three possibilities:
\begin{itemize}
\item[(a)] $\lambda_1 = \lambda_2 = 0$;
\item[(b)] $\lambda_1 = 0$, $\lambda_2 \ne 0$;
\item[(c)] $\lambda_1 \ne 0$, $\lambda_2 \ne 0$.
\end{itemize}
which are well-defined in the sense that the eigenvalues of two representatives of
$\fol$ belong to the same case (a),(b) or (c). A singularity is said to be \emph{simple}
if at least one of its eigenvalues is different from zero. If exactly one of its
eigenvalues is equal to zero, then the singularity is called a
\emph{saddle-node}.

The \emph{order} of a foliation $\fol$ at $(0,0)$ is the degree of
the first non-vanishing homogeneous component of the Taylor series
of $X$ based at $(0,0)$. Naturally, this order does not depend on
the chosen vector field $X$.

We will begin by studying simple singularities, taking into
account the problem of linearization. Afterwards the more general
case will be dealt with using the Seidenberg's Theorem.

Let $X, \, Y$ be two holomorphic vector fields defined in a neighborhood of
the origin, singular at this point. We say that $X$ is analytically (resp.
formally, $C^\infty$, $C^k$) conjugate to $Y$ if there  exists a holomorphic
(resp. formal, $C^\infty$, $C^k$) diffeomorphism $H$, such that $H(0)=0$ and
\[
Y=(DH)^{-1}(X \circ H).
\]
We say that $X$ and $Y$ are analytically (resp. formally, $C^\infty$, $C^k$)
equivalent if $X$ is analytically (resp. formally, $C^\infty$, $C^k$) conjugate
to $fY$, for some holomorphic function $f$ verifying $f(0) \ne 0$.

By analytic (resp. formal $C^{\infty}$, $C^k$) classification we mean a partition
of the set of holomorphic vector fields in classes whose elements are analytically
(resp. formally, $C^{\infty}$, $C^k$) conjugate to the others in the same class. A
vector field $X$ is said analytic (resp. formal, $C^{\infty}$, $C^k$) linearizable
if it belongs to the analytic (resp. formal, $C^{\infty}$, $C^k$) class of the vector
field $J^1_0X$ (i.e. the linear part of $X$).


\subsection{Vector Fields with Non-Zero Eigenvalues}\label{sectionvfnzeigenvalues}

Let us consider the ODE generated by a holomorphic vector field
$X$ with an isolated singularity at $(0,0)$. Suppose further that
its eigenvalues at $(0,0)$ are $\lambda_1$ and $\lambda_2$, both
different from \emph{zero}. In other words, we are considering the
foliation associated to the system of ODEs:
\begin{equation}
\begin{cases}
         \dot{x_1} = \lambda_1 x_1 + \varphi_1 (x_1 ,x_2)\\
         \dot{x_2} = \lambda_2 x_2 + \varphi_2 (x_1 ,x_2)
\end{cases} \, . \label{eq6}
\end{equation}
where $\varphi_1, \, \varphi_2$ have order at least~$2$. By $\dot{x}$ we mean
$dx/dT$. Note that we are assuming ${\rm Jac}X(0,0)$ diagonalizable. We now
wish to investigate the existence of a formal change of coordinates that
linearizes this system. That is, the existence of a formal map $H$ such that
$DH. X = J^1_{(0,0)}X \circ H$.

We shall adopt the the following notations. Let $Q = (q_1 ,q_2)$, $x^Q =x_1^{q_1} x_2^{q_2}$
and $\|Q\| = q_1 + q_2$. Suppose that, under the notations above, $\varphi_1, \, \varphi_2$
are written in the form
\[
\varphi_1 = \sum_{\| Q \| >1} \varphi_{1,Q} x^Q \;\; ; \;\; \varphi_2 = \sum_{\| Q \| >1} \varphi_{2,Q} x^Q \, .
\]
and consider the formal change of coordinates
\begin{eqnarray}
x_1 = y_1 + \zeta_1 (y_1 ,y_2) \; & {\rm where} & \;
\zeta_1 (y_1 ,y_2) = \sum_{\| Q \| >1} \zeta_{1,Q} y^Q \; \label{eq7} ,\\
x_2 = y_2 + \zeta_2 (y_1 ,y_2) \; & {\rm where} & \; \zeta_2 (y_1
,y_2) = \sum_{\| Q \| >1} \zeta_{2,Q} y^Q \; . \label{eq8}
\end{eqnarray}
Assume that, in these new coordinates, system (\ref{eq6}) is given by:
\begin{equation}
\begin{cases}
         \dot{y_1} = \lambda_1 y_1 + \psi_1 (y_1, y_2)\\
         \dot{y_2} = \lambda_2 y_2 + \psi_2 (y_1, y_2)
\end{cases}  \, . \label{eq9}
\end{equation}
Note that the right hand side corresponds to the linear part of $X$ if
and only if $\psi_1, \, \psi_2$ vanishes identically.

Substituting (\ref{eq7}) and (\ref{eq8}) in (\ref{eq6}), we obtain the following
relations:
\begin{eqnarray}
\sum_{\| Q\| >1} (\delta_{1,Q} \zeta_{1,Q} + \psi_{1,Q}) y^Q =
\varphi_1 (y_1+\zeta_1,y_2+\zeta_2) - \sum_{k=1,2} \frac{\partial
\zeta_1}
{\partial y_k} \psi_k \; , \label{eq10} \\
\sum_{\| Q\| >1} (\delta_{2,Q} \zeta_{2,Q} + \psi_{2,Q}) y^Q =
\varphi_2 (y_1+\zeta_1,y_2+\zeta_2) - \sum_{k=1,2} \frac{\partial
\zeta_2} {\partial y_k} \psi_k \; , \label{eq11}
\end{eqnarray}
where $\delta_{1,Q} = \lambda_1 q_1 + \lambda_2 q_2 - \lambda_1$
and $\delta_{2,Q} = \lambda_1 q_1 + \lambda_2 q_2 - \lambda_2$.
Notice that if $\delta_{i,Q}\neq 0$ ($i=1,2$) we can always find
$\zeta_{i,Q}$ such that $\psi_{i,Q}=0$. This can be seen by
induction on $\|Q\|$ and recalling that $\varphi_i$ has terms of order
at least $2$. In fact, this last implies that the coefficients
of $y^Q$ on the right hand side of Equations~(\ref{eq10})
and~(\ref{eq11}) depend only on terms $\zeta_{i,P}$ for $\|P\|<\|Q\|$.
Therefore $\zeta_{i,Q}$ depend on the coefficients of $\varphi_i$, on the
terms $\zeta_{i,P}$ with $\|P\|<\|Q\|$ and on $\delta_{i,Q}$ (which are
non-zero by assumption). More precisely, $\zeta_{i,Q}=
\frac{f}{\delta_{i,Q}}$, where $f$ is a function of certain
coefficients of $\varphi_i$ and $\zeta_{i,P}$ for $\|P\|<\|Q\|$. Thus if
$\delta_{i,Q}\neq 0$ for all $Q \in \N^2$ then there always exist a formal
change of coordinates linearizing $X$.

This settles the problem of finding a formal change of coordinates. We have
therefore that a vector field is formally conjugate to its linear part if
there are no \emph{resonances}.

\begin{defi}
Suppose that the origin is a singular point of a vector field $X$. Let $\dl=(\dl_1,\dl_2)$
be the vector constituted by the eigenvalues of ${\rm Jac}X(0,0)$. We say that the eigenvalues
are resonant if, for some $i$, there exists $I=(i_1, i_2) \in \N_0^2$ with
$\sum_{j=1}^2i_j \geq 2$ such that
\[
\dl_i=(I,\dl )=i_1\dl_1+i_2\dl_2.
\]
In this case the monomials $x^{I} \partial /\partial x_i$, where $x^I=x_1^{i_1} x_2^{i_2}$,
are said to be resonant. If $\dim \{m \in \Z^2:(m,\dl)=0\}=k$ then $X$ is said $k$-resonant.
\end{defi}

We have just seen that in the absence of resonances there exists a formal diffeomorphism
$H$ taking $X$ into its linear part. However, nothing guarantees that $H$, or equivalently
the series $\zeta_i (y_1 ,y_2)$ ($i=1,2$), do indeed converge. In fact this may not occur.

Before dealing with the problem of convergence, we shall introduce some useful notations.
Given a formal power series $\xi = \sum_{Q} \xi_Q x^Q$, we let $\overline{\xi} = \sum_{Q}
\| \xi_Q \| y^Q$. We also consider the series $\overline{\overline{\xi}}$, on one complex
variable $z$, obtained as $\overline{\overline{\xi}} = \sum_Q \| \xi_Q \| z^{\| Q \|}$.
Equivalently, $\overline{\overline{\xi}}(z) = \overline{\xi}(z,z)$. Given another series
$\varpi = \sum_Q \varpi_Q y^Q$, we say that $\varpi \prec \xi$ if and only if $\| \varpi_Q \|
\leq \| \xi_Q \|$ for all  $Q$.

The proof of convergence of the formal diffeomorphism $H$, under assumptions that will be
presented below, is based on the following result:

\begin{teo}[Cauchy Majorant Method]
For $i=1, 2$, let $\varphi_i$ be two holomorphic functions and $\zeta_i$ be two formal series
as in~$($\ref{eq7}$)$ and~$($\ref{eq8}$)$. Suppose there exists $\delta > 0$ such that
\[
\delta \overline{\zeta_i} \prec \overline{\varphi}_i (y_1 +
\overline{\zeta}_1 , y_2 +\overline{\zeta}_2)\; .
\]
Then the series $\zeta_i$ ($i=1,2$) converge, hence defining a
holomorphic change of coordinates. \label{t1}
\end{teo}

\begin{proof}
From $\delta \overline{\zeta_i} \prec \overline{\varphi}_i (y_1 +
\overline{\zeta}_1 , y_2 +\overline{\zeta}_2)$ we obtain directly that
\begin{equation}
\label{eq12} \llz_1+\llz_2 \prec
\frac{1}{\delta}(\overline{\varphi}_1(z + \llz_1 + \llz_2,z +
\llz_1 + \llz_2) + \overline{\varphi}_2(z + \llz_1 + \llz_2,z +
\llz_1 + \llz_2))\,.
\end{equation}
Since the series $\varphi_i$ are convergent by assumption, there exist
$a_0$, $a>0$ such that:
\begin{equation}
\label{eq13} \frac{1}{\delta}(\overline{\overline{\varphi}}_1 +
\overline{\overline{\varphi}}_2) \prec \frac{a_0 z^2}{1-az}\, .
\end{equation}
From (\ref{eq12}) and (\ref{eq13}) we obtain
\begin{eqnarray*}
\frac{1}{z} \left(\llz_1 + \llz_2 \right) & \prec & \frac{1}{z \delta}
\left(\overline{\overline{\varphi}}_1(z + \llz_1 + \llz_2) +
\overline{\overline{\varphi}}_2(z + \llz_1 + \llz_2)\right) \\
& \prec & \frac{1}{z}\Bigg(\frac{a_0(z + \llz_1 + \llz_2)^2}{1 - a
(z + \llz_1 + \llz_2)}\Bigg) \\
& \prec & \frac{a_0 z [\frac{1}{z}(z + \llz_1 + \llz_2)]^2} {1- az
(\frac{1}{z} (z + \llz_1 + \llz_2))}\; ,
\end{eqnarray*}

Denoting by $u=\sum_{i}u_i z^i$ the series $\frac{1}{z}(\llz_1 +
\llz_2)$ we have just proved that
\begin{equation}
\label{eq14} u \prec \frac{a_0 z (1+u)^2} {1 - a z (1+u)}\; .
\end{equation}
We shall now compare $u$ with the series $v=\sum_i v_i z^i$, that
is a solution of
\[
v=\frac{a_0 z(1+v)^2}{1-az(1+v)}\; .
\]
or, equivalently, a solution of
\[
\Big(\sum v_i z^i\Big)^2(-az-a_0z) + \Big(\sum v_iz^i\Big)(-az -2a_0z
+1)-a_0z=0\; ,
\]

The series $v$ is convergent and, in particular, $v_1=a_0$. Notice that for every
$i$, $v_i$ is a polynomial with positive coefficients in the
variables $v_1,\ldots,v_{i-1}$, denoted by $P_i(v_1,\ldots,v_{i-1})$.
We may choose $a_0>u_1$ and it follows from (\ref{eq14}) that
\[u_i\leq P_i(u_1,\ldots,u_{i-1})\; .\]

Now we show that $u \prec v$ by induction. Suppose that $u_j \leq
v_j$ for $j \leq i-1$. The fact that $P_i$ has positive
coefficients for every $i$ implies that
\[
u_i \leq P_i(u_1,\ldots,u_{i-1}) \leq
P_i(v_1,\ldots,v_{i-1})=v_i\; .
\]
Hence $u$ is convergent, and consequently so is are $\zeta_i$
($i=1,2$).
\end{proof}

We are now able to prove the convergence of $H$ under the conditions below:

\begin{teo}[Poincar\'{e} Linearization Theorem]
\label{PoincLintheo} Consider a system of differential equations as in $($\ref{eq6}$)$. Assume that
$\lambda_1 \lambda_2 \neq 0$ and that $\lambda_1 /\lambda_2 \not\in \R_-$. Suppose also that neither
$\lambda_1 /\lambda_2$ nor $\lambda_2 /\lambda_1$ belongs to $\N$. Then there is an analytic change
of coordinates in which the system becomes linear.
\end{teo}

\begin{proof}
We first note that neither $\delta_{1,Q}$ nor $\delta_{2,Q}$ vanish independently
of $Q$. Indeed if $\delta_{1,Q}$ vanishes then
\[
\frac{\lambda_1}{\lambda_2}=\frac{q_2}{(1-q_1)}
\]
So that whenever $q_2=0$, $\lambda_1=0$; when $q_1=0$, $\lambda_1/\lambda_2 \in \N$; and if
$q_1>1$ then $\lambda_1/\lambda_2 \in \R_-$. A similar argument holds for $\delta_{2,Q}$.
Therefore, $\zeta_{i,Q}$
can always be found in such a way that $\psi_{i,Q}=0$. This means that there exists a formal
change of coordinates that linearizes system (\ref{eq6}). Now it must be shown
that this change of coordinates is indeed convergent.

In fact our assumptions on $\lambda_1$ and $\lambda_2$ imply that
there exists $\delta >0$ such that $\inf_{Q} \{ |\delta_{1,Q}| , |\delta_{2,Q}| \} \geq
\delta$. This is a simple consequence of the convex separability theorem.
From equations (\ref{eq10}) and (\ref{eq11}), we obtain for $j=1,2$
\begin{eqnarray*}
\delta \overline{\zeta}_j  & \prec & \sum_{Q} \delta_{j,Q}
\| \zeta_{j,Q} \| y^Q + \overline{\psi}_j \\
& \prec & \overline{\varphi}_j (y_1 + \overline{\zeta}_1, y_2
+\overline{\zeta}_2) + \sum_{k=1,2} \frac{\partial
\overline{\zeta}_j}{\partial y_k} \overline{\psi}_k \, .
\end{eqnarray*}
Since $\overline{\psi}_k =0$, we have
\[
\delta \overline{\zeta_j} \prec \overline{\varphi}_j (y_1 +
\overline{\zeta}_1, y_2 +\overline{\zeta}_2) \, .
\]
Finally, Theorem~\ref{t1} implies that this change of coordinates is
convergent.
\end{proof}

Assume now that $\dl_1 /\dl_2 \in \N$. Although not necessarily linearizable the vector field
has a very simple normal form in that case.

\begin{lema}
Suppose that $\dl_1 /\dl_2 \in \N$. Then there is an analytic change of coordinates where the original
system is given $($in terms of vector fields$)$ by
\[
X = ( \dl_1 y_1 + a y_2^n) \partial /\partial y_1 + \dl_2 y_2 \partial /\partial y_2 \, .
\]
where $n$ is such that $\dl_1 = n\dl_2$ and $a \in \C$.
\end{lema}

\begin{proof}
We first note that, under the assumptions above, $\inf_{Q} \{ \delta_{1,Q} , \delta_{2,Q} \} \geq 1$
for all $Q$ with exception to $Q = (0,n)$. In fact, the vector field hs a unique resonant relation
($\dl_1 = n \dl_2$ for some $n \in \N$). This implies that the coefficient of $\zeta_{1,(0,n)}$
in Equation~(\ref{eq10}) (which is given by $\delta_{1,(0,n)} = n\dl_2 - \dl_1$) vanishes. Since no
more resonance relations appear, the vector field associated to the differential equation is then
formally conjugated by a diffeomorphism $H$ to
\[
X = ( n\dl_2 y_1 + a y_2^n) \partial /\partial y_1 + \dl_2 y_2 \partial /\partial y_2
\]
for some $a \in \C$. The fact that  $\inf_{Q} \{ |\delta_{1,Q}| , |\delta_{2,Q}| \} \geq 1$ for all
$Q$ distinct from $(0,n)$ ensures that we can follow the proof of Theorem~\ref{PoincLintheo} in
order to prove the convergence of $H$.
\end{proof}

Now we shall obtain a characterization of singular holomorphic foliations in the case where
$\lambda_1/\lambda_2 \in \R_-$.

\begin{lema}\label{Siegel}
If $\lambda_1 /\lambda_2 \in \R_-$, then there is an analytic change of coordinates
where the vector field associated to the original system is given by
\[
X = \lambda_1 y_1[1 + (h.o.t.)] \partial /\partial y_1 + \lambda_2
y_2 [1 + (h.o.t.)] \partial /\partial y_2 \, .
\]
In particular such vector field has two smooth transverse separatrizes.
\end{lema}

\begin{proof}
It will be shown that there exists a convergent change of coordinates which
allows us to suppose that $\varphi_1$ is divisible by the first variable and
that $\varphi_2$ is divisible by the second one, where $\varphi_1$ and $\varphi_2$
are as in (\ref{eq6}).

As before, we consider the change of coordinates $x_1 = y_1 + \zeta_1 (y_1 ,y_2)$
and $x_2 = y_2 + \zeta_2 (y_1 ,y_2)$ but setting:
\[
\left\{ \begin{array}{ll}
         \psi_{1,Q} = 0 & {\rm when} \; \; q_1=0\;\;{\rm or}\;\;q_2=0  \\
         \zeta_{1,Q} =0 & {\rm when} \; \; q_1\neq 0\;\;{\rm and}\;\;q_2\neq 0
        \end{array} \right.
\; \; \; \; {\rm and} \; \; \; \; \left\{ \begin{array}{ll}
         \psi_{2,Q} = 0 & {\rm when} \; \;q_1=0\;\;{\rm or}\;\;q_2=0  \\
         \zeta_{2,Q} =0 & {\rm when} \; \; q_1\neq 0\;\;{\rm and}\;\;q_2\neq 0
        \end{array} \right.
\]
If this change of coordinates is indeed convergent then, in these appropriate
coordinates, $\{y_1=0\}$ and $\{y_2=0\}$ are invariant which is
the contents of the lemma. So we must analyze the expressions of $\delta_{1,Q}$
only in the case where $q_1=0$ or $q_2=0$. After all, when $q_1\neq 0$ and
$q_2\neq 0$, $\zeta_{1,Q}=0$, so that these terms do not count in
the series $\overline{\zeta}_1$. Suppose that $q_1 = 0$. In this situation, we have
\begin{eqnarray*}
\delta_{1,Q} & = & \lambda_2q_2 -\lambda_1 \\
\Rightarrow \frac{\delta_{1,Q}}{\lambda_2}& = & q_2 -
\frac{\lambda_1}{\lambda_2}
>\varepsilon_1
\end{eqnarray*}
for some $\varepsilon_1 > 0$. This guarantees that $|\delta_{1,Q}|$ is bounded
from below by a constant $c > 0$. A similar argument implies that $|\delta_{2,Q}|$
is bounded from below either  by a positive constant. More
precisely, $|\delta_{2,Q}| > c > 0$ for all $Q$ as above with $\|Q\| \geq 2$, i.e.
such that $q_2>1$. The case $q_2 = 0$ is analogous.

Therefore, there exists $\delta >0$ such that $\inf_{Q} \{ |\delta_{1,Q}| ,
|\delta_{2,Q}| \} \geq \delta$, where the $\inf$ is taken between $Q$ such that
$q_1 = 0$ or $q_2 = 0$. Using the above mentioned fact along with Equations~(\ref{eq10})
and~(\ref{eq11}), we have:
\[
\delta \overline{\zeta}_1 \prec \sum_{Q} \delta_{1,Q}
\| \zeta_{1,Q} \| y^Q  \prec \overline{\varphi}_1 (y_1 +
\overline{\zeta}_1, y_2 +\overline{\zeta}_2) + \frac{\partial
\overline{\zeta}_1}{\partial y_1} \overline{\psi}_1 +
\frac{\partial \overline{\zeta}_1}{\partial y_2} \overline{\psi}_2
\, .
\]

By construction, the non-zero coefficients in $\overline{\zeta}_1$
are related only to monomials that are powers of $y_1$ or powers
of $y_2$, i.e. there are no monomials that mix the two variables
$y_1$ and $y_2$. On the other hand, all the non-zero terms that
enter $\overline{\psi}_1$ and $\overline{\psi}_2$ are such that
$q_1=0$ and $q_2=0$. So that the following stronger estimate
holds:
\[
\delta\overline{\zeta}_1 \prec \overline{\varphi}_1
(y_1 + \overline{\zeta}_1, y_2 +\overline{\zeta}_2)\; .
\]
A similar argument implies
\[
\delta\overline{\zeta}_2 \prec \overline{\varphi}_2
(y_1 + \overline{\zeta}_1, y_2 +\overline{\zeta}_2)
\]
and the application of Theorem~\ref{t1} guarantees the convergence
of the series
\end{proof}


\subsection{Elementary Aspects of Saddle-Node Singularities}\label{sectionsaddlenode}


\subsubsection{Dulac Normal Form and Consequences}

We shall now consider the case of \emph{saddle-nodes}, i.e. the case where $\lambda_1 \ne 0$
and $\lambda_2=0$. Obviously we can assume without loss of generality that $\dl_1 = 1$ and
$\dl_2 = 0$. We will see by successive appropriate changes of coordinates that, in this case,
the system of ODEs (or the $1$-form) that induces the foliation has a canonic
representation. This is the contents of the following result.

\begin{teo}[Dulac]
Let $\fol$ be a foliation which defines a saddle-node at $(0,0)$
in $\C^2$. Then, in appropriate coordinates $(y_1,y_2)$, $\fol$ is
given by the holomorphic $1$-form
\[
\omega = [y_1 (1 + \lambda y_2^p) + y_2 R(y_1,y_2)] \, dy_2 \; - \;
y_2^{p+1} \, dy_1 \, .
\]
for some $p \in \N$.
\end{teo}

\begin{proof}
Assume that
\begin{equation}\label{eqsd}
\begin{cases}
         \dot{x_1} = x_1 + \varphi_1 (x_1 ,x_2)\\
         \dot{x_2} = \varphi_2 (x_1 ,x_2)
\end{cases} \, ,
\end{equation}
where $\varphi_1, \, \varphi_2$ have order at least~$2$, is a representative of
$\fol$. First of all we will prove that there exists a change of coordinates
$x_1 = y_1 + \zeta_1 (y_1 ,y_2)$, $x_2 = y_2 + \zeta_2 (y_1 ,y_2)$ in which the
vector field associated to the differential equation above can be written in the
form:
\[
[y_1 +y_2R(y_1,y_2)] \partial /\partial y_1+ y_2 \phi(y_1,y_2) \partial /\partial y_2, \, .
\]
where $R(0,0) = \phi(0,0) = 0$.

Let us first note that, in that case, $\delta_{1,Q} = q_1 - 1$ whereas $\delta_{2,Q} = q_1$.
We have that $\delta_{1,Q} = 0$ if and only if $q_1 = 1$ while $\delta_{2,Q} = 0$ if
and only if $q_1 = 0$. This implies that we can always solve the correspondent to
Equation~(\ref{eq10}) (resp.~(\ref{eq11})) in order to $\zeta_{1,Q}$ (resp. $\zeta_{2,Q}$)
for $q_1 \ne 1$ (resp. $q_1 \ne 0$). So, let us set
\[
\left\{ \begin{array}{ll}
         \psi_{1,Q} = 0 & {\rm whenever} \; \; q_2 = 0 \\
         \zeta_{1,Q} =0 & {\rm whenever} \; \; q_2 \ne 0
        \end{array} \right.
\; \; \; \; {\rm and} \; \; \; \; \left\{ \begin{array}{ll}
         \psi_{2,Q} = 0 & {\rm whenever} \; \; q_2 = 0 \\
         \zeta_{2,Q} =0 & {\rm whenever} \; \; q_2 \ne 0
        \end{array} \right.
\]
for each index $Q= (q_1, q_2)$. Since $\| Q \| \geq 2$, it follows that
$\delta_{1,Q} = q_1 -1 \ne 0$ (resp. $\delta_{2,Q} = q_1 \ne 0$) whenever
$\zeta_{1,Q} \ne 0$ (resp. $\zeta_{2,Q} \neq 0$). This proves  that the two vector
fields are formally conjugate. In order to prove the convergence of the change of
coordinates we shall also note that $\delta_{1,Q} = q_1 -1 \geq 1$ (resp. $\delta_{2,Q}
= q_1 \geq 2$) whenever $\zeta_{1,Q} \ne 0$ (resp. $\zeta_{2,Q} \neq 0$). Thus
we have
\[
\delta\overline{\zeta}_1 \prec \overline{\varphi}_1 (y_1 +
\overline{\zeta}_1, y_2 +\overline{\zeta}_2) + \frac{\partial
\overline{\zeta}_1}{\partial y_1} \overline{\psi}_1 +
\frac{\partial \overline{\zeta}_1}{\partial y_2} \overline{\psi}_2
\]
for $0 < \delta < 1$.

Notice that $\zeta_{1,Q}=0$ whenever $q_2\neq 0$ then
$\overline{\zeta}_1$ depends only on $y_1$, so that
$\frac{\partial \overline{\zeta}_1}{\partial y_2}=0$. In
particular, the non-zero coefficients of the series
$\overline{\zeta}_1$ are such that $q_2=0$, then it follows from
the above change of coordinates that all the monomials entering
$\frac{\partial \overline{\zeta}_1}{\partial y_1} \overline{\psi}_1$
depend on $y_2$. Therefore these monomials do not appear in the
series of $\zeta_1$ and we conclude the stronger estimate
\[
\delta\overline{\zeta}_1 \prec \overline{\varphi}_1
(y_1 + \overline{\zeta}_1, y_2 +\overline{\zeta}_2)\; .
\]
A similar argument implies that
\[
\delta\overline{\zeta}_2 \prec \overline{\varphi}_2
(y_1 + \overline{\zeta}_1, y_2 +\overline{\zeta}_2)\; .
\]

The convergence of the coordinate change follows by applying the
Cauchy Majorant Method. This allows us to suppose that $\varphi_1$
and $\varphi_2$ are divisible by $x_2$. In other words, the
original system of ODEs is given by:
\[
X(x_1,x_2)=[x_1 +x_2R(x_1,x_2)]\frac{\partial}{\partial x_1}+ x_2 \phi(x_1,x_2)\frac{\partial}{\partial x_2}
\]
as we intended to prove.

Now, set $A(x_1,x_2)=x_1 +x_2R(x_1,x_2)$ and $B(x_1,x_2)=x_2 \phi(x_1,x_2)$.
The set $\{A=0\}\cap\{B=0\}$ is reduced to the origin $(0,0)$. The ideal
associated to the point $(0,0)$ is therefore
maximal and generated by $x_1$ and $x_2$. It follows from the
appropriate version of Hilbert's Nullstellensatz that this maximal
ideal is the radical of the ideal generated by $A$ and $B$. In
particular, there is $p+1\geq 2$ such that $x_2^{p+1}$ belongs to
the ideal generated by $A$ and $B$ (since $x_2$ itself cannot
belong to this ideal). Next we expand both $A$ and $B$ in terms of
$x_2$, i.e. we set $A=a_0(x_1)+\sum_{i=1}^\infty a_i(x_1)x_2^i$
and $B=\sum_{i=1}^\infty b_i(x_1)x_2^i$. The division of $B$ by
$A$ in the ring $C\{x_1\}$ then gives us
\[
B=AQ+x_2^{p+1} U(x_2)
\]
being $Q = 0$ when $\{x_2 = 0\}$.

Indeed, $a^{\prime}_0(x_1)=1$ so that the rest does not depend on $x_1$.
Besides $U(0)\neq 0$ since $p+1$ is the smallest positive power of
$x_2$ belonging in the ideal generated by $A$ and $B$. Now
consider the vector field
\[
Y = \frac{1}{U}\frac{\partial}{\partial x_1} - \frac{Q}{U}
\frac{\partial}{\partial x_2} \, ,
\]
which satisfies $Y(0,0)\neq (0,0)$ since $U(0,0)=U(0)\neq 0$. It
follows from the fact that $Q=0$ when $\{x_2=0\}$) that $\{x_2=0\}$ is a
solution of $Y$. By the Flow Box Theorem, there exist coordinates
$(z_1,z_2)$ such that the vector field becomes
\[
Y=\frac{\partial}{\partial z_1}\;\;\;\;\mbox{and}\;\;\;z_2=x_2\, ,
\]
and the result follows.
\end{proof}

Consider a foliation $\fol$ as in the previous theorem, i.e. defining a saddle-node
singularity. Hence, there exist appropriate coordinates where the associated vector
field is given by the normal form:
\[
X = [y_1 (1 + \lambda y_2^p) + y_2 R(y_1,y_2)] \frac{\partial}{\partial y_1}+y_2^{p+1}
\frac{\partial}{\partial y_2}\, .
\]
In particular we can see that $\fol$ admits a separatrix through the origin which
is given, in the coordinates above, by $\{y_2 = 0\}$. Consider a loop on this separatrix
encircling the saddle-node singularity and let $\Sigma$ be a transverse section to
the separatrix passing through a point of the loop. We shall now compute the holonomy
$h(z)$, where $z$ is a local coordinate on $\Sigma$. This simple
calculation  will prove to be quite useful in the next sections.
For the time being, it already provides a geometric interpretation
of the number $p+1$ appearing in the normal form above.

\begin{lema}\label{holvariedadefraca}
The holonomy associated to the separatrix $\{y_2=0\}$ is given by $h(z) = z + z^{p+1} + \cdots$.
\end{lema}

{\it Proof}. In order to prove this result we proceed as follows. Set
$y_1(t)=re^{2 \pi i t}$. Thus,
\begin{eqnarray}\label{eq34}
\frac{dy_2}{dt} &=& \frac{dy_2}{dy_1} \frac{dy_1}{dt} \nonumber\\
&=& \frac{y_2^{p+1}}{y_1 (1 + \lambda y_2^p) + y_2 R(y_1,y_2)} 2\pi i re^{2 \pi i t} \nonumber\\
&=& \frac{y_2^{p+1}}{re^{2 \pi i t}[1 + \lambda y_2^p + y_2 Q(re^{2 \pi i t},y_2)]}2\pi i re^{2 \pi i t}\nonumber\\
&=& 2\pi i y_2^{p+1}(t)(1+ {\rm h.o.t.})
\end{eqnarray}
where $Q$ is holomorphic relatively to $y_2$. Denote $y_2(t)=\sum_{k\geq 1} a_k(t)z^k$
and with initial data $y_2(0)=z$, so that
\begin{equation}\label{eq35}
\frac{dy_2}{dt}=\sum_{k\geq 1} a^{\prime}_k(t)z^k \, .
\end{equation}
By comparing the expressions obtained for $dy_2/dt$ on (\ref{eq34}) and on (\ref{eq35})
and taking account that
\begin{equation}\label{eq36}
y_2^{p+1}(t) = \left(\sum_{k \geq 1} a_k(t) z^k \right)^{p+1}
\end{equation}
we see that $a^{\prime}_k(t)=0$ for $k\leq p$, i.e. the functions $a_k(t)$ are all constants
for $k \leq p$. Since we have set $y_2(0)=z$, we have that $a_1(0)= 1$, $a_2(0) = \cdots = a_p(0) = 0$.

Now we compare the term $k=p+1$. Using~(\ref{eq36}) we obtain:
\[
a^{\prime}_{p+1}(t)=2\pi i a_1^{p+1}(t) \, ,
\]
So that $a_{p+1}(t)= 2\pi i t$ since $a_1(t)=1$ for all $t$.

Since $h(z) = y_2(1)$ we conclude that $h(z) = z + 2\pi i z^{p+1} + \cdots$. By
performing a change of coordinates we obtain the desired result, i.e. the holonomy can
be written in the form $h(z) = z + z^{p+1} + \cdots$.
\qed


\subsubsection{Fatou Coordinates and the Leau Flower}

In view of Lemma~\ref{holvariedadefraca}, we are naturally led to investigate
the topological dynamics of diffeomorphisms which are tangent to
the identity. In particular, we would like to understand the role
of the multiplicity ``$p+1$'' on the topological dynamics of these
diffeomorphisms. Here we follow closely the approach given in
\cite{Car-G}.

So let us first analyze applications of the form
$f(z)=z+z^{p+1}+\cdots$, in the prototypical case where $p=1$. To
begin with, let us apply a holomorphic change of coordinates,
$A_1(z)=-1/z$, taking $0$ to $\infty$. In these new coordinates
$f$ becomes
\[
g(z)=z+1+b/z+\cdots
\]
Fix $c \in \R^+$ sufficiently large so that $|g(z) - (z+1)| < 1/2$ for all $z \in \C$ such that
$|z| > c$ and let $R_c=\{z\in \C\, : \; {\rm Re}(z)> c\}$. We will first prove that $g$ is analytically
conjugate to $z \mapsto z+1$ on $R_c$.

First of all we note that $g(R_1)\subset R_1$. In fact, since ${\rm Re}(z) > -|z|$ we have
\begin{equation}\label{estimateg}
\frac{1}{2} < 1 + {\rm Re}\left(\frac{b}{z} + O\left(\frac{1}{z^2}\right)\right) < \frac{3}{2}
\end{equation}
and therefore
\[
{\rm Re}(g(z)) = {\rm Re}(z) + 1 + {\rm Re}\left(\frac{b}{z} + O\left(\frac{1}{z^2}\right)\right) > {\rm Re}(z) > c \, .
\]
This implies that the map $\varphi_n(z)=g^n(z)-n-b\log n$ is well defined on $R_c$, where $g^n=g\circ\cdots\circ g$,
$n$ times. If $\varphi_n$ converges to a holomorphic function $\varphi$ then $g$ is holomorphically conjugate to the
translation $T(z)=z+1$. Indeed, one has
\[
\lim_{n \rightarrow \infty} \varphi_n(g(z)) = \lim_{n \rightarrow \infty}
[\varphi_{n+1}(z) + 1 + b\log(1+1/n)] = \lim_{n \rightarrow \infty}
\varphi_{n+1}(z) + 1 \, .
\]
The convergence of $\varphi_n$ is the contents of the next lemma.

\begin{lema}
The sequence $\varphi_n$ converges to a conformal function
$\varphi$.
\end{lema}

\begin{proof}
We note that $g^n(z) = z + n + O(1/n)$. The estimate (\ref{estimateg}) implies that
\[
\frac{n}{2}\leq |g^n(z)|\leq |z|+2n \, .
\]
This can be easily verified by induction on $n$. In fact for the upper bound estimate we have
\[
|g(z)| \leq |z| + |1 + O(1/z)| \leq |z| + 2 \, .
\]
Assuming that the estimate is valid for $n$ then
\[
|g^{n+1}(z)| = |g^n(g(z))| \leq |g(z)| + 2n \leq |z| + 2(n+1)
\]
Relatively to the lower estimate we have that
\[
|g^n(z)| > {\rm Re}(g^n(z)) >{\rm Re}(z) + \frac{n}{2} > \frac{n}{2} \, .
\]

We have therefore that $g^{k+1}(z)=g^k(z)+1+\frac{b}{g^k(z)}+ O(1/k^2)$. Thus we obtain that
\[
\varphi_{k+1}-\varphi_k(z)=b[\log k-\log(k+1)]+\frac{b}{g^k(z)}+
O(1/k^2)=O(1/k) \, .
\]
Hence, for $z\in R_1$ the following estimate holds
\[
|\varphi_n(z) - z| \leq|\varphi_1(z) - z| + \sum_{k=1}^{n-1} |\varphi_{k+1}(z) -
\varphi_k(z)| = O(\log n) \, .
\]
It remains to prove that $\varphi_n$ is uniformly convergent on the compact subsets of $R_c$. We have
the estimates
\begin{eqnarray*}
\varphi_{n+1}(z) - \varphi_n(z) &=& b\log n - b\log(n+1) + g^{n+1}(z) - g^n(z) - 1\\
&=& -\frac{b}{n} + \frac{b}{g^n(z)} \left(= O \left(\frac{1}{n^2} \right) \right)\\
&=& b \left[\frac{1}{n + b\log n + \varphi_n(z)} - \frac{1}{n} \right] + O(\frac{1}{n^2})\\
&=& \frac{1}{n^2} O\left( \left|b\log n + \varphi_n(z) \right| \right) +
O \left(\frac{1}{n^2} \right)\\
&=& O \left( \frac{\log n}{n^2} \right) \, ,
\end{eqnarray*}
so that $\sum|\varphi_{n+1}(z)-\varphi_n(z)|<\infty$. Since all the $\varphi_n$ are conformal,
so is the uniform limit $\varphi$.
\end{proof}

Next we note that we can extend $\varphi$ analytically to any domain $\Omega$,
contained in the domain of $g$, verifying $g(\Omega)\subset \Omega$ and such that
${\rm Re}(g^n(z))$ tends to $\infty$, for $z\in \Omega$. In fact we can construct
one such invariant domain $\Omega$ with smooth boundary, as is shown in Figure~\ref{f6}.
We have therefore that $f(z)=z+z^2+\cdots$ is conjugate to the translation $T(z)=z+1$
on the cardioid-shaped region, $A_1^{-1}(\Omega)$, as shown in Figure~\ref{f7}. The set
$A_1^{-1}(\Omega)$ is therefore what we call an attracting petal centered at an attracting
direction. Basically this means that $A_1^{-1}(\Omega)$ is an invariant set whose orbit
converges to the origin tangentially to the direction $v$.

\begin{figure}[!htb]
\begin{minipage}[b]{0.45\linewidth}
\centering
\includegraphics[scale=0.9]{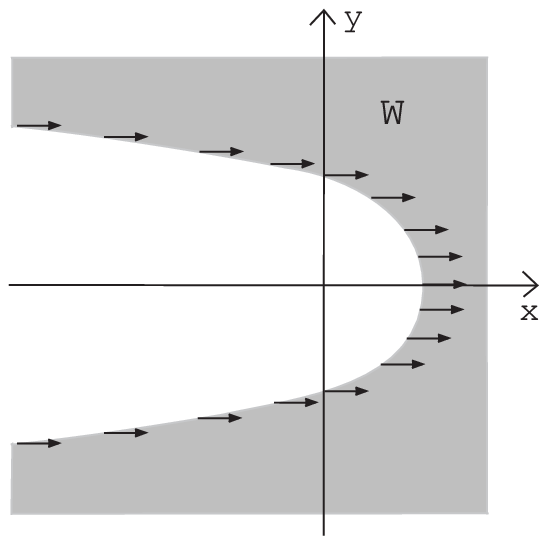}
\caption{Invariant domain by $g$} \label{f6}
\end{minipage} \hfill
\begin{minipage}[b]{0.45\linewidth}
\centering
\includegraphics[scale=0.9]{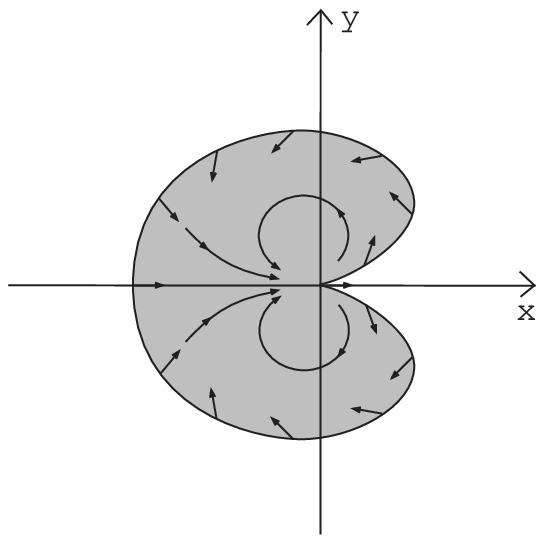}
\caption{Cardioid-shaped Dynamics of~$f$} \label{f7}
\end{minipage}
\end{figure}

\begin{defi}
Let $f \in {\rm Diff} (\C, 0)$ be such that $f(z) = z + az^{p+1} + \cdots$, where $a \ne 0$. An attracting
petal for $f$ centered at an attracting direction $v$ is a simply connected open set $P$ such that
\begin{itemize}
\item[a)] $0 \in \partial P$
\item[b)] $f(P) \subseteq P$
\item[c)] $\lim_{n \rightarrow +\infty} f^{(n)} = 0$ with $\lim_{n \rightarrow +\infty} \frac{f^n(z)}{|f^n(z)|} = v$ for all $z \in P$
\end{itemize}
A repelling petal centered at a repelling direction $v$ is an attracting petal for $f^{-1}$ centered at
the attracting direction $v$ for $f^{-1}$.
\end{defi}

We should note that an attracting (resp. repelling) direction is an element $v$ of $\C$ such that
$av^k/|a| \in \R_-$ (resp. $av^k/|a| \in \R_+$). Moreover, the union of the attracting
petals and the repelling ones constitutes a neighborhood of the origin.

In the case that $p = 1$ we have exactly one attracting petal (resp. direction) and one repelling petal
(resp. direction). The general case where $p\geq 2$ can be treated in a similar way. In fact, this case
can be reduced to previous one. Let $f_p(z) = z + z^{p+1} + \cdots$. Up to conjugation by an homothety,
we can suppose that $f_p(z) = z + \frac{1}{p} z^{p+1} + \cdots$. Now conjugating by $A_p(z) = -z^{1/p}$
we obtain
\[
A_p^{-1}\circ f_p \circ A_p(z) = z \left(1 + \frac{1}{p} z + \cdots \right)^p = z + z^2 + \cdots \, .
\]
Thus we are reduced to the case $p = 1$, which has been analyzed before. We note that the sectors
$|\arg z - 2k\pi/p| <\pi/p$ are mapped conformally, by $A_p^{-1}$, onto the plane minus the
negative real axis. This implies that in the case $p \geq 2$ we have essentially a ramification
of order $p$ of the previous case. The mapping $f_p$ has therefore $p$ {\it attracting petals}
$P_k$. The petals are (invariant domains) bounded by piecewise analytic Jordan curves ($c_1,
\ldots, c_k$) and they are symmetric relatively to the rays $\arg z= (2k\pi)/p$. At the origin,
$c_i$ has two tangents $\arg z= (2i \pm 1)\pi/p$. The final picture is summarized by the theorem below.

\begin{teo}[Flower Theorem]\label{t14}
Let $f \in {\rm Diff} (\C, 0)$ be given by $f(z) = z + az^{p+1} + \cdots$, where $a \ne 0$. Denote
by $v_1^+, \ldots, v_p^+$ the attracting directions and by $v_1^-, \ldots, v_p^-$ the repelling
ones. Assume that they are ordered by the following rule: starting at $v_i^+$ and moving in the
counterclockwise direction we first meet $v_i^-$ and then $v_{i+1}^+$. Then
\begin{itemize}
\item[a)] For each $v_i^+$ (resp. $v_i^-$) there exists an attracting (resp. repelling) petal $P_i^+$
(resp. $P_i^-$) centered at $v_i^+$ (resp. $v_i^-$)
\item[b)] The union of all attracting and repelling petals constitutes a neighborhood of the origin
\item[c)] $P_i^+ \cup P_j^+ = \empty$ and $P_i^- \cup P_j^- = \empty$ for $i \ne j$
\item[d)] The diffeomorphism $f$ is holomorphically conjugated to the translation map $T(z) = z+1$ on
each attracting petal.
\end{itemize}
\end{teo}

We recall that a domain $V$ is called a {\it Leau domain} if $f$ is conjugate to the
translation $T(z)=z+1$ on $V$ and also if the sequence $f^n$ converges to a point on
$\partial V$. The cardioid-shaped region illustrated on Figure~\ref{f7} is an example
of a Leau domain.


\subsection{Some Normal Forms in Higher Dimensions}

The Cauchy Majorant Method was essential to prove the convergence of the formal conjugating
diffeomorphisms for generic vector fields in dimension~$2$. We begin this section
noticing that the generalization of this result to higher dimension is straightforward. We
will not prove it.

\begin{teo}[Cauchy Majorant Method in Several Variables]\label{CMMhigherdim}
Let $\phi_i$, $i=1, \ldots, n$, be~$n$ holomorphic functions with trivial linear part and
let $\zeta_i$, $i=1, \ldots, n$, be~$n$ formal series. Consider the change of coordinates
\begin{equation}\label{eq16}
x_i = y_i + \zeta_1 (y_1 ,\ldots,y_n) \, \, \, {\rm where} \, \, \,
\zeta_i (y_1 ,\ldots, y_n) = \sum_{\| Q \| >1} \zeta_{i,Q} y^Q \, \, \, \, \, (i=1, \ldots, n)
\end{equation}
and assume the existence of $\delta>0$ such that
\[
\delta\overline{\zeta}_i \prec \overline{\varphi}_i
(y_1 + \overline{\zeta}_1,\ldots, y_n +\overline{\zeta}_n)
\]
for all $i=1,\ldots, n$. Then the series of $\zeta_i$ converges
and hence defines a holomorphic change of coordinates. \qed
\end{teo}

We shall obtain the correspondent to Theorem~\ref{PoincLintheo} for several variables.
Before that let us introduce some definitions.

\begin{defi}
Suppose that the origin is a singular point of a vector field $X$ on $(\C^n, 0)$. Let
$\dl=(\dl_1,\ldots, \dl_n)$ be the vector of eigenvalues of ${\rm Jac}X(0)$. We say that
the eigenvalues are resonant if, for some $i \in \{1, \ldots, n\}$, there exists
$I=(i_1, \ldots, i_n) \in \N_0^n$ with $\sum_{j=1}^n i_j \geq 2$ such that
\[
\dl_i=(I,\dl )=i_1\dl_1 + \ldots + i_n\dl_n
\]
i.e. if at least one of the eigenvalues can be written as a non-trivial positive linear
combination of all of them. In this case the monomials $x^{I} \partial /\partial x_i$ are
said to be resonant. Like in the two dimensional case, if $\dim \{m \in \Z^n : (m,\dl) = 0\}
= k$ then $X$ is said to be $k$-resonant.
\end{defi}

\begin{defi}
We say that a $n$-tuple $(\lambda_1 ,\ldots ,\lambda_n) \in \C^n$ belongs to the
{\it Poincar\'e's domain} if the convex hull of the complex numbers $\lambda_1 ,\ldots ,\lambda_n$,
i.e. if the set $\{ z \in \C \, : \, t_1\lambda_1+\cdots + t_n\lambda_n=z \, , \,
t_1+\cdots+t_n=1\}$, does not contain the origin $0 \in \C$. Otherwise we say that
$(\dl_1, \ldots, \dl_n)$ belongs to the {\it Siegel domain}.

A linear vector field $X$ defined on $\C^n$ is said to be of {\it Poincar\'e-type}
(resp. {\it Siegel-type}) if its spectrum is in the Poincar\'e's domain (resp.
Siegel's domain).
\end{defi}

To belong to the Poincar\'e domain is equivalent to the existence of a straight line
through the origin such that all the eigenvalues $\dl_1, \ldots, \dl_n$ belong to a
same half-plane defined by this straight line.

\begin{teo}[Poincar\'e Linearization Theorem]\label{PLT}
Let $X$ be a holomorphic vector field defined in a neighborhood of the origin of
$\C^n$ such that its linear part is of Poincar\'e-type. Let $\lambda_1, \ldots,
\lambda_n$ be the eigenvalues of ${\rm Jac}X(0)$. Assume that there there exists no
resonance relation. Then there exists a holomorphic change of coordinates that linearizes
$X$.
\end{teo}

\begin{proof}
The idea of the proof is the same as in the case of
dimension $2$. Substituting Equations~(\ref{eq16}) in the differential equation
associated to $X$, we obtain the following relations
\[
\sum_{\| Q\| > 1} (\delta_{i,Q} \zeta_{i,Q} + \psi_{i,Q}) y^Q =
\varphi_i (y_1 + \zeta_1,\ldots, y_n + \zeta_n) - \sum_{k=1}^n
\frac{\partial \zeta_i} {\partial y_k} \psi_k \; ,
\]
for $i=1, \ldots, n$, where $y^Q=y_1^{q_1}\ldots y_n^{q_n}$ and $\delta_{i,Q} =
q_1 \lambda_1 + \cdots + q_n \lambda_n  - \lambda_i$, with $q_i\in \N$. Due to
the non-resonant assumption, we have that $\delta_{i,Q}\neq 0$ for all $i=1,\ldots, n$.
Thus the equations above are solvable in order to $\zeta_{i,Q}$, for all $i$ and $Q$.
and, consequently, $X$ is formally linearizable. In order to prove the convergent
it is sufficient to prove that $\{|\delta_{i,Q}|\}_{i, Q}$ is bounded from below by
a positive constant. The proof goes as in the two dimensional case.

Let $\delta$ be a positive constant such that $\inf_{i,Q} \{\delta_{i,Q}|\} \geq \delta > 0$.
Therefore,
\begin{eqnarray*}
\delta \overline{\zeta}_j  & \prec & \sum_{Q} \delta_{j,Q}
\| \zeta_{j,Q} \| y^Q + \overline{\psi}_j \\
& \prec & \overline{\varphi}_j (y_1 + \overline{\zeta}_1,\ldots,
y_n +\overline{\zeta}_n) + \sum_{k=1}^n \frac{\partial
\overline{\zeta}_j}{\partial y_k} \overline{\psi}_k \, .
\end{eqnarray*}
Since $\overline{\psi}_k =0$, we have
\[
\delta \overline{\zeta_j} \prec \overline{\varphi}_j (y_1 +
\overline{\zeta}_1,\ldots, y_n +\overline{\zeta}_n)\; .
\]

The convergence of the desired coordinate change results by
applying Theorem~\ref{Siegel}.
\end{proof}

Now we shall approach the case of a saddle-node singularity in
dimension $3$. First we will consider the case where only one of
the eigenvalues is {\it zero}.

\begin{teo}\label{saddlenode3}
Let $X$ be a holomorphic vector field defined in a neighborhood of
the origin of $\C^3$. Denote by $\dl_1, \dl_2, \dl_3$ the eigenvalues
of the linear part of $X$ at the origin. Assume that $\dl_3 = 0$. Assume
also that $\dl_1, \, \dl_2$ are non-vanishing eigenvalues not satisfying
any resonance relation and belonging to the Poincar\'e domain. Then there
is an analytic change of coordinates where the original system is given
$($in terms of vector fields$)$ by
$$
X = [\lambda_1 y_1 + y_3\psi_1(y_1,y_2,y_3)] \partial /\partial y_1 +
[\lambda_2 y_2 + y_3\psi_2(y_1,y_2,y_3)] \partial /\partial y_2 +
y_3H(y_1, y_2, y_3)\partial /\partial y_3\, .
$$
with $\psi_1(0,0,0) = \psi_2(0,0,0) = H(0,0,0) = 0$.
In particular, $\{y_3=0\}$ is an invariant 2-plane.
\end{teo}

\begin{proof}
Using the same notation as before, we
consider a formal change of coordinates such that:
\[\left\{ \begin{array}{ll}
         \psi_{1,Q} = 0 & {\rm when} \; \; q_3 =0  \\
         \zeta_{1,Q} =0 & {\rm when} \; \; q_3\neq 0
        \end{array} \right.
\;\; ;\;\;
 \left\{ \begin{array}{ll}
         \psi_{2,Q} = 0 & {\rm when} \; \; q_3 =0 \\
         \zeta_{2,Q} =0 & {\rm when} \; \; q_3 \neq 0
        \end{array} \right.
\; \; \;  {\rm and} \; \; \; \left\{ \begin{array}{ll}
         \psi_{3,Q} = 0 & {\rm when} \; \; q_3 =0 \\
         \zeta_{3,Q} =0 & {\rm when} \; \; q_3 \neq 0
        \end{array} \right.
\]
As in Theorem~\ref{PLT}, the assumption over the eigenvalues guarantee
the existence of a positive constant $\delta>0$ such that $\inf_i
\{|\delta_{i,Q}| : q_1 q_2 q_3 \ne 0\} > \delta > 0$.
Hence,
\[\delta \overline{\zeta}_i \prec \overline{\varphi}_j (y + \overline{\zeta}) + \frac{\partial
\overline{\zeta}_i}{\partial y_1} \overline{\psi}_1 +
\frac{\partial \overline{\zeta}_i}{\partial y_2} \overline{\psi}_2
+ \frac{\partial \overline{\zeta}_i}{\partial y_3}
\overline{\psi}_3 \, .\]

Notice that $\overline{\zeta}_i$ depends only on $y_1$ and $y_2$,
so that the last term in the above estimate vanishes for
$i=1,2,3$. On the other hand, by construction, all the non-zero
coefficients $\psi_{i,Q}$ are associated to monomials that depend
on $y_3$, in such a way that we are able to conclude:
\[
\delta \overline{\zeta_i} \prec \overline{\varphi}_i (y +
\overline{\zeta})\; ,
\]
and, as before, the application of Theorem~\ref{Siegel} concludes the
proof.
\end{proof}

Next, let us consider the case where two eigenvalues are equal to
{\it zero}. We will notice that the method used in the
previous case to prove the existence of an invariant plane does
not work in the present situation.

Let $X$ be a holomorphic vector field defined in a neighborhood of the origin
of $\C^3$, whose linear part has eigenvalues $\lambda_1 = \lambda_2 = 0$ and
$\lambda_3\neq 0$. We could try to see if the previous method allows us to prove
the existence of an invariant plane. Let $X$ be given, in local coordinates, by:
\[
\varphi_1(x_1,x_2,x_3)\partial /\partial x_1 + \varphi_2(x_1,x_2,x_3)\partial /\partial
x_2 +  [\lambda_3x_3 + \varphi_1(x_1,x_2,x_3)]\partial /\partial
x_3\,
\]
where $\varphi_1, \, \varphi_2, \, \varphi_3$ are holomorphic functions of order
at least~$2$. Using the same method it is possible to prove the existence of a
formal change of coordinates such that $\varphi_1$ is divisible by $x_1$, and
$\varphi_2$, $\varphi_3$ are divisible by $x_2$. In this case, the
plane $\{x_1=0\}$ is invariant under $X$. However, the method used in the above
situation does not allow us to prove the convergence of the series.

Let us now go back to the saddle-node case with exactly one eigenvalue equal to zero.

\begin{teo}
\label{t6} Let $X$ be a holomorphic vector field defined in a neighborhood of
the origin of $\C^3$, with eigenvalues $\lambda_1\neq 0$,
$\lambda_2\neq 0$ and $\lambda_3=0$. Let $n_0 \in \N$ be an
arbitrary positive integer. Then there is an analytic change of
coordinates where the original system is given $($in terms of vector
fields$)$ by
$$
X = F(y_1,y_2,y_3) \partial /\partial y_1 + G(y_1,y_2,y_3)
\partial /\partial y_2 + H(y_1, y_2, y_3)\partial /\partial
y_3\, ,
$$
such that $F(0,0,y_3)=G(0,0,y_3)\equiv 0$, i.e, the axis
$\{y_1=y_2=0\}$ is invariant under $X$.
\end{teo}

\begin{proof}
Consider a change of variables of the form
\begin{equation}
\label{eq17} x_1 = y_1 + \sum_{i=2}^N a_i y_3^i\,,\;\;\; \;\;\;
x_2 = y_2 + \sum_{i=2}^N b_i y_3^i\,,\; \;\;\; \; \; x_3 =y_3\,.
\end{equation}
We have to show that $N,a_i ,b_i$ can be chosen so as to fulfill
our requirements. In the coordinates $(y_1 ,y_2, y_3)$ the vector
field $X$ is given by
\begin{equation}
 \left[ \begin{array}{cccc}
           1 & 0 & -\sum_{i=2}^N ia_i y_3^{i-1} \\
           0 & 1 & -\sum_{i=2}^N ib_i y_3^{i-1} \\
           0 & 0 & 1
          \end{array} \right] \left[ \begin{array}{c}
                                       F_1 \\
                                       G_1 \\
                                       H_1
                                      \end{array} \right]
\, . \label{eq20}
\end{equation}
In the above formula, the functions $F_1 ,G_1 ,H_1$ admit
respectively the expressions below:
\begin{eqnarray} F_1 & = & \lambda_1 y_1 +
\sum_{i=2}^N\lambda_1a_iy_3^i + (\sum_{i=2}^N a_i
y_3^i)(\sum_{i=2}^N b_i y_3^i) (f_{1,1} (y_3)
+ (\ast)) + f_{1,2} (y_3) + (\ast \ast), \label{eq18} \\
G_1 & = & \lambda_2 y_2 + \sum_{i=2}^N\lambda_2 b_iy_3^i +
(\sum_{i=2}^N a_i y_3^i)(\sum_{i=2}^N b_i y_3^i) (g_{1,1} (y_3) +
(\ast \ast \ast)) + g_{1,2} (y_3) + (\ast \ast \ast \ast) ,
\label{quinze22} \\
H_1 & = & h(y_3) + (\ast \ast \ast \ast \ast). \label{eq19}
\end{eqnarray} In the above equations, the components represented
as $(\ast), \ldots , (\ast \ast \ast \ast \ast)$ do not contain
neither constants nor terms depending only on $y_3$. In other
words these components belong to the ideal $I (y_1) \cup I (y_2)$.

On the other hand, we want to consider the terms depending only on
$y_3$ which appear in the $2$ first coordinates of $X$. After
performing the matricial product~(\ref{eq20}), these coordinates
are respectively given by
\begin{eqnarray}
& & \sum_{i=2}^N \lambda_1 a_i y_3^i + (\sum_{i=2}^N a_i
y_3^i)(\sum_{i=2}^N b_i y_3^i) f_{1,1} (y_3) + f_{1,2} (y_3) -
(\sum_{i=2}^N ia_i y_3^{i-1}) h(y_3) \; \, \mbox{and}
\label{dezesseis} \\
& & \sum_{i=2}^N \lambda_2 b_i y_3^i + (\sum_{i=2}^N a_i
y_3^i)(\sum_{i=2}^N b_i y_3^i) g_{1,1} (y_3) + g_{1,2} (y_3) -
(\sum_{i=2}^N ib_i y_3^{i-1}) h(y_3) \, .\label{dezessete}
\end{eqnarray}
We now consider the Taylor expansions of $f_{1,1} ,g_{1,1}$,
namely we set
\begin{equation}
f_{1,1} (y_3) = \alpha_0^{(1)} + \alpha_1^{(1)} y_3 + \cdots \; \;
\mbox{and} \; \; g_{1,1} (y_3) = \beta_0^{(1)} + \beta_1^{(1)} y_3
+ \cdots \, . \label{dezoito}
\end{equation}
Furthermore the assumption on the linear part of $X$ and the
expression of the change of coordinates in~(\ref{eq17}) allow us
to write
\begin{eqnarray}
f_{1,2} (y_3) = \alpha_2^{(2)} y_3^2 + \cdots & ; &
g_{1,2} (y_3) = \beta_2^{(2)} y_3^2 + \cdots \, , \nonumber \\
h (y_3) = c_k y_3^k + \cdots &  & (c_k \neq 0 , \; k \geq 2) \, .
\label{vinte}
\end{eqnarray}

Because we just want to cancel the coefficients of degree less
than $n_0+1$ which depend solely on $y_3$, the
Equations~(\ref{dezesseis}) and~(\ref{dezessete}) can respectively
be replaced by~(\ref{vintetres}) and~(\ref{vintequatro}) without
loss of generality, where
\begin{eqnarray}
 & & \sum_{i=2}^{n_0} \lambda_1 a_i y_3^i + (\sum_{i=2}^{n_0-1} a_i y_3^i)
(\sum_{i=2}^{n_0-1} b_i y_3^i) (\alpha_0^{(1)} + \alpha_1^{(1)}
y_3
+ \cdots + \alpha_{n_0-2}^{(1)} y_3^{n_0-2}) + \nonumber \\
 & & \mbox{    } + (\alpha_2^{(2)} y_3^2 + \cdots + \alpha_{n_0}^{(2)}
y_3^{n_0} ) - (\sum_{i=2}^{n_0-k} ia_i y_3^{i-1})(c_k y_3^k +
\cdots + c_{n_0} y_3^{n_0}) \,  \label{vintetres}
\end{eqnarray}
and
\begin{eqnarray}
 & & \sum_{i=2}^{n_0} \lambda_2 b_i y_3^i + (\sum_{i=2}^{n_0-1} a_i y_3^i)
(\sum_{i=2}^{n_0-1} b_i y_3^i) (\beta_0^{(1)} + \beta_1^{(1)} y_3
+ \cdots + \beta_{n_0-2}^{(1)} y_3^{n_0-2}) + \nonumber \\
 & & \mbox{    } + (\beta_2^{(2)} y_3^2 + \cdots + \beta_{n_0}^{(2)}
y_3^{n_0} ) - (\sum_{i=2}^{n_0-k} ib_i y_3^{i-1})(c_k y_3^k +
\cdots + c_{n_0} y_3^{n_0}) \, . \label{vintequatro}
\end{eqnarray}

In view of the two formulas~(\ref{vintetres})
and~(\ref{vintequatro}) above, the proposition is reduced to the
following claim.

\vspace{0.1cm}

\noindent {\bf Claim:} The coefficients $\alpha_0^{(1)},
\beta_0^{(1)}$ are constant (that is they do not depend on $a_i
,b_i$). Furthermore $\alpha_l^{(1)} ,\beta_l^{(1)}$ are
polynomials on $a_1 ,b_1 ,\ldots ,a_l ,b_l$ for $1 \leq l \leq
n_0-2$. In addition, $\alpha_l^{(2)},\beta_l^{(2)}$ are
polynomials on $a_1  ,b_1, \ldots ,a_{l-1} ,b_{l-1}$ for $1 \leq l
\leq n_0$. In particular none of these coefficients depend on
$a_{n_0} ,b_{n_0}$.

\vspace{0.1cm}

\noindent {\it Proof of the Claim}\,. The facts concerning
$\alpha_l^{(1)} ,\beta_l^{(1)}$ ($l=0,1, \ldots ,n_0-2$) are
immediate consequences of~(\ref{dezoito}) and of the form of the
change of coordinates~(\ref{eq17}).

The assertion regarding $\alpha_l^{(2)} ,\beta_l^{(2)}$
($l=1,\ldots ,n_0$) has a similar justificative which however
deserves further comments.

When we perform the change of coordinates given by~(\ref{eq17})
and consider, for instance, $dy_1 /dT$, we obtain
$$
\lambda_1 y_1 + \sum_{i=2}^N \lambda_1a_i y_3^i + (\sum_{i=2}^N
a_i y_3^i)(\sum_{i=2}^N b_i y_3^i) (f_1^1 (y_3)  + (\ast)) + f_2^1
(y_3) + (\ast \ast) \, ,
$$
for appropriate holomorphic functions $f_1^1 ,f_2^1$. Moreover
$(\ast)$ represents terms which do not depend solely on the
variable $y_3$ while $(\ast \ast)$ represents terms which are
neither constant nor depend solely on the variable $y_3$. Fix a
monomial of $f_2^1$ having degree $r \in \N$. The coefficients
entering into the term $\alpha_q^{(2)}$ are those corresponding to
the monomials $y_1^{q_1} y_2^{q_2} y_3^{q_3}$ such that $1 < q_1
+q_2 +q_3 < r$. This implies the claim. The proof of the theorem
is over.
\end{proof}




\section{Advanced Aspects of Singularities}


\subsection{Saddle-Node in Dimension $2$}


\subsubsection{Basic Properties}

As was seen earlier, for a foliation $\fol$ (or vector field) with
a saddle-node singularity in $(0,0)\in \C^2$ there is a
holomorphic change of coordinates, by means of which $\fol$ may be
given by the $1$-form (the Dulac's normal form):
\[
\omega(y_1,y_2) = [y_1 (1 + \lambda y_2^p) + y_2 R(y_1,y_2)] \,
dy_2 \; - \; y_2^{p+1} \, dy_1 \, ,
\]
where $\lambda\in \C$, $p\in \N^\ast$ and the order of $R$ at
$(0,0)$ with respect to $y_1$ is at least $p+1$.

Let us consider the {\it formal} change of coordinates:
\begin{equation}
\label{eq20} (y_1,y_2)\mapsto (\varphi(y_1,y_2), y_2) \, ,
\end{equation}
taking $\omega$ into its formal normal form $\omega_{p,\lambda}$,
where $\varphi(y_1,y_2)=y_1 + \sum_{i=1}^\infty a_i(y_1) y_2^i$
and
\[
\omega_{p,\lambda}(y_1,y_2)=[y_1 (1 + \lambda y_2^p)] \,
dy_2 \; - \; y_2^{p+1} \, dy_1 \, .
\]
A careful look shows that the functions $a_1(y_1)$ are holomorphic on a
common neighborhood of $0\in \C$, though the change of variables
is {\it not} necessarily convergent. In other words, the
$1$-forms, $\omega$ and $\omega_{p,\lambda}$ are {\it not}
holomorphically conjugate, in general.

\begin{ex}
Let us consider the following example due to Euler:
\[
(y-x^2) dx - x^2 dy=0\,.
\]
It admits a formal solution given by
\[
y(x) = \sum_{n=1}^{\infty} n! x^{n+1} \, ,
\]
which does not converge in any neighborhood of $0$. Therefore the vector field cannot
be holomorphically conjugated to its formal normal form.
\end{ex}

In this section we are going to discuss the problem of the
analytic classification of saddle-nodes following \cite{Ma-R}. Our
purpose is to summarize the main points of \cite{Ma-R} by
explaining the role of the sectorial normalizations along with the
analytic invariants of the diffeomorphisms arising from ``changing
the sector''. By a {\it sector $V$ with vertex at $0$}, we mean an
angular sector of angle $\theta < 2\pi$, intersected with the ball
$B_r\subset \C$ of radius $r$ centered in $0\in \C$.

Let $f$ be a holomorphic function in $U \times V$, where $U$ is a
neighborhood of $0 \in \C$ and $V\subseteq \C$ is a sector with
vertex at $0$. We say that $\hat{f}=\sum_{i=1}^\infty a_i(y_1)y_2^i$
is the {\it asymptotic expansion} of $f$ at $0\in \C$ if for each
$n \in \N$ there exists $A_n(y_1)>0$ such that for every $y_1\in U$:
\[
\left| f(y_1,y_2)-\sum_{r=0}^{n-1} a_r(y_1) y_2^r \right| \leq A_n(y_1)y_2^n \; .
\]

The idea is that certain formal series can be realized as
asymptotic expansions of holomorphic functions that are defined on
sectors of angles $\theta<2\pi$ and sufficiently small radius $r$.
In other words they are not defined on a neighborhood of zero,
otherwise the series would have to converge since they would agree
with the Taylor series of the functions.

The next theorem, due to H. Hukuara, T. Kimura and T. Matuda
\cite{H-K-M} implies that, although in general the $1$-forms
$\omega$ and $\omega_{p,\lambda}$ are not holomorphically
conjugate in neighborhoods of $(0,0)\in \C^2$, they are actually
holomorphically conjugate in conveniently chosen sectors.

\begin{teo}[Hukuara-Kimura-Matuda \cite{H-K-M}]
\label{t7}
Let $\varphi$ be a formal series as defined in
(\ref{eq20}), i.e. the change of coordinates that takes $\omega$
into $\omega_{p,\lambda}$. Then, for every sector $V \subseteq \C$ of
angle less than $\frac{2\pi}{p}$, there exists a bounded
holomorphic map
\begin{eqnarray*}
\Phi_V : B_r \times (V\setminus \{0\}) & \rightarrow & \C \times
(V\setminus \{0\})\\
(y_1,y_2) &\mapsto & (\varphi_v(y_1,y_2),y_2)\; ,
\end{eqnarray*}
such that:
\begin{enumerate}
\item $\Phi_V^\ast \omega\wedge\omega_{p,\lambda}=0$;

\item $\varphi$ is the asymptotic expansion of $\Phi_V$ at $0\in \C$.
\end{enumerate}

$\Phi_V$ is called a normalizing application.
\end{teo}

To simplify, we will consider the particular case of $p=1$. Note
that the cases $p>1$ may be dealt with in a similar way, except
that the number of sectors which are necessary to cover the ball
centered in $0 \in \C$ increases. The general procedure would be a
straightforward generalization of the case $p=1$ somehow in the
spirit of the Flower Theorem~(\ref{t14}) compared to the case
$p=1$ (cardioid-shaped region).

According to Hukuara-Kimura-Matuda Theorem, we can cover
the neighborhood $B_r$ of $0\in \C$ with two sectors of angles
less than $2\pi$. To fix ideas, we may suppose that $V_1=B_r \cap
\{z\in \C; \arg\in [0, 5\pi/4) \cup (7\pi/4, 2\pi] \}$ and
$V_2=B_r \cap \{z\in \C; \arg\in [0,\pi/4) \cup (3\pi/4, 2\pi]
\}$. Notice that $V_1 \cap V_2$ has two connected components. The
bisectrix of one of them is the positive real semi-axis, therefore
we shall denote it by $V^+$. Analogously for $V^-$, the component
whose bisectrix is the negative real semi-axis.

Let us consider the foliation induced by the equation
$\omega_{1,\lambda}=0$ on a neighborhood of $(0,0)\in \C^2$. Let
$\fol_1$ (resp. $\fol_2$) be the restriction of the foliation to $B_r
\times V_1$ (resp. $B_r \times V_2$). Naturally $\fol_1$ and
$\fol_2$ coincide in the intersection of the domains, so we denote
the foliation in $B_r \times V^+$ (resp. $B_r \times V^-$) by
$\fol^+$ (resp. $\fol^-$).

Theorem~\ref{t7} assures the existence of normalizing applications
$H_1$ and $H_2$ defined on $B_r \times (V_1\setminus \{0\})$ and
$B_r \times (V_2\setminus \{0\})$, respectively.
In the following, our aim is to understand the behavior of the
applications that are responsible for the changing of sectors.
First of all, the change of sectors $H_1 \circ H_2^{-1}$ gives
rise to two diffeomorphisms denoted by $g^+$ and $g^-$. The first
one corresponds to the restriction of $H_1 \circ H_2^{-1}$ to $V^+$
while the second one corresponds to the restriction of the same map
to $V^-$. The next proposition gives a characterization for these
functions.

\begin{prop}
\label{p1} The diffeomorphism $g^+=H_1 \circ H_2^{-1}|_{V+}$ is a
translation and the diffeomorphism $g^-=H_1 \circ H_2^{-1}|_{V^-}$
is tangent to the identity.
\end{prop}

Before proving this proposition we will give a geometric approach
to $g^+$ and $g^-$. Firstly, notice that the solutions of
$\omega_{1,\lambda}=0$, i.e. the leaves of the foliation, are given
by
\begin{equation} \label{eq40}
y_1(y_2)=c y_2^\lambda \exp \left(-\frac{1}{y_2} \right)\, ,
\end{equation}
with $c\in \C$. So each leaf of $\fol^+$ (resp. $\fol^-$) is in
correspondence with $c \in \C$. In other words, the leaf space is
isomorphic to $\C$, being parameterized by the constants $c \in
\C$. Therefore $g^+$ (or $g^-$, depending on whether $Re(x)>0$ or
$Re(x)<0$) is defined on $\C$ and $g^+(c)$ (or $g^-(c)$) is the
corresponding leaf when we change sectors.

The following is classical after \cite{Ma-R}, we follow however
the discussion in \cite{R1}. Fix a sector $V$ and let us consider
the group of automorphisms $\Lambda_{\omega_{1,\lambda}}(V)$, such
that each element is a diffeomorphism $\phi$ defined on $V$ with
the following properties:
\begin{enumerate}
\item $\phi(y_1,y_2)=(\varphi(y_1,y_2), y_2)$,

\item $\phi$ is asymptotic to the identity,

\item $\phi$ preserves the foliation $\fol$ induced by
$\omega_{1,\lambda}=0$, i.e. $\phi^{\ast} \omega_{1,\lambda}
\wedge \omega_{1,\lambda} = 0$.
\end{enumerate}

Note that if $H_1, \, H_2$ are normalizing applications on
$B_r \times V_1, \, B_r \times V_2$, respectively, then the restriction
of $H_1 \circ H_2^{-1}$ to $V^+$ (resp. $V^-$) is an element of
$\Lambda_{\omega_{1,\lambda}}(V^+)$ (resp. $\Lambda_{\omega_{1,\lambda}}(V^-)$).
The normalizing application obtained in Theorem~\ref{t7} is
{\it not} however uniquely defined. Indeed, if $H_1$ is a normalizing
application in $B_r \times (V_1\setminus \{0\})$, then so is $H_1 \circ
\phi$, for all $\phi \in \Lambda_{\omega_{1,\lambda}}(B_r \times
V_1)$. In other words, the normalizing application
is uniquely defined up to the composition with an element of
$\Lambda_{\omega_{1,\lambda}}(V)$. This is why the change $H_1
\circ H_2^{-1}$ is {\it not} the identity in general. The change
is indeed an {\it asymptotic expansion of the identity}. More
precisely, it is an element of $\Lambda_{\omega_{1,\lambda}}
(V_1)$. The special case in which the gluing of the leaves
is the identity is exactly the case where the formal normal form
is holomorphically conjugate to Dulac's normal form.

\begin{proof}[Proof of Proposition~\ref{p1}]
As already mentioned, the sector change $\phi=H_1 \circ H_2^{-1}$ is an element of
$\Lambda_{\omega_{1,\lambda}}$. Indeed Theorem~\ref{t7} guarantees
that conditions~$1$ and~$2$ are fulfilled. As to the third
condition, we notice that $H_i$ ($i=1,2$) are such that:
\[
dH_i(X) = X_{1, \dl}(H_i)
\]
where $X$ is the given vector field (i.e. $X = y_1 (1 + \lambda y_2) + y_2
R(y_1,y_2) \partx + y_2^2 \party$) and $X_{1, \dl}$ is its correspondent formal
normal form (i.e. $X_{1, \dl} = y_1 (1 + \lambda y_2) \partx + y_2^2 \party$). Thus,
\begin{eqnarray*}
d(H_1 \circ H_2^{-1})(X_{1, \dl})(H_1 \circ H_2^{-1})^{-1} & = & dH_1 \circ dH_2^{-1}(X_{1, \dl})H_2 \circ H_1^{-1} \\
&=& dH_1 \circ dH_2^{-1}\circ dH_2(X) \circ H_1^{-1} \\
&=& dH_1(X) \circ H_1^{-1} \\
&=& X_{1, \dl}(H_1) \circ H_1^{-1} \\
&=& X_{1, \dl} \, .
\end{eqnarray*}

The map $\phi(y_1,y_2)=(y_1 + b_0(y_2) + \sum_{i=1}^\infty
b_i(y_2)y_1^i, y_2)$ belongs to $\Lambda_{\omega_{1,\lambda}}(V)$, if
it is asymptotic to the identity and preserves the foliation. If we
think in terms of vector fields the last condition implies that:
\begin{equation}\label{preservedvf}
d \phi(X_{1, \dl}) = X_{1, \dl}(\phi)\, ,
\end{equation}
while the first one is satisfied if and only if the functions
$b_j(y_2)$ are asymptotic to the zero function when $y_2$ tends
to $0 \in \C$. The left hand side of Equation~\ref{preservedvf}
is given by
\[
y_1 (1 + \lambda y_2)(1 + \sum_j j b_j(y_2) y_1^{j-1}) +
y_2^2 b^{\prime}_0(y_2) + y_2^2 \sum_j b^{\prime}_j(y_2) y_1^j
\]
while the right one is given by
\[
\Big[y_1 + b_0(y_2) + \sum_j b_j(y_2) y_1^j \Big] (1 + \lambda y_2) \, .
\]
The equality between those expression leads us to the following ordinary
differential equation:
\[
b^{\prime}_j(y_2) y_2^2 + b_j(y_2) (j-1) (1 + \lambda y_2) = 0\, .
\]
whose solution is given by
\[
b_j(y_2)= c_j y_2^{(1 - j)\lambda} \exp \Big( \frac{j - 1}{y_2} \Big) \, ,
\]
where $c_j$ is the initial data. Therefore unless $c_j = 0$ we have
\[
\left\{ \begin{array}{ll}
         \displaystyle\lim_{\Re (y_2) \rightarrow 0^+} b_j(y_2) = \infty \\
         \displaystyle\lim_{\Re (y_2) \rightarrow 0^-} b_j(y_2) = 0
\end{array} \right.
\]
for all $j > 1$, where by $\Re (y_2)$ we mean the real part of $y_2$.
We have therefore that, for $j > 1$, $b_j(y_2)$ is asymptotic to the null
function if and only if $y_2 \in V^-$ or $c_j = 0$. This is equivalent to
say that $c_j = 0$ on $V^+$ for all $j > 1$, i.e. that $g^+$ takes the form
\[
g^+(y_1,y_2) = (y_1 + b_0(y_2), y_2)
\]

In the same way we notice that
\[
\left\{ \begin{array}{ll}
         \displaystyle\lim_{\Re (y_2) \rightarrow 0^+} b_0(y_2) = 0 \\
         \displaystyle\lim_{\Re (y_2) \rightarrow 0^-} b_0(y_2) = \infty
\end{array} \right.
\]
Therefore or $y_2 \in V^+$ or $c_0 = 0$ in order to $b_0$ to be asymptotic to
the null function. This implies that $b_0$ vanishes identically on $V^-$, i.e.
that $g^-$ is tangent to the identity.
\end{proof}

Since each leaf of $X_{1, \dl}$ is parametrized by a constant $c \in \C$
(Equation~\ref{eq40}) the diffeomorphismes $g^-, \, g^+$ can be expressed
in terms of the parametrization by those constants. In this way Proposition~\ref{p1}
says that
\[
\left\{ \begin{array}{ll}
         g^+(c) = c + a_0\\
         g^-(c) = c + \sum_{i=2}^\infty a_i c^i
\end{array} \right.
\]

Now given two saddle-node foliations that are holomorphically
conjugate, it is interesting to work out the relation between
their sector changing diffeomorphisms (which in the $p=1$ case, we
denoted by $g^+$ and $g^-$). Indeed, we are going to see that if
the saddle-nodes are holomorphically conjugate, then their sector
changing diffeomorphisms are also conjugate by an automorphism of
the leaf space. The converse is also true, though it will not be
fully proved.

Let $\fol_1$ and $\fol_2$ be holomorphically conjugate foliations
given by $1$-forms $\omega_1$ and $\omega_2$, respectively. In
other words, there exists a diffeomorphism $H$ that takes the
leaves of $\fol_1$ in leaves of $\fol_2$. There are also formal
changes of coordinates $h_1$ and $h_2$ that take $\omega_1$ and
$\omega_2$, respectively in formal normal forms $\tilde{\omega}_1$
and $\tilde{\omega}_2$. As was seen in the earlier discussion,
though $h_1$ (resp. $h_2$) is not analytic on a neighborhood of
$0\in \C$, there are sectors in which the series does converge. In
the case $p=1$ there are  functions $g_1^+$, $g_1^-$ (resp. $g_2^+$,
$g_2^-$) that glue together the leaves of the sectors $V^+$ and $V^-$.

We shall define the following equivalence relation: two
diffeomorphisms $f$ and $\tilde{f}$ are said to be equivalent (and
we write $f \sim \tilde{f}$) if there exists $\sigma\in {\rm
Diff} (\C,0)$ with $\sigma^{\prime}(0)=1$ such that:
\[
\tilde{f}=\sigma^{-1}\circ f\circ \sigma\, .
\]
In other words, conjugate functions belong to the same equivalence
class. In this sense, if the foliations are holomorphically
conjugate, the transition function $g_1^+$ (resp. $g_1^-$) is
equivalent to $g_2^+$ (resp. $g_2^-$). Indeed, notice
that the function $\sigma= h_2\circ H\circ h_1^{-1} \in {\it Diff}
(\C,0)$ is an automorphism of the leaf space (which is isomorphic
to a neighborhood of $0\in \C$ as observed in Formula~\ref{eq40})
tangent to the identity. Denoting by $\sigma_+$ the restriction of
$\sigma$ to $V^+$ and by $\sigma_-$ the restriction of $\sigma$ to
$V^-$ it follows that $g_1^+ = \sigma_+^{-1} \circ g_2^+ \circ \sigma_+$
and $g_1^- = \sigma_-^{-1} \circ g_2^- \circ \sigma_-$.

Summarizing, the analytic type of saddle-nodes with $p=1$ are in
correspondence with the conjugacy class of the diffeomorphisms
$g^+$ and $g^-$. Since the conjugacy class of translations is
obvious, we can think that the information is totally encoded in
the conjugacy of the diffeomorphism $g^-$ tangent to the identity.
It is then natural to study the moduli space of diffeomorphisms of
$(\C,0)$ tangent to the identity so as to have concrete invariants
for saddle-node singularities.


\subsubsection{Automorphisms of $(\C, 0)$ with identity linear part}

Now we will concentrate on the study of the automorphisms $\sigma$
of $(\C,0)$ tangent to the identity that were used in the above
definition of equivalence relation. The description of the moduli
spaces is independently due to Ecalle and Voronin. In what follows
we shall follow Voronin construction for the prototypical case
$p=1$.

Precisely, we shall give an analytic classification of the
mappings of the set $\cal{A}=$ $\{f \in {\rm Diff}(\C,0)\; ; \;
f(z)=z+az^2+\cdots , \; a\neq 0 \} $. As usual, we define two
mappings $f_1$, $f_2 \in \cal{A}$ to be equivalent if and only if
there exists a holomorphic diffeomorphism $H$ such that $H\circ
f_1=f_2\circ H$. Here we describe these classes of equivalence.

To begin with we give some Analysis results and definitions that
will be used in this section.

\begin{defi}
A homeomorphic mapping $f:\Omega\rightarrow\C$ of a domain
$\Omega\subset\C$ is said to be quasiconformal if
\[|f_{\bar{z}}|\leq k |f_z|\]
for $k<1$ almost everywhere.
\end{defi}

The function $h_f=f_{\bar{z}}/f_z$ is called the {\it
characteristic} of the quasiconformal map $f$, and the quantity
\[
K_f(z_0)={\lim}\;{\rm sup}_{r \rightarrow 0} \frac{\sup_{|z-z_0|=r}
|f(z)-f(z_0)|}{\inf_{|z-z_0|=r} |f(z)-f(z_0)|}\,.
\]
is called the {\it quasiconformal deviation} of the map $f$ at the
point $z_0$. Notice that if $g$ is a quasiconformal map with
characteristic $h_g=h_f$, then $g\circ f^{-1}$ is conformal.

\begin{prop}
\label{p3} A map is quasiconformal in $\Omega$ if and only if
$K_f(z_0)<\infty$ for all $z_0\in \Omega$ and
$K=\|K_f\|_{L_\infty(\Omega)}<\infty$.
\end{prop}

\begin{teo}[Measurable Riemann Theorem]
\label{t8} For any measurable function $h$ such that
$\|h\|_{L_\infty}<1$, there exists a quasiconformal map $f$ of the
plane $\C$ onto itself, having the function $h$ as its
characteristic $h=h_f$.
\end{teo}

Next, we consider the specific case of the function $f_0(z)=
z/(1-z)$ in $\cal{A}$. Notice that the inversion $A_0(z)=-1/z$
conjugates $f_0$ conformally to the translation $T(z)=z+1$ on
$\C^\ast$, that is,
\begin{equation}
\label{eq23} A_0\circ f_0=T\circ A_0\,.
\end{equation}
Next we are going to see that the last remark holds in general.
Namely for every function $f$ belonging to $\cal{A}$, there always
exist certain domains in $\C$ where $f$ is {\it quasiconformally}
conjugate to $T$. We prove this by using the next two results.

\begin{teo}[A. Shcherbakov]
\label{t9}Let $f\in \cal{A}$, $f_0(z)=z/(1-z)$. For every
$\varepsilon
>0$ one can find $\delta$ and a homeomorphism $H(z)=z+h(z)$ of the
disk $K_\delta$ (of radius $\delta$) onto itself such that:
\begin{eqnarray}
H\circ f_0 &=& f\circ H\; \mbox{on}\; K_\delta \label{eq21} \\
|h(z_1)-h(z_2)| &<& \varepsilon|z_1-z_2|, \; z_j\in K_\delta
\label{eq22}
\end{eqnarray}
\end{teo}

This theorem basically asserts that every function tangent to the
identity is, in a sufficiently small disk, conjugate to $f_0$ by a
Lipschitz mapping with Lipschitz constant close to $1$. The proof
of the theorem amounts to some finer estimates involving the Fatou
coordinates \cite{Car-G}.

\begin{lema}
\label{l4} Suppose that for the homeomorphism $H(z)=z+h(z)$,
Estimate (\ref{eq22}) holds for $\varepsilon<1$, then $H$ is
quasiconformal in $K_\delta$.
\end{lema}

\begin{proof}
This lemma follows immediately from
Proposition~\ref{p3}. Indeed, under these conditions, the
quasiconformal deviation of $H$, $K(z_0)$ is such that
$K(z_0)\leq(1+\varepsilon)/(1-\varepsilon)$ for $z\in K_\delta$.
\end{proof}

\begin{prop}
For each $f\in \cal{A}$ there exist domains $R \, , L \subseteq \C$
and a quasiconformal mapping $G$ defined on $R\cup L$, satisfying
\begin{itemize}
\item $G\circ T=f\circ G$ on $R$,

\item $G\circ T^{-1}=f^{-1}\circ G$ on $L$.
\end{itemize}
\end{prop}

\begin{proof}
Fix $\varepsilon$ such that $0<\varepsilon<1$. From Theorem~\ref{t9}
there exists a homeomorphism $H:K_\delta\rightarrow K_\delta$,
conjugating $f$ and $f_0$ as in (\ref{eq21}). Moreover, this
homeomorphism is quasiconformal by Lemma~\ref{l4}.

Let $R=\{z\in \C;\; |{\rm Im}\, z|> c\, ,\; {\rm Re}> c\}$ for a
fixed complex number $c=1/\rho$ with $\rho<\delta$. The mapping
$G=H\circ A_0^{-1}$ defined on $R$ is well-defined and
quasiconformal, being a composition of a quasiconformal map with
a conformal one. Set $\Omega_1=G(R)$. It follows from
(\ref{eq23}) and (\ref{eq21}) that $G$ is precisely the conjugacy
we were looking for. In conclusion, $f$ is quasiconformally
conjugate to a translation.

Moreover, since $T(R)\subset R$ it follows that $\Omega_1$ is invariant
by $f$, i.e. $f(\Omega_1)\subset\Omega_1$. Analogously, we can define a
domain $\Omega_2$ that is invariant by $f^{-1}$. More precisely,
$\Omega_2=G(L)$ where $L=\{z\in \C;\; |{\rm Im}\, z|> c\, ,\;
{\rm Re}< -c\}$ with $c$ as before. Naturally, $f^{-1}$ is defined on
$L$ and is quasiconformally conjugate to the translation $T^{-1}$. Thus
establishing the proposition.
\end{proof}

Equivalently, one may define the domains $R$ and $L$ so as to make
their boundaries smooth (refer to \cite{Car-G}). Under these
circumstances $A_0^{-1}(R)$ is the well-known cardioid-shaped
region (cf. Section $3.2.2$). The reader may keep this figure in
mind whenever we refer to $R$ and $L$ as it may help intuitively.

The next lemma gives analytic coordinates $A_1$ and $A_2$ defined
on the domains $\Omega_1$ and $\Omega_2$. In what follows, we can
trace a parallel between the saddle-node case, analyzed in the
previous section, and the gluing of the domains $\Omega_1$ and
$\Omega_2$. In this sense, we are still considering conjugacies
which are not defined on a full neighborhood of the origin. In the
saddle-node case, we have sectorial normalizations that are
provided by Hukuara-Kimura-Matuda Theorem. The analogue of this
theorem in the present case is the lemma below.

\begin{lema}[Basic Lemma]
\label{l6} Let $R$ and $\Omega_1$ be as before. Set $T(z)=z+1$ and
consider a quasiconformal homeomorphism $G_1$ of the domain $R$
onto $\Omega_1$. Let $f$ be analytic in $\Omega_1$ such that
\begin{equation}
\label{eq25} G_1\circ T= f\circ G_1\, .
\end{equation}
Then there exists an analytic mapping $A_1:\Omega_1\rightarrow\C$
satisfying
\begin{enumerate}
\item $A_1$ is univalent in $\Omega_1$.

\item $A_1 \circ f = T \circ A_1$.

\item If $A^{\prime}_1$ is another analytic mapping on $\Omega_1$
verifying conditions $1$ and $2$, then there exists $c\in \C$ such
that $A^{\prime}_1= A_1+c$ on $\Omega_1$.
\end{enumerate}
\end{lema}

Notice that an analogous lemma may be formulated for the existence
of the analytic mapping $A_2:\Omega_2\rightarrow \C$ with the
obvious modifications.

\begin{proof}
As already seen, $f\in \cal{A}$ is conjugated to a translation. What
this lemma asserts is that the conjugating homeomorphism is, in fact,
analytic.

The quotient space $R/T$ is conformally equivalent to the
punctured plane $\C^{\ast}$, where $\pi_0:R\rightarrow \C^\ast$ is
the projection. Since $G_1$ is a homeomorphic mapping satisfying
(\ref{eq25}) the orbits of $f$ are discrete and $\Omega_1/f=S$ is
a Riemann Surface. Let $\tilde{\pi}:\Omega_1\rightarrow S$ be the
projection of the quotient space. From (\ref{eq25}), we obtain a
quasiconformal homeomorphism $\tilde{G}:\C^\ast\rightarrow S$
between the quotient spaces.

{\it Claim $1$} : If there exists a conformal mapping
$B:S\rightarrow\C^{\ast}$ then there is an analytic mapping
$A_1:\Omega_1\rightarrow\C$ verifying $1$, $2$ and $3$.

Indeed, notice that $\Omega_1$ is a covering for $\C^{\ast}$ with
holomorphic projection $\pi = B \circ \tilde{\pi}$. However, the
universal covering of $\C^\ast$ is $\C$ with projection
$\hat{\pi}$. Since $\C$ is a simply connected covering of
$\C^\ast$ then there exists an inclusion
$A:\Omega_1\hookrightarrow \C$ such that $\pi=\hat{\pi}\circ A$.
This mapping satisfies $1$, $2$ and $3$ (details may be found in
\cite{Vo}).

{\it Claim $2$} : If there exists a quasiconformal mapping
$\tilde{G}:\C^\ast\rightarrow S$, then $S$ is conformally
equivalent to $\C^\ast$.

Since $\tilde{G}^{-1}:S\rightarrow U$ is a quasiconformal mapping
from $S$ to a domain $U$, the characteristic $h$ of its inverse is
independent on the choice of a local parameter on $S$, and
$\|h\|_{L_\infty(U)}=k<1$. Setting $h|_{\C\setminus U}$, $h$
remains measurable and $\|h\|_{L_\infty(U)}=k$. By
Theorem~\ref{t8} there exists a quasiconformal homeomorphism $F$
of $\C$ onto itself with characteristic $h$ and we may suppose
that $F(0)=0$. So that the map $B=F\circ
\tilde{G}^{-1}:S\rightarrow \C^{\ast}$ is conformal.
\end{proof}

Now we analyze the functions that change the sector, i.e. the
analogous to the functions $g_\pm$ in the saddle-node case.

Let $\Phi_+$ (resp. $\Phi_-$) be the restriction of $A_2 \circ
A_1^{-1}|_{V_+}$ (resp. $A_2 \circ A_1^{-1}|_{V_-}$), where
$V_+=\Omega_1\cap\Omega_2|_{ \{z:\;\Im (z)>0\}}$ and
$V_-=\Omega_1\cap\Omega_2|_{ \{z:\;\Im (z)<0\}}$ ($\Im (z)$ denotes
the imaginary part of $z$).

\begin{obs}
{\rm It is not obvious, but with some effort an expression for
$\Phi_\pm$ may be obtained. More precisely,
$\Phi_\pm(z)=z+\sum_{k\geq 0} c_k^\pm\exp(\pm 2 \pi i k z)$ (cf.
\cite{Vo}).}
\end{obs}

To each $f\in \cal{A}$ we may associate a pair $\Phi_\pm$ of
holomorphic functions. Since we are interested in the classes of
conjugacy of the elements of $\cal{A}$, it is natural to work out
the relations between the maps $\Phi_\pm$ and $\tilde{\Phi}_\pm$
associated to conjugate diffeomorphisms $f$ and $\tilde{f}$,
respectively. The advantage of the preceding construction lies in
the fact that the corresponding transition maps are related in a
particularly simple way. Indeed, by Lemma~\ref{l6} there exists
analytic mappings $A_1$ and $A_2$ such that:
\begin{eqnarray*}
A_1 \circ f &=& T \circ A_1\\
A_2 \circ f^{-1} &=& T^{-1} \circ A_2\, .\\
\end{eqnarray*}
By assumption, $f=H_0^{-1}\circ\tilde{f}\circ H_0$ so that
\begin{eqnarray*}
A_1\circ H_0^{-1}\circ\tilde{f} &=& T \circ A_1\circ H_0^{-1}\\
A_2 \circ H_0^{-1}\circ\tilde{f}^{-1} &=& T^{-1} \circ A_2\circ H_0^{-1} \\
\end{eqnarray*}
Suppose that $\tilde{A}_1$ is given by Lemma~\ref{l6} for
$\tilde{f}$. Then by item $3$ of this same lemma, it follows that
\[A_1\circ H_0^{-1}\equiv \tau_1\circ\tilde{A}_1\, .\]
Similarly,
\[A_2\circ H_0^{-1}\equiv \tau_2\circ\tilde{A}_2\, ,\]
where $\tau_1$ and $\tau_2$ are translations. Then
\begin{eqnarray*}
\Phi_{\pm}&=& A_2\circ A_1^{-1}|_{V_\pm}\\
&=&\tau_2\circ\tilde{A}_2\circ H_0\circ
H_0^{-1}\circ\tilde{A}_1^{-1}\circ\tau_1^{-1}\\
&=&\tau_2\circ\tilde{\Phi}_{\pm}\circ\tau_1^{-1}\, .
\end{eqnarray*}
In conclusion, we have obtained the following

\begin{teo}
\label{t15} If $f$ and $\tilde{f}$ belong to the same class of
analytic equivalence then the transition functions $\Phi_\pm$ and
$\tilde{\Phi}_\pm$ are conjugate by a translation.
\end{teo}

\begin{obs}{\rm The reader may notice that, starting from the
problem of classifying the saddle-nodes, we have iterated twice a
procedure of ``sectorial normalization''. This might give the
impression that no real progress was made in the second iteration
towards a concrete description of a suitable moduli space. This is
however not the case. In fact, in the saddle-node case, we have
used sectorial normalizations provided by the Hukuara-Kimura
Matuda Theorem to obtain diffeomorphisms $g^+$, $g^-$ where
conjugacy classes determine the analytic type of the saddle-node.
The difficulty is that $g^-$ is tangent to the identity but it is
{\it not} uniquely determined. Only the conjugacy class of $g^-$
in ${\rm Diff}(\C,0)$ is canonical in the sense that it is
uniquely determined. Thus, in a ``concrete'' comparison between
two saddle-nodes we would be reduced to tell whether or not two
``different'' diffeomorphisms tangent to the identity are
conjugate in ${\rm Diff}(\C, 0)$. The answer to this problem is by
no means obvious. To tackle this new question, we again applied
sectorial normalizations (this time provided by the Basic Lemma)
to obtain new functions $\phi_\pm$. The advantage of this second
normalization is that $\phi_\pm$ are ``almost uniquely
determined'', in the sense that two of them are conjugate by a
translation. In particular, it is easy to decide whether or not
two of them define the same point in the corresponding moduli
space.}
\end{obs}


\subsection{Reduction of Singularities in Dimension $2$}

So far we have only been analyzing vector fields with simple
singularities. However, as it will become clear throughout this
section, there is a particularly effective way of dealing with
higher order singularities. The Seidenberg Theorem basically
asserts that by composing a finite number of blow-up applications,
it is possible to reduce the order of an isolated singularity
until we only obtain simple ones. This section is devoted to
explain how it can be done.

The reader will notice that the present exposition is very
strongly inspired in the treatment given in [M-M], which, in turn,
is a blend between the original work of A. Seidenberg \cite{Se}
and that of Ven den Essen.

Let us now fix notations and give a few definitions. Suppose that
$\fol$ is a singular holomorphic foliation associated to the
holomorphic vector field $X$ having an isolated singularity at
$(0,0)$. We can also think in terms of $1$-forms. After all,
the $1$-form $\omega= Adx+Bdy$ and the vector field
$X(x,y)=B(x,y)\frac{\partial}{\partial
x}-A(x,y)\frac{\partial}{\partial y}$ define the same foliation
$\fol$. The {\it eigenvalues} of $\omega=Adx+Bdy$ at $(0,0)$ are
defined to be the eigenvalues $\lambda_1$, $\lambda_2$ of $X$ at
the same point.

Let $A_n$ (resp. $B_n$) denote the homogeneous component of degree $n$ of
the Taylor series of $A$ (resp. $B$) centered at the origin. Let $k$ be the
degree of the first non-trivial homogeneous component of the $1$-form
$\omega=Adx+Bdy$. The blow-up of $\omega$ in the coordinates $(x,t)$ is given by:
\[
\pi^\ast(\omega)=[A_k(1,t) + tB_k(1,t) + x (\tilde{A}(x,t) + t\tilde{B}(x,t))] dx + x [B_k(1,t) + x\tilde{B}(x,t)]dt\, .
\]
while in the coordinates $(t,y)$ it is given by:
\[
\pi^\ast(\omega) = y[A_k(t,1) + y\tilde{A}(t,y)] dt + [B_k(t,1) + tA_k(t,1) + y(\tilde{B}(t,y) + t\tilde{a}(t,y))] dy\,
\]
where
\[
\tilde{A}(x,t) = A_{k+1}(1,t) + xA_{k+2}(1,t) + x^2A_{k+3}(1,t) + \cdots=\frac{1}{x^k} \sum_{n>k}A_n(x,tx) \, .
\]
We adopt an analogous notation for $\tilde{B}$.

Let $J^1_{(0,0)}(\omega)=A_1(x,y) dx + B_1(x,y)dy$, and denote by
$C_\omega$ the subset of the exceptional divisor
$E \simeq \C\mathbb{P}(1)$ formed by the singularities of the
``new'' foliation $\widetilde{\fol}$ (i.e. the foliation
associated to $\tilde{\omega}=\pi^\ast\omega$). Notice that in
the non-dicritical case ($A_1(1,t)+tB_1(1,t)$ is {\it not}
identically zero) the set $C_\omega$ is given by the solutions
of $A_1(1,t)+tB_1(1,t)=0$. Denote by $\mu_c$ the order of the
singularity $c\in C_\omega$.

Suppose that $F$ is a mapping defined on $U\subset\C^2$ with a
singularity at $(0,0)\in \C^2$. The {\it order} of $F$ at $0$,
denoted by $\nu_0(F)$, is the degree of the first non-trivial
homogeneous component of the Taylor series of $F$ based at the
origin. We may also define the order of a vector field (as was
done in Section $2.4$). If $X(x,y)=(F(x,y), G(x,y))$ is singular
at $(0,0)$ the {\it order} of $X$ is the minimum between the
orders of $F$ and $G$. Analogously one may define the {\it order}
of a $1$-form $\omega$ at $0$.

Let $f, \, g$ be two polynomials defined on $\C^2$. Consider two
algebraic curves $V$ and $W$ associated, respectively, to $f$ and $g$, i.e.
consider
\[
V=\{(x,y)\in\C^2;\;f(x,y)=0\},\;\;\;\;\;\;\;\;\;\;W=\{(x,y)\in\C^2;\;g(x,y)=0\} \, ,
\]
and suppose that they intersect each other at $(0,0)\in \C^2$.
Recall that if $g$ is irreducible then there exists a local
normalization (Puiseux Parametrization)
\[
\gamma: t\mapsto (t^k, \sum_{n=m}^\infty a_n t_n) \, ,
\]
for a certain $k$ and $m<k$. We are now able to introduce the notion
of {\it intersection number} of two algebraic curves.

\begin{defi}
If $g$ is irreducible then the {\it intersection number} of $V$
and $W$ at $(0,0)$ is defined to be
\[
I(f,g;0)=\nu_0(f\circ\gamma) \, .
\]
If $g$ is not irreducible but $g = g_1^{\alpha_1}\ldots
g_p^{\alpha_p}$, where $g_1,\ldots,g^p$ are irreducible, then the
{\it intersection number} of $V$ and $W$ at $(0,0)$ is defined to
be
\[
I(f,g;0)=\sum_{i=1}^p \alpha_i I(f,g_i;0) \, .
\]
Finally, the {\it intersection number} of the $1$-form $\omega=
Adx + Bdy$ is defined to be
\[
I_0(\omega)=I(A,B;0) \, .
\]
\end{defi}

At this point we are able to state Seidenberg's Theorem.

\begin{teo}[Seidenberg]
\label{t10} Let $\fol$ be a singular holomorphic foliation
associated to the $1$-form $\omega= Adx+Bdy$ defined on an
open set $U \subseteq \C^2$ admitting
an isolated singularity at $(0,0)$. There exists a proper
analytic application $\pi:V\rightarrow U$ $($obtained as a
composition of blow-ups$)$ of a complex $2$-dimensional manifold $V$
onto $U$, such that:
\begin{itemize}
\item $\pi^{-1}(0,0)=E$, where $E$ is the exceptional divisor of
$V$;

\item $\pi:V \setminus E \rightarrow U \setminus \{(0,0)\}$ is a
holomorphic diffeomorphism;

\item $\nu_p(\pi^\ast(\omega))\leq 1$ for all $p\in V$

\end{itemize}
In other words, the proper transform of $\fol$ is defined on a complex manifold $V$
and is such that all of its points $p \in V$ are either regular or simple
singularities.
\end{teo}

\begin{proof}
Suppose that $(0,0)$ is the only singularity of $\omega$ in $U$.
Let $k$ be the order of $\omega$ at $(0,0)$. If $k>1$, we use the
blow-up procedure, described in Section $2.4$. In other words,
we consider the pull-back of $\omega$ by the proper mapping
$\pi: \wdc^2 \rightarrow \C^2$ given by $\pi(x,t)=(x,tx)$. Let
$\pi^{-1}(0,0)=E$ denote the exceptional divisor. In coordinates
$(x,t)$ the blow-up of $\omega$ is given by
\begin{equation}
\tilde{\omega} = [A_k(1,t) + tB_k(1,t) + x(\tilde{A}(x,t) + t\tilde{B}(x,t))]dx + x[B_k(1,t) + x\tilde{B}(x,t)]dt
\label{eq26}
\end{equation}
with $\tilde{A}, \, \tilde{B}$ as above.

As previously seen, the behavior of $\tilde{\omega}$ varies
significantly on the neighborhood of $E$ depending on whether or
not $A_k(1,t)+tB_k(1,t)$ is identically zero. By analyzing these
two different cases, it is not hard to obtain equalities relating
the intersection number of $\omega$ and its order at $(0,0)$
(refer to \cite{M-M} for details). More specifically,
\begin{itemize}
\item If $A_k(1,t)+tB_k(1,t)$ is \emph{not} identically zero, i.e. if the
exceptional divisor is invariant under the foliation, then
\begin{equation}
\label{eq27} I_0(\omega)=k^2-k+1+\sum_{c\in E}
I_c(\tilde{\omega})\, .
\end{equation}

\item If $A_k(1,t)+tB_k(1,t) \equiv 0$, then
\begin{equation}
\label{eq28} I_0(\omega)=k^2+k-1+\sum_{c\in E}
I_c(\tilde{\omega})\, .
\end{equation}
This corresponds to the dicritical case, i.e. to the case where the leaves
of the foliation $\tilde{\fol}$ associated to $\tilde{\omega}$ are regular
and generically transverse to $E$.
\end{itemize}

There are two possibilities, either the $1$-form $\tilde{\omega}$
has only simple singularities and regular points on $\pi^{-1}(U)$
and the theorem is proved, or there still are points $c\in\pi^{-1}(U)$,
such that $\nu_c(\tilde{\omega})>1$. So, let us assume that the second
case occurs. We first set
\[
V_1 = \pi^{-1}(U) \, ,\;\;\;\;\;\; \pi_1 = \pi|_{V_1} : V_1 \rightarrow U.
\]
Next, we simultaneously blow-up all the points $c\in E\subset V_1$
such that $\nu_c(\tilde{\omega})>1$. Let $\pi_2:V_2\rightarrow
V_1$ be the resulting application corresponding to these blow-ups.
We then define
\[
\pi^2 = \pi_1 \circ \pi_2 : V_2 \rightarrow U \, .
\]

Inductively we construct applications $\pi_i:V_i\rightarrow
V_{i-1}$, by blowing-up all the points $c$ of $V_{i-1}$ such that
$\nu_c((\pi^{i-1})^\ast (\omega))>1$, where
\[
\pi^{i-1} = \pi_1 \circ \cdots \circ \pi_{i-1} \, .
\]

The Equations~(\ref{eq27}) and~(\ref{eq28}) guarantee that this
procedure is finite. Indeed, despite the two different behaviors
the blown-up foliation $\widetilde{\fol}$ may assume, in both
cases the intersection number $I_c$ decreases if $k>1$. So if we
repeat this procedure sufficiently many times, there will be a
large enough $i$ such that for every $c$ on $V_i$,
$\nu_c((\pi^i)^\ast (\omega))\leq 1$.
\end{proof}

Now we shall treat the case of a foliation $\fol$ with only simple
singularities. By performing additional blow-ups it is possible to
obtain a simpler expression for the $1$-form $\omega=Adx+Bdy$
associated to $\fol$. This corresponding form cannot be further
simplified by extra blow-ups, and in this sense they are ``final''
or ``irreducible''.

\begin{teo}
\label{t11} Consider the singular holomorphic foliation associated
to the $1$-form $\omega= Adx+Bdy$ defined on $U\in\C^2$ with an
isolated singularity at $(0,0)$. Let $\pi:V\rightarrow U$ the
blow-up application obtained by Theorem~\ref{t10}. Then for every
$p\in V$ where $\nu_p(\pi^\ast(\omega))= 1$, there exists a
coordinate chart $(u,v)$ centered on $p$ such that
\begin{equation}\label{eq29}
\pi^\ast(\omega) = (\lambda_1 + {\rm h.o.t}) v du -
(\lambda_2 + {\rm h.o.t.}) u dv\, ,
\end{equation}
where $\lambda_1.\lambda_2\neq 0$, and $\lambda_1/\lambda_2$,
$\lambda_2/\lambda_1$ do not belong to $\N$, or
\begin{equation}\label{eq30}
\pi^\ast (\omega) = (v + {\rm h.o.t.}) du\, .
\end{equation}
\end{teo}

Let us begin our approach to Theorem~\ref{t11} by stating the following result:

\begin{lema}\label{l7}
Suppose that the $1$-form $\omega=Adx+Bdy$ is non-dicritical. If there
exists $c\in \C\mathbb{P}(1)$ such that $\mu_c=1$, then
$J^1(\tilde{\omega}_c) = \tilde{A}_1 dx + \tilde{B}_1 dy\neq 0$. In
other words, the blow-up of $\omega$ at $c$, $\tilde{\omega}_c$,
does not raise the order $\mu_{\pi^{-1}(c)}$. Moreover,
$\tilde{\omega}_c$ has a non-zero eigenvalue.
\end{lema}

\begin{proof}
We may suppose that the order
of $\omega$ at $(0,0)$ is $k=1$. According to the eigenvalues of
$\omega$, there are $5$ different situations that should be
analyzed separately:
\begin{enumerate}
\item $\lambda_1=\lambda_2=0$;

\item $\lambda_1=\lambda_2=\lambda\neq 0$;

\item $\lambda_1\neq\lambda_2$, $\lambda_1.\lambda_2\neq 0$ and
$\lambda_1/\lambda_2$, $\lambda_2/\lambda_1$ belong to $\N$ ;

\item $\lambda_1\neq\lambda_2$, $\lambda_1.\lambda_2\neq 0$ and
$\lambda_1/\lambda_2$, $\lambda_2/\lambda_1$ do {\it not} belong
to $\N$ ;

\item $\lambda_1\neq\lambda_2$, $\lambda_1.\lambda_2= 0$.

\end{enumerate}

Notice that the cases $4$ and $5$ correspond precisely to the
situations reflected by Equations~(\ref{eq29}) and~(\ref{eq30}),
respectively. Now we show that by blowing up the manifold on
certain points, the $3$ previous cases are reduced to cases $4$ or
$5$.

\textbf{Case $1$} \newline In this situation, the jacobian matrix
of the vector field $X=(B, -A)$ at $(0,0)$ is similar to
\[\left(%
\begin{array}{cc}
  0 & 1 \\
  0 & 0 \\
\end{array}%
\right).\]

Therefore, $J^1_{(0,0)}(\omega)=ydy$, so that the set $C_\omega$
is formed by the point $c=(1,0)$ and $\mu_c=2$.

Now we blow-up the foliation at $c$, obtaining
\[\tilde{\omega}_c=[t^2+x(t \tilde{b}(x,t)+\tilde{a}(x,t))]dx + x[t+x\tilde{b}(x,t)]dt\, ,\]

Suppose that $A_2(1,t)=\alpha_1+\alpha_2 t + \alpha_3 t^2$.

Notice that the order of $\tilde{\omega}$ may be $2$ or $1$
depending on whether or not $\alpha_1$ is equal to zero. We
analyze the two possibilities separately:
\begin{itemize}
\item \textbf{(1.a)} $\alpha_1\neq 0$

In this case, $\nu(\tilde{\omega})=1$ and to simplify the notation
we denote $\tilde{\omega}$ by:

\[\eta=[y^2+x(y \tilde{b}(x,y)+\tilde{a}(x,y))]dx + x[y+x\tilde{b}(x,y)]dy\, ,\]

Notice that $C_\omega$ is given by the equation
\[\alpha_1 x^2=0,\]

hence formed by the point $c=(0,1)$ with $\mu_c=2$. Now we blow-up
at $c$ and in the coordinate chart $(t,y)$, $t=x/y$ to obtain

\[\tilde{\eta}_c=y[\alpha_1 t + y(1+\cdots)]dt+[\alpha_1
t^2+y(2t+\cdots)]dy.\]

Therefore, $C_{\tilde{\eta}_c}$ is given by the equation
\[ty(3y + 2\alpha_1 t)=0,\]

and hence it contains only simple points. By applying
Lemma~\ref{l7}, we conclude that in these coordinates
$\lambda_1\neq 0$. Therefore we have reduced $\omega$ to a
$1$-form which does not belong to \textbf{Case $1$}.

\item \textbf{(1.b)} $\alpha_1= 0$

Here, $\nu(\tilde{\omega})=2$, and as before, denote
$\tilde{\omega}$ by $\eta$. The set $C_\eta$ is given by
\[x(2y^2+\gamma x y+\delta x^2)=x(y-a_1x)(y-a_2x)=0\, ,\]

for certain constants $\gamma$, $\delta$, $a_1$ and $a_2$.

\begin{itemize}
\item If $a_1\neq a_2$ we apply Lemma~\ref{l7} and the conclusion
is the same as before.

\item If $a_1=a_2=a\neq 0$ then
\[\tilde{\eta}= [(t-a)^2+x(\cdots)]dx + x[t+x(\cdots)]dt\, .\]

Hence the jacobian matrix associated to $\tilde{\omega}$ is $\left(%
\begin{array}{cc}
  \ast & \ast \\
  0 & 2a \\
\end{array}%
\right)$, therefore admitting a non-zero eigenvalue.

\item If $a=0$ then the point $c=(1,0)\in C_\eta$, has
multiplicity $\mu_c=2$ and
\[\tilde{\eta}_c=[2t^2+ x(\cdots)] dx+ x[t+x(\cdots)]dt \]

So, $\tilde{\eta}$ is still of the same type as $\eta$. Notice
that by Theorem~\ref{t10} $\nu(\tilde{\eta})=1$. Therefore this
possibility leads to \textbf{Case (1.a)}.
\end{itemize}

\end{itemize}

In conclusion, there always exists a way to blow-up $\omega$ so
that both of its eigenvalues are different from {\it zero}.

\textbf{Case $2$}

Without loss of generality we may suppose that $\lambda=1$. Hence
there are basically two possibilities to be considered:
\begin{itemize}
\item The jacobian matrix of the vector field $X$ at $(0,0)$ is
similar to
\[\left(%
\begin{array}{cc}
  1 & 0 \\
  0 & 1 \\
\end{array}%
\right)\]

In this situation, $J^1_{(0,0)}(\omega)=ydx-xdy$. It is therefore
the dicritical case. By blowing-up $\omega$ at each $c\in
\C\mathbb{P}(1)$, notice that $\tilde{\omega}_c$ has no
singularities ($\nu_p(\tilde{\omega}_c)= 0$ for all $p\in V$).

\item The jacobian matrix of the vector field $X$ at $(0,0)$ is
similar to
\[\left(%
\begin{array}{cc}
  1 & 1 \\
  0 & 1 \\
\end{array}%
\right)\]

Here, $J^1_{(0,0)}(\omega)=-xdx+(x+y)dy$, and the set $C_\omega$
is the point $c=(1,0)$ with $\mu_c=2$. By blowing-up $\omega$ at
$c$ we obtain:
\[\tilde{w}_c=[t^2+x(\cdots)]dx+x[1+t+x(\cdots)]dt\]

Notice that the jacobian matrix associated to $\tilde{\omega}$
at $c$ is equivalent to $\left(%
\begin{array}{cc}
  1 & 0 \\
  0 & 0 \\
\end{array}%
\right)$. This corresponds to \textbf{Case $5$}, so that $\omega$
is reduced to Equation~(\ref{eq30}).
\end{itemize}

\textbf{Case $3$}

In this case, $J^1_{(0,0)}(\omega)=-\lambda_2 ydx+\lambda_1 xdy$.
We may suppose that $\lambda_1=1$ and $\lambda_2=n$. So, $\omega$
may be written as
\[\omega=[-ny+(\cdots)]dx+[x+(\cdots)]dy\, .\]
Notice that there are exactly two points on $C_\omega$, namely
$c_1=(1,0)$ and $c_2=(0,1)$. Blowing-up $\omega$ at $c_1$ on the
coordinate chart $(x,t)$, $t=y/x$ we obtain
\[\tilde{\omega}_{c_1}=[t(1-n)+x(\cdots)]dx + x[1+x(\cdots)]dt\]
so that the jacobian matrix is:
\[\left(%
\begin{array}{cc}
  \lambda_1 & 0 \\
  \ast & \lambda_2-\lambda_1 \\
\end{array}%
\right)\thicksim \left(%
\begin{array}{cc}
  1 & 0 \\
  0 & n-1 \\
\end{array}%
\right)\]

Therefore, we have reduced $\omega$ to a $1$-form having
eigenvalues $\tilde{\lambda}_1=1$ and there are three
possibilities for $\tilde{\lambda}_2$:
\begin{itemize}
\item If $n=0$ then $\tilde{\lambda}_2=-1$ and the reduction of
$\omega$ belongs to \textbf{Case $4$}.

\item If $n=1$ then $\tilde{\lambda}_2=0$ and the reduction of
$\omega$ belongs to \textbf{Case $5$}.

\item If $n>1$ then $\tilde{\lambda}_2=n-1$ and by blowing-up
$n-1$ times, the reduction of $\omega$ belongs to \textbf{Case
$2$}.
\end{itemize}

Now let us analyze the blowing-up of $\omega$ at $c_2$ on the
coordinate chart $(t,y)$, $t=y/x$ where
\[\tilde{\omega}_{c_2}=y[-n+y(\cdots)]dt + [t(1-n)+y(\cdots)]dy\, .\]

Hence the jacobian matrix is given by:
\[\left(%
\begin{array}{cc}
  \lambda_1-\lambda_2 & \ast \\
  0 & \lambda_2 \\
\end{array}%
\right)\thicksim \left(%
\begin{array}{cc}
  1-n & 0 \\
  0 & n \\
\end{array}%
\right)\]

\begin{itemize}
\item If $n=0$ then $\tilde{\lambda}_1=1$ and
$\tilde{\lambda}_2=0$, hence the reduction of $\omega$ belongs to
\textbf{Case $5$}.

\item If $n=1$ then $\tilde{\lambda}_1=0$ and
$\tilde{\lambda}_2=1$, hence the reduction of $\omega$ belongs to
\textbf{Case $5$}.

\item If $n>1$ then $\tilde{\lambda}_1<0$ and
$\tilde{\lambda}_2>0$, hence the reduction of $\omega$ belongs to
\textbf{Case $4$}.
\end{itemize}
This completes the proof of the theorem.
\end{proof}


\subsection{Existence of First Integrals}

To illustrate the applications of the preceding techniques on
concrete problems about singularities of foliations, we shall
provide in this section a detailed proof of a fundamental result
due to J.-F Mattei and R. Moussu \cite{M-M} concerning the
existence of first integral for those singularities.

Recall that a holomorphic {\it first integral} of a holomorphic
foliation $\fol$ is a non-constant holomorphic function
$f:(\C^2,0)\rightarrow \C$ that is constant when restricted to the
leaves of $\fol$. If $\omega$ is a $1$-form defining $\fol$ then $f$
is such that $\omega\wedge df=0$. Equivalently, the leaves of
$\fol$ are the irreducible components of the level surfaces of
$f$.

This section is devoted to the topological characterization of
foliations admitting a holomorphic first integral, which is itself
the main result of [M-M].

Let $\fol$ be a holomorphic foliation defined on an open set $U\subset \C^2$,
with an isolated singularity at $(0,0)$. Suppose that $f$ is a
holomorphic first integral for $\fol$. Then the following
conditions hold:
\begin{enumerate}
\item Only a {\it finite} number of leaves of $\fol$ accumulates
on $(0,0)$.
\item The leaves of $\fol$ are closed as subsets of $U\setminus
\{(0,0)\}$.
\end{enumerate}

Indeed, if this function exists then the set $f^{-1}(0)$ is the set of
the separatrizes of $\fol$. Recall that a separatrix is simply an
irreducible analytic curve invariant by $\fol$, which contains the
singularity. Now $f$ can have only finitely many irreducible
factors, i.e. $f=g_1^{k_1}\ldots g_n^{k_n}H$ for some positive integer
constants $k_1, \ldots, k_n$ and holomorphic functions $g_1, \ldots, g_n, H$ such
that $H(0,0) \neq 0$. Hence the number of separatrizes must also be finite, as
these are given by the equations $\{g_i=0\}$. Clearly this is the
contents of Condition $1$.

Since $f$ is constant when restricted to the leaves, Condition $2$
also holds. In particular, a leaf $L$ of $\fol$ that is not
contained in the set $f^{-1}(0)$ cannot accumulate on a
separatrix. This last remark will often be used in the sequel.

Conversely, a foliation satisfying Conditions $1$ and $2$ admits a
first integral. This is the contents of the next theorem that
plays the main role in this section.

\begin{teo}[Mattei-Moussu \cite{M-M}]
\label{t12} Consider the holomorphic foliation $\fol$ defined on an
open set $U\subset \C^2$ with an isolated singularity at $(0,0)$. Suppose
that it satisfies the following conditions:
\begin{enumerate}
\item Only a {\it finite} number of leaves of $\fol$ accumulates
on $(0,0)$;
\item The leaves of $\fol$ are closed on $U\setminus \{(0,0)\}$.
\end{enumerate}
Then $\fol$ has a holomorphic non-constant first integral
$f:U\rightarrow \C$.
\end{teo}

Before proving the theorem, consider two holomorphic foliations
with isolated singularities, $\fol$ and $\fol '$ defined on a
neighborhood $U$ of the origin. We say that $\fol$ and $\fol '$
are {\it topologically equivalent} if there exists a homeomorphism
$h: U\rightarrow U$, such that $h(0,0)=(0,0)$ taking the leaves of
$\fol$ into the leaves of $\fol '$. It follows from the above
theorem that the existence of a holomorphic first integral obeys a
topological criterion: if $\fol$ admits one such first integral so
does $\fol '$.

Roughly speaking, the proof of Theorem~\ref{t12} consists of three
main steps. These are as follows:

\textbf{Part $1$}: If $\fol$ satisfies Conditions~$1$ and~$2$,
then, by performing successive blow-ups, the only irreducible
singularities found in the Seidenberg tree are those having two
non-zero eigenvalues $\lambda_1$ and $\lambda_2$ verifying in addition
$\lambda_1/\lambda_2\in \R_-$. Thus the singularities belong to the
Siegel domain.

\textbf{Part $2$}: Study of $\fol$ on a neighborhood of a
singularity in the Siegel domain. In particular the
characterization of those possessing a holomorphic first integral.

\textbf{Part $3$}: Extension of the local ``first integrals''
obtained in \textbf{Part $2$} to a neighborhood of the exceptional
divisor $E$. If $\tilde{f}$ denotes this extension, then
$f=\pi_\ast \tilde{f}$ is the first integral we were looking for.

We will give a detailed approach to each part described above in
the next $3$ sections. Together, they will conclude the proof of
the Mattei-Moussu Theorem.


\subsubsection{Part $1$: Analysis of Singularities in the Seidenberg
Tree of $\fol$}

Let us now begin by analyzing the irreducible singularities
obtained by the Seidenberg procedure. From now on we shall assume
that $\fol$ satisfies the assumptions of Theorem~\ref{t12}.

\begin{prop}
\label{l10} Let $\fol$ be a foliation satisfying Conditions $1$
and $2$. Then the only irreducible singularities that exist in the
Seidenberg tree of $\fol$ are those having non-zero eigenvalues
$\lambda_1$ and $\lambda_2$ such that the quotient
$\lambda_1/\lambda_2$ belongs to  $\R_-$.
\end{prop}

As seen in the previous section, every singularity of a given
foliation $\fol$ may be reduced by successive blow-ups so that the
blown-up foliation $\widetilde{\fol}$ only contains simple
singularities with eigenvalues $\lambda_1$ and $\lambda_2$ which
fall in one of the following cases:
\begin{itemize}
\item $\lambda_1.\lambda_2= 0$ where at least one between $\dl_1, \, \dl_2$
is distinct from zero (the case of a saddle-node singularity).

\item $\lambda_1.\lambda_2\neq 0$ with neither $\lambda_1/\lambda_2$ nor
$\lambda_2/\lambda_1$ being a positive integer.
\end{itemize}

Let us discuss the different possibilities separately.

\begin{lema}
The Seidenberg tree contains no saddle-nodes.
\end{lema}

\begin{proof}
Suppose for a contradiction that the statement is
false. As we have seen in previous sections there are local
coordinates $(x,y)$ about a saddle-node singularity where the
blown-up foliation $\widetilde{\fol}$ is given by
\[
\tilde{\omega}(x,y) = [x (1 + \lambda y^p) + y R(x,y)]dy- y^{p+1}dx \, .
\]
Note that $\{y=0\}$ is invariant by the blown-up foliation
$\widetilde{\fol}$. Next let us consider the holonomy application
of $\widetilde{\fol}$ associated to a small loop contained in
$\{y=0\}$ encircling the origin. We have already seen on
Section~\ref{sectionsaddlenode} that the holonomy is given
by $h(z)=z+z^{p+1}+\cdots$ for $z$ on some local transverse
section $\Sigma$.

From the dynamical structure of $h$ (in particular the existence
of Fatou coordinates) we see that there are infinitely many points
whose orbit under the iteration of $h$ accumulats on $0\in \C \simeq
\Sigma$. These points naturally correspond to distinct leaves of $\fol$
accumulating on the regular part of $\{y=0\}$. Now we have two
possibilities.

If $\{y=0\}$ is contained on the exceptional divisor, then it is
projected onto $(0,0)$. Hence there are infinitely many leaves
accumulating on the origin. This is not possible, for we are
assuming Condition~$1$.

On the other hand, if $\{y=0\}$ is transverse to the exceptional
divisor, then it is projected by $\pi$ as a separatrix (recall
that $\pi$ is proper). There are infinitely many leaves
accumulating on this separatrix, which again is impossible, since
it contradicts Condition~$2$.

In conclusion the reduction of $\omega$ does not contain any
saddle-node singularity.
\end{proof}

Next we shall discuss irreducible singularities having two
eigenvalues $\lambda_1$, $\lambda_2$ different from zero. Recall
that such a singularity is called {\it hyperbolic} if
$\lambda_1/\lambda_2\in \C\setminus \R$.

\begin{lema}
There is no hyperbolic singularity in the Seidenberg tree.
\end{lema}

\begin{proof}
By the Poincar\'{e} Linearization Theorem (cf. Section~\ref{sectionvfnzeigenvalues})
we may suppose that the vector field associated to
$\widetilde{\fol}$ is given in local coordinates by
\[
\widetilde{X}=\lambda_1 x \frac{\partial}{\partial
x}+\lambda_2y \frac{\partial}{\partial y}
\]
Again, $\{y=0\}$ is invariant by the foliation and, in
order to study the behavior of the leaves close to $\{y=0\}$, we
are going to consider its local holonomy.

Let $x(t) = \varepsilon e^{2 \pi i t}$ be a small loop around the origin
of $\C_x$. Let $\Sigma = \{(\varepsilon, y): y \in \C\}$ be a transversal
section to the leaf $\{y=0\}$ through the point $(\varepsilon, 0)$. The lift
of the loop to the leaf through the point $(\varepsilon, y) \in \Sigma$ is
described by the differential equation
\[
\frac{dy}{dt} = \frac{dy}{dx} \frac{dx}{dt} = 2\pi i\frac{\dl_2}{\dl_1} y
\]
whose solution is given by $y(t) = y e^{2 \pi i \dl_2/\dl_1 t}$.
Hence the holonomy map turns out to be $h(y)=e^{2\pi i \lambda_2/\lambda_1}y$.
However we have that $\Re(2\pi i \lambda_2/\lambda_1) \neq 0$ since $\lambda_2/\lambda_1$
is not in $\R$. Now by iterating $h$, we obtain
\[
h^n(y) = r^n \alpha^n y \, .
\]
where $r = e^{\Re(2\pi i \lambda_2/\lambda_1)}$ and $\da = e^{\Im(2\pi i \lambda_2/\lambda_1)}$.
The absolute value of $h^n(y)$ is determined by the value of $r$. More specifically, if
$r = e^{- 2\pi {\rm Im}(\lambda_2/\lambda_1)} < 1$ then $|h^n(y)|$ tends to zero as $n$
increases and, consequently the correspondent leaf accumulate on $\{y=0\}$. If $r > 1$ then
$|h^n(y)|$ goes to zero as $n$ goes to minus infinity. This also ensures that our leaf accumulate
on $\{y = 0\}$. However, Conditions~$1$ and~$2$ prevent this type of behavior
from happening and we conclude that the reduced foliation does not
contain this kind of singularities.
\end{proof}

\begin{lema}
The Seidenberg tree does not contain singularities with non-zero
eigenvalues such that $\lambda_1/\lambda_2\in \R_+$.
\end{lema}

\begin{proof}
By Theorems~\ref{PoincLintheo} and~\ref{t11}, the vector field
is still linearizable and $\{y=0\}$ is invariant by the
correspondent foliation. Notice that in this case the holonomy
application will not help us to conclude that this singularity
does not occur. In fact the holonomy map has the form $y\mapsto
e^{2\pi i \theta}y$ for some $\theta \in \R$, i.e. the holonomy
is conjugate to a rotation of ${\rm Diff}(\C,0)$.

However, by following the radial lines connecting a point $x_0$ on
the plane $\{y=0\}$ to the origin, we see that the neighboring
leaves accumulate on the origin. Indeed, fix a point $(x_0, y_0)
\in \C^2$ and consider the lift of the radial line $x(t) = x_0 e^{-t}$,
$t \in \R_+$, contained in $\{y=0\}$ to the leaf through $(x_0, y_0)$.
This lift is described by the differential equation
\[
\frac{dy}{dt} = \frac{dy}{dx} \frac{dx}{dt} = -\frac{\dl_2}{\dl_1} y
\]
whose solution is given by $y(t) = y_0 e^{-\dl_2/\dl_1 t}$. Since the
quotient $\dl_2 /\dl_1$ is positive it follows that $(x(t), y(t))$ goes
to $(0,0)$ as $t$ goes to infinity. Ultimately, this implies that there
are infinitely many leaves accumulating on $(0,0)$. Again, such
singularities do not appear on the blown-up foliation.
\end{proof}

Finally, we have reached the conclusion that we may assume that
the blown-up foliation $\widetilde{\fol}$ only has singularities
with eigenvalues such that $\lambda_1/\lambda_2\in \R_-$. Before
proceeding we only remark that:

\begin{lema}
All the irreducible components of the Seidenberg tree are
invariant by the corresponding foliation.
\end{lema}

\begin{proof}
If it were not so there would be infinitely many
separatrizes, contradicting our assumptions.
\end{proof}

\begin{obs}
{\rm The argument used in the preceding lemma can be extended to
show that a singularity has infinitely many separatrizes if and
only if by performing a finite sequence of blow-ups, we arrive to
an irreducible component of the total exceptional divisor which is
{\it not} invariant by the corresponding foliation. This type of
singularity is called {\it dicritical}.}
\end{obs}

The above lemmata leading to Proposition~\ref{l10}, corresponds to
the first part of our proof.


\subsubsection{Part $2$: Existence of Local First Integrals}

Let us take a closer look at the only singularities that appear in
the Seidenberg tree, i.e. those having eigenvalues $\dl_1, \, \dl_2$ such that
$\lambda_1/\lambda_2$ is in $\R_-$. We obtain from Lemma~\ref{Siegel}
that in local coordinates $(x,y)$, the vector field is given by:
\begin{equation}\label{eq32}
\lambda_1 x (1 + {\rm h.o.t.}) \partial /\partial x + \lambda_2 y (1 + {\rm h.o.t.}) \partial /\partial y \, ,
\end{equation}
for $\lambda_1/\lambda_2 \in \R_-$. First, what needs to be
studied is whether or not this vector field is linearizable. More
generally we would like to tell when two vector fields as above
(with $\lambda_1$, $\lambda_2$ fixed) are conjugate. The next
proposition due to  J.-F. Mattei and R. Moussu is of fundamental
importance for it allows us to restrict ourselves to
understanding whether or not the corresponding {\it holonomy} is
linearizable (conjugate).

\begin{teo}[Mattei-Moussu \cite{M-M}, \cite{M}]\label{TMM}
Assume that $\fol_1 ,\fol_2$ are given in coordinates
$(x,y)$ by Equation~(\ref{eq32}) with $\lambda_1 /\lambda_2 \in
\R_-$. Denote by $h_1$ (resp. $h_2$) the holonomy of $\fol_1$ (resp.
$\fol_2$) relative to the axis $\{y=0\}$. Then there is an
analytic diffeomorphism $\phi$ defined on a neighborhood of $(0,0)
\in \C^2$ and conjugating $\fol_1 ,\fol_2$ if and only if there is
an analytic diffeomorphism $\varphi$ defined on a neighborhood of
$0 \in \C$ and conjugating $h_1, h_2$.
\end{teo}

There are two markedly different cases to be considered according
to whether or not $\lambda_1/\lambda_2$ is rational. Firstly,
suppose that $\lambda_1/ \lambda_2=-m/n$ where $m,n \in
\N$. In this case, the holonomy of $\fol$ associated to a loop
around the origin in the separatrix $\{y=0\}$, is given by
\[
h(z) = e^{-2\pi i n/m}z + {\rm h.o.t.} \, ,
\]
where $z$ is the parameter on a local section transverse to
$\{y=0\}$. This application is linearizable if and only if
$h^m(z)=z$. It is obvious that if $h$ is linearizable then
$h^m(z)=z$. The converse is not obvious but can easily be verified
by noticing that $f=(\frac{1}{m}\sum_{i=0}^{m-1}
\lambda^{-i}h^i)^{-1}$ linearizes $h$. Note that when it is not
linearizable (i.e. when $h^m(z) = z + {\rm h.o.t.}$) we end up in the case
already studied. Namely, we end up on the dynamics of the cardioid-shaped
region analyzed on Section~\ref{sectionsaddlenode}. Similarly to the saddle-node
case, this cannot occur since otherwise we would have infinitely many
leaves accumulating on $\{y=0\}$, what contradicts our
assumptions.

Hence, we may suppose that $h$ is indeed linearizable. From
Theorem~\ref{TMM}, the vector field associated to
$\widetilde{\fol}$ admits the normal form, given in local
coordinates $(x,y)$ by
\begin{equation}\label{eq31}
\widetilde{X} = m x \frac{\partial}{\partial x} - n y \frac{\partial}{\partial y}\,.
\end{equation}

Let us now focus on the case where $\alpha=\lambda_1/ \lambda_2$
is irrational. As before, the holonomy application is given by
$h(z)= e^{2\pi i \alpha}z +\cdots$. However, since $\alpha$ is an
irrational number, the difficulty of knowing whether or not $h$ is
linearizable increases considerably. The problem of finding out
which are the irrational values of $\alpha$ that make the holonomy
$h$ linearizable is known as the ``Siegel Problem''. In fact, for
certain values of $\alpha$, it is possible to obtain local
diffeomorphisms $h$ as above that are non-linearizable.

\begin{obs}
{\rm Note that it is not difficult to obtain a {\it formal}
conjugacy $f$ that linearizes any such $h$. More precisely, by
formally solving $f^{-1}\circ h\circ f(z)=\lambda z$, one obtains
$f(z)=\sum^{+\infty}_{i=1} f_iz^i$, with $f_1=1$ and
\[
f_i=\frac{1}{\lambda^i-\lambda} \left[f_i + \sum^{i-1}_{p=2} f_p
\sum_{j_1+\cdots+j_p=i,\, h_{j_k}\geq 1}h_{j_1}\cdots h_{j_p}
\right]\, .
\]
}
\end{obs}

To deal with the irrational values of $\alpha$, let us first
consider the linearizable case. Then $h$ is conjugate to an
irrational rotation of ${\rm Diff}(\C,0)$. Let $C$ be the set
formed by the points of intersection between the leaf of
$\widetilde{\fol}$ passing by $(1, z_0)\in \C^2$ and the circle
$S=\{z\in \Sigma\,;\, |z|=z_0\}$. As always, $\Sigma$ is the
transverse section to $\{y=0\}$ passing by $1\in \C$. Since an
irrational rotation has dense orbits, it follows that $C$ is dense
over the circle $S$. In other words, the leaves are not locally
closed, which contradicts Condition~$2$. In particular, any
possible holomorphic first integral $f$, would be constant
everywhere. Indeed, the restriction of $f$ to $\Sigma$ is constant
over circles about $0\simeq \Sigma\cap\{y=0\}$. As a result $f$
would be constant everywhere.

We are then reduced to discuss the case in which $h$ is not
linearizable. To deal with this case we are going to show the
existence of leaves of $\widetilde{\fol}$ that are not closed on
arbitrarily small neighborhoods of $0\in \Sigma$.

Let us begin with a simple lemma attributed to J. Lewowicz but,
before stating it, we give some definitions that will be useful
throughout the text. Let $U$ be a neighborhood of $0$, where
$h\in{\rm Diff}(\C, 0)$ is defined, holomorphic, and injective.
Let $V$ be a subset of $U$ and fix a point $x \in V$.

\begin{defi}
The $V$-orbit of $x$ is the set of points on $V$, obtained by
iterating $h$ forward (i.e., $h^p=h\circ \cdots \circ h$, $p$
times) and backwards (i.e., $h^ {-p}=h^{-1}\circ \cdots \circ
h^{-1}$, $p$ times), along with $x$ itself (denoting by
$h^0(x)=x$). In other words,
\[O_V(x)=\{y\in V\,;\;\;y=h^i(x)\,,\;
i\in \Z\}\]
\end{defi}

The {\it number of iterations} of $x$ is the number of times $h$
is iterated taking $x$ to a point on $O_V(x)$. It is denoted by
$\mu_V(x)$ and belongs to $\N\cup \{\infty\}$.

Note that there may exist points $x$ on $V$ such that
$\mu_V(x)=\infty$ but $\# O_V(x)<\infty$. These points are called
{\it periodic} on $V$. Naturally, if $\mu_V(x)$ is finite then $\#
O_V(x)$ is necessarily finite.

\begin{defi}\label{d3}
An element $h\in {\rm Diff}(\C,0)$ is said to have
finite orbits if there exists an arbitrarily small open
neighborhood $V$ of $0$ where $h$ is defined, holomorphic,
injective and satisfies
\[
\# O_V(x)<\infty \, ,
\]
for every $x\in V$.
\end{defi}

\begin{lema}\label{l9}
Let $K$ be a compact connected neighborhood of $0\in
\R^n$ and $h$ a homeomorphism of $K$ onto $h(K) \subseteq \R^n$,
verifying $h(0)=0$. Then there exists a point $x$ on the boundary
$\partial K$ of $K$ such that the number of iterations in $K$ is
infinite, i.e. such that $\mu_K(x)=\infty$.
\end{lema}

\begin{proof}
Denote by $K \degree$ the interior of $K$ and let $\mu_K(x)$
(resp. $\mu_{K \degree}(x)$) denote the number
of iterations of $x$ on $K$ (resp. $K \degree$). Suppose, by
contradiction, that $\mu_K$ only attains finite values on the
boundary $\partial K$ of $K$. Since $K$ is compact, $\mu_K$ is
uniformly bounded on $\partial K$, i.e. there exists
$N\in \N$ such that
\[
\mu_K(x) < N<\infty \, \, \, \, , \forall x \in \partial K \, .
\]

Consider the sets
\begin{eqnarray*}
A & = & \{x \in K \, : \ ; \; \mu_K(x) < N \} \supset \partial K\\
B & = & \{x \in K \degree \, : \; \; \mu_{K \degree}(x) \geq N\} \ni 0
\end{eqnarray*}
which are non-empty open sets. In fact we can verify that $\partial K
\subseteq A$ and $0 \in B$.
Moreover, since $\mu_{K \degree}(x)\leq \mu_K(x)$ for all $x \in K$, those sets are
disjoint. Notice that there exists $x_0\in K$ such that $x_0$ does
not belong to $A\cup B$. Indeed, if it did not exist, then every
$x\in K$ would be such that $x\in A\cup B$, in other words,
$K=A\cup B$. The connectivity of $K$ would imply that $K=A$ or
$K=B$, which is impossible. Thus there exists $x_0$ such that
$\mu_K(x_0) \geq N>\mu_{K \degree}(x_0)$. It follows that the
orbit sets of $x_0$ on $K$ and on $K \degree$ are different
($O_K(x_0)\neq O_{K \degree)}(x_0)$). In other words, the orbit
of $x_0$ passes by $\partial K$. Let $y\in O_K(x_0)\cap\partial
K$, therefore, $\mu_K(y)=\mu_K(x_0)\geq N$, contradicting
$\mu_K|_{\partial K} < N$.
\end{proof}

Lewowicz's Lemma (Lemma~\ref{l9}) ensures the existence of points
whose orbit never leaves the neighborhood of $0\in \C$. However,
it does not exclude the possibility that all these points are
periodic. In this direction, the next proposition will be
fundamental.

\begin{prop}\label{p6}
If $h\in {\it Diff}(\C,0)$ is not periodic then there
exists open neighborhoods $U$ of $0$ on $\C$ such that $h$ is
holomorphic, injective and for each $U$, the set of points $x\in
U$ with infinite $U$-orbit is uncountable and $0$ is an
accumulation point.
\end{prop}

The basic idea behind its proof is to consider sets $U_n$, for
$n\in\N$ such that $U_n=\{x\in U;\,h(x), \cdots, h^n(x)\,\mbox{
are defined, and}\,\, h^n(x)=x\}$. If the domain of $h^n$ were
connected and $h^n \neq {\rm id}$ then each $U_n$ would be a
finite set. In particular, $\bigcup_{n\in\N} U_n$ would be
countable. By the lemma the set of points with infinite number of
iterations is uncountable, hence there would be an uncountable set
of points with infinite orbit. The difficulty with this argument
is that the domain of $h^n$ may be disconnected and $h^n$ may
coincide with the identity on one connected component and not on
the component containing $0$. This is why the proof of this
proposition is slightly more subtle as we need to consider the
various connected components of the domain of $h^n$.

\begin{proof}[Proof of Proposition~\ref{p6}]
Let $D_{\rho_0}$ be a closed disc centered at $0$ with radius
$\rho_0$. Using the same notations as in Lemma~\ref{l9}, we
define the following sets:
\begin{eqnarray*}
P & = & \{x \in D_{\rho_0} : \; \mu_{D_{\rho_0}}(x) = \infty , \; \#O_{D_{\rho_0}}(x) < \infty\}\\
F & = & \{x \in D_{\rho_0} : \; \mu_{D_{\rho_0}}(x) < \infty , \; \#O_{D_{\rho_0}}(x) < \infty\}\\
I & = & \{x \in D_{\rho_0} : \; \mu_{D_{\rho_0}}(x) = \infty , \; \#O_{D_{\rho_0}}(x) = \infty\}\\
\end{eqnarray*}
Naturally, $D_{\rho_0} = P \cup F \cup I$. Furthermore, Lemma~\ref{l9}
implies that for every $\rho\leq\rho_0$
\[
(P\cup I) \cap \partial D_\rho\neq \emptyset \, .
\]

{\it Claim}: One of the sets, either $P$ or $I$ (or both) is uncountable.

Indeed by Lemma~\ref{l9}, for every compact disc $D_\rho$ with
$\rho \leq \rho_0$ there is a point $x_0$ on its boundary with an
infinite number of iterations. Hence there are {\it uncountably}
many points $x_0$ in $D_{\rho_0}$ such that
$\mu_{D_{\rho_0}}(x_0)=\infty$. These points belong to $I\cup P$,
thus either $P$ or $I$ must be {\it uncountable}.

Under these notations, the contents of the proposition is that the
set $I$ is uncountable. Thus, it suffices to show that if $P$ is
uncountable then $I$ is necessarily uncountable as well. This is
what we shall do in the sequel.

Before continuing, let us define the following sets containing
$0$:
\[
A_1 = D_{\rho_0} , \; \; \; \; A_2 = D_{\rho_0} \cap h^{-1}(A_1) , \; \; \; \;
\cdots \; \; \; \; A_n = D_{\rho_0} \cap h^{-1}(A_{n-1}) , \; \; \; \; \cdots \, .
\]
Note that $A_n$ is precisely the domain of definition of $h^n$.

Let $C_n$ be the connected (compact) component of $A_n$ that
contains $0$ and let
\[
C=\bigcap_{n\in\N} C_n \, .
\]
$C$ is, in particular, connected.

\begin{lema}\label{l11}
We can assume, without loss of generality, that $C$ is countable.
\end{lema}

\begin{proof}Consider the case where $C$ is uncountable. Suppose
for a contradiction that $I$ is countable (and $P$ uncountable).
Then $I\cap C$ would also be countable. We consider now $C \cap P$
and note that it must be {\it uncountable}, otherwise $C$ would be
countable, since $C\subset P\cup I$. Let
\[
C\cap P = \bigcup_{n\in\N} P_n \, ,
\]
where $P_n$ is the set of points $x \in C \cap P$ of period $n$.

Note that there exists a certain $n_0\in \N$ such that $P_{n_0}$
is infinite, otherwise all of the $P_n$ would be finite and $C \cap
P$ would be countable. Being infinite, $P_{n_0}$ has a non-trivial
accumulation point in $C_{n_0}$. The application $h^{n_0}$ is
holomorphic on an open neighborhood $U_0$ of $C_{n_0}$ and it is
the identity on $P_{n_0} \cap U_0$. Since this set has an accumulation
point on $C_{n_0}$ then $h^{n_0}(z)=z$ on $U_0$, by the Identity Theorem.
By construction, $C_{n_0}$ contains the origin so that
$U_{n_0}$ is a neighborhood of $0\in\C$. This contradicts the
hypothesis of non-periodicity of $h$. Hence $I$ is
uncountable.
\end{proof}

In view of the preceding lemma, in the sequel we always assume
that $C$ consists of countably many points. First, note that
there exists $\rho<\rho_0$ such that $C\cap \partial D_\rho =
\emptyset$, otherwise $C$ would contain a point $x_0$ on the
boundary of $D_\rho$ for every $\rho<\rho_0$. This is obviously
impossible since $C$ countable. From now on let us fix one
such $\rho>0$. We note that, in particular, $C \subseteq D_{\rho}$.

Next, notice that the sets $C_1 \cap \partial D_\rho, (C_1 \cap
C_2) \cap \partial D_\rho, (C_1 \cap C_2 \cap C_3)\cap
\partial D_\rho, \ldots$, for $\rho$ as above, form
a decreasing sequence of compact sets. Hence the
intersection $\bigcap_{n\in\N} C_n \cap \partial D_\rho$ is
nonempty, unless there exists $n_0 \in \N$ such that $C_{n_0}\cap
\partial D_\rho=\emptyset$. The last case must occur, since $\rho$
was chosen so that $C \cap \partial D_\rho=\emptyset$.

Let $K$ be a compact connected neighborhood of $C_{n_0}$ that
does not intersect the other connected components of $A_{n_0}$, if
they exist. In particular one has $\partial K \cap
A_{n_0} = \emptyset$.

\begin{lema}\label{l14}
For every $x\in\partial K$ there exists $m\leq n_0$
such that $h^{m}(x)$ is not on $D_\rho$. Besides $\partial K\cap
P=\emptyset$.
\end{lema}

\begin{proof}
To verify the existence of $m$, suppose for a
contradiction that for every $m\leq n_0$, $h^m(x)$ belongs to
$D_\rho$, for $x\in
\partial K$. In this case, $x$ would have at least $n_0$ positive
iterations of $h$. This means that $x\in A_{n_0}$, what
contradicts the construction of $K$.

Moreover, if there existed a periodic point of $D_{\rho_0}$ on
$\partial K$, then $x$ would belong to every set $A_n$. In
particular it would belong to $A_{n_0}$, what is a
contradiction.
\end{proof}

Before continuing to prove that $I$ is uncountable, we define the
following sets:
\begin{eqnarray*}
P^{\prime} & = & \{x \in K ; \; \mu_K(x) = \infty , \; \#O_K(x) < \infty\}\\
F^{\prime} & = & \{x \in K ; \; \mu_K(x) < \infty , \; \#O_K(x) < \infty\}\\
I^{\prime} & = & \{x \in K ; \; \mu_K(x) = \infty , \; \#O_K(x) = \infty\}\, .\\
\end{eqnarray*}
Let
\[
P^{\prime}=\bigcup_{n\in\N} P'_n \, ,
\]
where $P^{\prime}_n$ is the set of the periodic points on $K$ with period $n$.

\begin{lema}\label{l12}
$P^{\prime}_n={\rm int}(P^{\prime}_n)\cup \partial P^{\prime}_n$, where ${\rm
int}(P^{\prime}_n)$ is open without boundary and $\partial P^{\prime}_n$ is
finite, consisting of isolated points.
\end{lema}

\begin{proof}
Let $p$ be the limit of a sequence of points in
$P^{\prime}_n$. Note that we do not assume that this sequence is
constituted by pairwise distinct points so that every point
belonging to $P^{\prime}_n$ automatically satisfies this condition.
Clearly, it suffices to show that there is a neighborhood of $p$
contained in $P^{\prime}_n$ in the case that the sequence is
not trivial, i.e. $p$ is an accumulation point of the sequence.
By definition, there exists an open connected neighborhood $W$ of
$p$ where $h^n$ is defined and holomorphic. Moreover $h^n$ is the identity
on $P^{\prime}_n\cap W$. Hence, if $p$ is an accumulation point of the
limit sequence of points~in $P^{\prime}_n$ then it $h^n(z)=z$ on $W$ from
the Identity Theorem. On the other hand, the $K$-orbit of $p$ does
not intersect $\partial K$ (cf. Lemma~\ref{l14}). Thus $W$ can be
chosen sufficiently small so that $h(W), h^2(W), \ldots,
h^n(W)\subset {\rm int}(K)$. This proves that $W\subset {\rm
int}(P^{\prime}_n)$. Hence $P^{\prime}$ consists on the union
of isolated points with ${\rm int}(P^{\prime}_n)$. Clearly the
number of isolated points must be finite, hence implying the lemma.
\end{proof}

\begin{lema}\label{l16}
The boundary of $P^{\prime}$ is countable $($consisting of the
union of the isolated points in $P^{\prime}_n$, for $n\in\N)$.
\end{lema}

\begin{proof}
Let us first consider $\bigcup_{n\in\N}{\rm int}(P^{\prime}_n)$.
Naturally it is an open set and we will prove that its
boundary is empty. Suppose, for a contradiction, that there exists
$p$ in the boundary of $\bigcup_{n\in\N}{\rm int}(P^{\prime}_n)$. Then there is
a sequence of points $\{p_k\}$ in $\bigcup_{n\in\N}{\rm int}
(P^{\prime}_n)$ converging to $p$. Fix a small disc
about $p$ and choose a point $p_k\in \bigcup_{n\in\N}{\rm
int}(P^{\prime}_n)$ belonging to this ball. Hence there is $n_0$ such that
$p_k\in P^{\prime}_{n_0}$. Clearly $p_k$ belongs to $K$ since
$P^{\prime}\subseteq K$. Furthermore, $p$ belongs to $K$ as well since
$K$ is closed. Finally, because $K$ is connected (by
construction), there is a path $c:[0,1]\rightarrow K$ joining
$p_k$ to $p$ (i.e. such that $c(0)=p_k$, $c(1)=p$). Since $p_k$ lies in
$P^{\prime}_{n_0}$ and $p$ does not, there is $t_0\in [0,1]$ such that
$c(t_0)$ belongs to the boundary of ${\rm int}(P^{\prime}_{n_0})$. This is
a contradiction to the previous lemma. Thus $\partial P^{\prime}$ is
countable as the union of the finite sets $\partial P^{\prime}_n$.
\end{proof}

Finally, we have one last simple lemma that will allow us to
conclude the proof.

\begin{lema}
\label{l15} $F^{\prime}$ is an open set of $K$.
\end{lema}

\begin{proof}
Note that, by definition, for points $x\in F^{\proof}$ we have
$\# O_K(x) = \mu_K(x) + 1$. Let
\[
F^{\prime} = \bigcup_{n\geq 0} F^{\prime}_n \, ,
\]
where $F^{\prime}_n$ is the set of points in $F^{\prime}$ such that $\mu_K(x)=n$.
We claim that $F^{\prime}_n$ are open sets. Indeed by continuity of $h$,
every point of $F^{\prime}_n$ has a neighborhood of points also having $n$
iterations on $K$. Since $F^{\prime}$ is the union of open sets, it is
itself an open set.
\end{proof}

To conclude the proposition we proceed as follows. Clearly we have
\[
K={\rm int}(P^{\prime})\cup F^{\prime}\cup( I^{\prime}\cup \partial P^{\prime}) \, .
\]
We shall analyze separately the following possibilities:
\begin{itemize}
\item Suppose that $F^{\prime}=\emptyset$.

Then $\partial K\subseteq {\rm int}(P^{\prime})\cup (I^{\prime} \cup \partial P^{\prime})$.
However, by Lemma~\ref{l14}, $\partial K$ does not contain
periodic points and so $\partial K\subseteq I^{\prime} \cup\partial P^{\prime}$.
Now $\partial P^{\prime}$ is countable, by Lemma~\ref{l16} and by
construction $\partial K$ in uncountable. Therefore we conclude
that $I^{\prime}$ is uncountable.

\item Suppose that $F^{\prime} \neq \emptyset$ and ${\rm
int}(P^{\prime})=\emptyset$.

In this case, we have that $K = F^{\prime} \cup (I^{\prime} \cup \partial P^{\prime})$. Note
that for sufficiently small values of $r>0$, the compact disks
$D_r$ are contained in $K$. And so, by Lewowicz's Lemma
(Lemma~\ref{l9}), we have that $\partial D_r\cap (I^{\prime} \cup \partial
P^{\prime})$ is non-empty. Therefore, $(I^{\prime} \cup \partial P^{\prime})$ is
uncountable, and as before, $I^{\prime}$ must be uncountable as well.

\item Suppose that $F^{\prime} \neq \emptyset$ and ${\rm
int}(P^{\prime}) \neq \emptyset$.

Note that ${\rm int}(F^{\prime})$ and ${\rm int}(P^{\prime})$ are non-empty
disjoint open sets of ${\rm int}(K)$. Thus $K\setminus ({\rm
int}(F^{\prime})\cup{\rm int}(P^{\prime}))$ is non-empty so it must be
uncountable. Naturally, it is contained on $(I^{\prime} \cup \partial P^{\prime})$
and so, $I^{\prime}$ must also be uncountable.
\end{itemize}

This completes the proof of Proposition~\ref{p6}.
\end{proof}

\begin{prop}\label{p8}
Suppose that the holonomy associated to
$\widetilde{\fol}$ and to a local transverse section $\Sigma$ is
given by $h=e^{2\pi i \alpha}z + {\rm h.o.t.}$, where $\alpha$ is
irrational. Then there exist leaves of $\widetilde{\fol}$ that are
not closed on arbitrarily small neighborhoods of $0$.
\end{prop}

\begin{proof}
It follows from the preceding proposition that there exists
a sequence of points $p_n = h^n(p_0) \in \Sigma$, pairwise
distinct and accumulating on a point $p \in \Sigma$. In fact
there exists an uncountable number of such sequences. Note
that each $p_n$ belongs to a same leaf $L$. We will show
that the leaf $L$ passing through $p_n$ is not closed.

Suppose that the accumulation point $p$ is in $L$, otherwise the
claim is obvious. Suppose that there exists an isolated point $q$
of $L\cap\Sigma$. Consider a path on $L$ connecting these two
points, by continuity of the holonomy application, there is also a
sequence of points of $L\cap\Sigma$ accumulating on $q$. In other
words, $L\cap \Sigma$ has no isolated points. Therefore, the
closure $\overline{L\cap \Sigma}$ of $L\cap \Sigma$ is uncountable
(being a perfect set). However, $L\cap\Sigma$ is countable, for
every leaf in a foliation can intersect a transverse section only
countably many times. Hence, $L\cap \Sigma$ has points of
accumulation that do not belong to it. In particular, $L$ is not
closed.
\end{proof}


\begin{prop}\label{p4}
Let $\fol$ be a foliation satisfying Conditions~$1$ and~$2$.
Then its Seidenberg tree only contains singularities that
admit a local first integral. In particular, the local holonomy of
their separatrizes is finite.
\end{prop}

\begin{proof}
Indeed, the preceding discussion allows us to
conclude that the only possibility which does not contradict the
hypothesis is when $\lambda_1/\lambda_2=-m/n$ is a rational
number. It follows from Theorem~\ref{TMM} that the vector field
associated to $\fol$ is holomorphically conjugate to a linear one,
locally given by Equation~(\ref{eq31}). Notice that its solution
is given by $\phi(T) = (x_0e^{mT}, y_0e^{-nT})$. Hence there exists a local first
integral. More precisely, $f(x,y)=x^ny^m$ is constant over the
orbits of the vector field.
\end{proof}

Now we go back to Theorem~\ref{TMM} and give an idea of the
basic principles behind its proof.

\begin{proof}[Proof of Theorem~\ref{TMM}]
Suppose that there exists a
holomorphic diffeomorphism $\varphi$ conjugating the holonomies
$h_1$ and $h_2$ relative to the foliations $\fol_1$ and $\fol_2$,
respectively and to a loop $S^1$ encircling the origin of a separatrix.
Let $D(\varepsilon)$ be a disk of radius $\varepsilon>0$ on a local
section $\Sigma$ transverse to the separatrix (which we can assume to
be given, in local coordinates, by $\{y=0\}$) at the point $(1,0)$, whose
image under $h_1$ is still contained in the previously chosen
neighborhood.

Recall that the vector field $X$, associated to the blown-up
foliation $\widetilde{\fol}$ is given by:
\[
X = \lambda_1 x (1 + {\rm h.o.t.}) \partial /\partial x + \lambda_2 y (1
+ {rm h.o.t.}) \partial /\partial y \, .
\]
Note that away from $\{x=0\}$, the leaves of the foliation
$\widetilde{\fol}$ are transverse to the vertical complex lines,
i.e. to the fibers of $\pi_1(x,y)=x$.

First, we define a holomorphic application $\phi$ on the solid
torus $S^1\times D(\varepsilon)$. This is done by lifting, with
respect to $\pi_1$ the path $\gamma$ along $S^1$ that connects $(1,0)$
to $(e^{i \theta_0},0)$, to a path $\gamma_1$ on the leaf $L_1$ of
$\fol_1$ that passes by $z \in \Sigma$. Note that this lift is
well-defined since the leaves are transverse to the fibers of
$\Pi_1$. We also lift, with respect to $\pi_1$, $\gamma$ to a path
$\gamma_2$ on the leaf $L_2$ of $\fol_2$ passing by $\varphi(z)$.
Now we extend the diffeomorphism $\varphi$ by declaring that
$\varphi$ takes the final extremity of $\gamma_1$ to the final
extremity of $\gamma_2$. We have to show that this extension of
$\phi$ is well-defined when $\gamma$ becomes a loop around
$0\in\{y=0\}$. This is however clear, since the final extremity of
$\gamma_1$ (resp. $\gamma_2$) is, by construction the image of $z$
by $h_1$ (resp. $h_2$) and we have $\varphi \circ
h_1 = h_2 \circ \varphi$. Thus we conclude that $\phi$ is well-defined
on $S^1\times D(\varepsilon)$. $\phi$ is moreover an extension of
$\varphi$, i.e. $\phi|_{1 \times D(\varepsilon)} = \varphi$.

Next we wish to extend $\phi$ holomorphically to a neighborhood of
$(0,0)$ and so that it still conjugates the foliations. Consider
them the radial lines $R_{\theta_0}$ connecting each point
$(e^{i\theta_0},0)$ on $S^1$ to the origin. Let $L_1$ be the leaf of
$\fol_1$ passing by $(e^{i \theta_0}, z)\in S^1\times
D(\varepsilon)$ and let $\gamma_{\theta_0}$ be the path over $L_1$ such
that $\Pi(\gamma_{\theta_0})=R_{\theta_0}$. Analogously, let $L_2$
be the leaf of $\fol_2$ that passes by $\phi(e^{i \theta_0}, z)$
and let $\eta_{\theta_0}$ be the path over $L_2$ that projects on
$R_{\theta_0}$ by $\pi_1$. We define $\phi$ to be the application
that takes $\gamma_{\theta_0}$ into $\eta_{\theta_0}$ by
preserving the projection $\pi_1$. By construction, $\phi$
is a holomorphic conjugation between $\fol_1$ and $\fol_2$ on its
domain. Note that the domain of $\phi$ is precisely the saturated
of $\Sigma$ by $\fol$ which is going to be denoted by
$\fol_\Sigma$.

However note that, {\it a priori}, as we follow the radial lines
approaching the origin, the union of $\fol_\Sigma$ and  $\{x=0\}$
may not contain a neighborhood of $(0,0)$.

To be more precise, let us revisit the construction of $\phi$
along the radial lines $R_{\theta_0}$. For example, consider this
construction over the radial line $R_0$ contained in $\{y = 0\}$
and joining the point $(1,0) = \Sigma\cap \{y=0\}$ to $(0,0)$. By
definition, we begin with a point $z_0\in\Sigma$ and consider the path
$\gamma_{z_0}:[0,1]\rightarrow L_{z_0}$, $\gamma(0)=z_0$, that
lifts only a segment of the radial line $R_0$ on $\{y=0\}$, going
from $(1,0)$ to a point $(q,0)$, with $q$ close to $0\in \C$. Let $\Sigma_{z_0}$ be
the set consisting of the final extremities, $\gamma_{z_0}(1)$ of
the paths $\gamma_{z_0}$ as above for every $z\in \Sigma$. In
particular, $\Sigma_{z_0}$ is contained in the vertical line
$\{x=q\}$. The difficulty to ensure that $\phi$ will lead to a
conjugacy defined around $(0,0)\in\C^2$ is related to the fact
that $\Sigma_{z_0}$ may not contain a uniform disc about $0\in \C$
as $q \rightarrow 0$.

In his manuscript \cite{M}, J.-F. Mattei estimates the size of the
sets $\Sigma_{z_0}$ as above (for all the radial leaves involved
and not only the example considered above of $R_0$). The
fundamental result in \cite{M} is the existence of a uniform
$\varepsilon>0$, such that every set $\Sigma_{z_0}$ contains the
ball of radius $\varepsilon$ about $0\in\C$. In particular, he
obtains:

\begin{prop}[Mattei \cite{M}]
\label{p7} Let $\fol$ be a foliation with eigenvalues $\lambda_1$
and $\lambda_2$, such that ${\rm Re}(\lambda_1/\lambda_2)<0$, then
$\fol_\Sigma \cup \{x=0\}$ contains a neighborhood of $(0,0)\in
\C^2$.
\end{prop}

The estimate of a uniform $\varepsilon>0$ as above is done by
integrating along $R_{\theta_0}$ the differential equation
inducing the foliation $\fol_1$ (resp. $\fol_2$). It is exactly at
this point that the assumptions $\lambda_1/\lambda_2\in \R_-$
plays its role.

This behavior is particularly clear when we consider {\it real}
ODEs with real eigenvalues $\lambda_1$ and $\lambda_2$ satisfying
$\lambda_1/\lambda_2\in\R_-$. We have that $|y(t)|$ increases as
$|x(t)|$ decreases, where $\gamma_{z_0}(t) = (x(t), y(t))$. Indeed
this would yield the classical picture of the phase space of a
saddle (cf. Figure~\ref{f4}).

\begin{figure}[h]
\begin{center}
\includegraphics[scale=0.9]{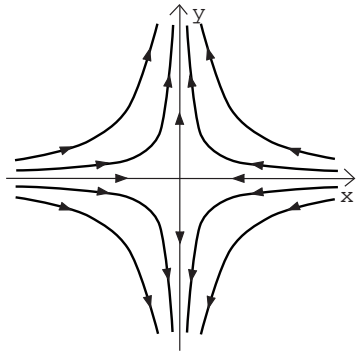}
\caption{Phase space of a saddle.} \label{f4}
\end{center}
\end{figure}

Note also that, if we consider the case
$\lambda_1/\lambda_2\in\R_+$, the statement would be false.
Indeed, with analogous notations, the sets $\Sigma_{z_0}$ would be
contained in balls with radii converging to {\it zero} when
$q\rightarrow 0\simeq (0,0)$.

Although Mattei's manuscript was not published, an extension to
higher dimension of this result can be found in \cite{R2} and will
be treated in the next section.

Let us now go back to the proof of Theorem~\ref{TMM}. Whereas
$\phi$ is holomorphic on its domain, this domain clearly does not
contain $\{x=0\}$. To extend $\phi$ to $\{x=0\}$ it is however
sufficient to note that $\phi$ is bounded by construction. On the
other hand, we already know that $\phi$ is defined on $U\setminus
\{x=0\}$, where $U$ is a neighborhood of $(0,0)\in\C^2$. Hence the
extension of $\phi$ to $\{x=0\}$ follows immediately from the
Riemann Extension Theorem. Thus the proof of Theorem~\ref{TMM}
is over.
\end{proof}

This concludes the second part of the proof of the Mattei-Moussu
Theorem.


\subsubsection{Part $3$: Extension to a Global First Integral}

In this part we wish to extend the local first integrals, obtained
in the previous section, to a global first integral defined on a
neighborhood of the exceptional divisor of the blown-up foliation.
Naturally this implies the existence of a first integral of the
original foliation, hence yielding Theorem~\ref{t12}.

To explain the idea of the construction suppose first that a
single blow-up application is sufficient to obtain reduced
singularities $p_1,\ldots, p_n$ on the exceptional divisor
$E\simeq \C\mathbb{P}(1)$. It follows from what precedes that
these singularities admit local first integrals. It will also
become clear that this case already encloses the main difficulties
of the general construction.

The basic idea is to extend the local first integrals around the
singularities along the leaves of the blown-up foliation. If it is
shown that such an extension is well-defined, then the first
integral will be defined on a neighborhood of the exceptional
divisor. Indeed, due to Mattei's estimate (Proposition~\ref{p7}),
the saturated of the local transverse sections $\Sigma_i$ by the
foliation contains a neighborhood of the singularities $p_i$
(except for the separatrizes at each $p_i$ that are transverse to
$E=\pi^{-1}(0,0)$). Consequently the extension would be defined on
a neighborhood of $\C\mathbb{P}(1)$, since they can again be
extended over the transverse separatrizes thanks to the Riemann
Extension Theorem.

To extend the mentioned first integral, it is necessary to study
the {\it projective holonomy} of $\widetilde{\fol}$. Namely, note
that $E \setminus \{p_1, \ldots, p_n\}$ is a regular leaf of $\widetilde{\fol}$.
The holonomy associated to this leaf, which is called the projective holonomy
of $\widetilde{\fol}$ (or of $\fol$), gives us a representation
\[
\rho: \Pi_1(E\setminus \{p_1, \ldots, p_n\})\rightarrow {\rm
Diff}(\C, 0)
\]
where $\Pi_1(E\setminus \{p_1, \ldots, p_n\})$ stands for the
fundamental group of $E\setminus \{p_1, \ldots, p_n\}$. In the
sequel we often make no distinction between the projective
holonomy of $\widetilde{\fol}$ and the subgroup of ${\rm
Diff}(\C,0)$ defined as the image of $\rho$.

\begin{lema}\label{l13}
The projective holonomy of $\widetilde{\fol}$ is
finite $($i.e. $\rho(\Pi_1(E\setminus \{p_1, \ldots, p_n\}))$ is a
finite subgroup of ${\rm Diff}(\C,0))$.
\end{lema}

\begin{proof}
Note that $\rho(\Pi_1(E\setminus \{p_1, \ldots,
p_n\}))$ is Abelian. Otherwise we can consider $h \neq {\rm id}$
having the form $h = f \circ g \circ f^{-1} \circ g^{-1}$ for $f$, $g\in
\rho(\Pi_1(E\setminus \{p_1, \ldots, p_n\}))$. Note however that
$h'(0)=f'(0)g'(0)(f'(0))^{-1}(g'(0))^{-1}=1$, i.e. $h$ is a
non-trivial difeomorphism tangent
to the identity. Then the local dynamics of $h$ is given by the
cardioid-shaped region (cf. Section~\ref{sectionsaddlenode}), and thus there are
infinitely many leaves of $\widetilde{\fol}$ accumulating on
$E=\pi^{-1}(0,0)$. These leaves produce infinitely many leaves of
$\fol$ accumulating on $(0,0)$ what is impossible.

To complete the proof note that $\Pi_1(E\setminus \{p_1, \ldots,
p_n\})$ is generated by small loops around the singularities
$p_i$, ($i=1,\ldots,n$). Hence $\rho(\Pi_1(E\setminus \{p_1,
\ldots, p_n\}))$ is generated by the local holonomies $h_i$
associated to those singularities. Moreover each $h_i$ has finite
order thanks to Proposition~\ref{p6}. Since $\rho(\Pi_1(E\setminus
\{p_1, \ldots, p_n\}))$ is Abelian, the statement follows.
\end{proof}

\begin{lema}\label{14}
The group $G = \rho(\Pi_1(E\setminus \{p_1, \ldots,
p_n\}))$ is indeed cyclic.
\end{lema}

\begin{proof}Consider the following group:
\[
G'=\{f'(0);\,\,f \in G\}
\]
and consider also the homomorphism $\alpha: G (\subset {\rm
Diff}(\C,0)) \rightarrow G' (\subset \C^\ast)$ associating to each
element $f$ of $G$ its derivative at $0$. We claim that the
homomorphism is one-to-one and hence a bijection onto its image.
Indeed, suppose that $f_1$ and $f_2$ are two distinct elements
of $G$ such that $\alpha(f_1)=\alpha(f_2)$. Then $f = f_1\circ
f_2^{-1}$ is a non-trivial diffeomorphism tangent to the identity,
i.e. such that $f'(0)=1$. In other words, $f$ has
infinite order, which is impossible because the group is finite.

Now let $f$ be an element of $G$. The derivative $f'(0)$ has modulus
equal to $1$ since $f$ has finite order. In other words, $G'$ is
actually a subgroup of $S^1$ identified with the complex number of
norm $1$. A discrete subgroup of $S^1$ is cyclic: to find a
generator it suffices to choose the element of ``smallest''
argument (different from zero). In particular, $G'$ is cyclic since
it is finite. Finally, because $\alpha$ induces a bijection
between $G$ and $G'$ we conclude that $G$ is cyclic as
desired.
\end{proof}

Before continuing, let us work out in detail a simple example of a
foliation having a global first integral $f$.
As before we are considering the case where a single blow-up of
the foliation gives us singularities $p_1,p_2,p_3$ in the
exceptional divisor $E$, all in the Siegel domain. They are
obtained by the intersection of $E$ with the proper transforms of
the separatrizes of $\widetilde{\fol}$ (which is simply the set
$f^{-1}(0)$). Suppose that the first integral is given by the
polynomial
\[
f(x,y)=x^{n_1}y^{n_2}(x-y)^{n_3} \, .
\]
with $n_1, \, n_2, \, n_3 \in \N$.

Note that the vector field $X$, associated to the foliation
$\fol$, having $f$ as its first integral is given by
\[
X = \frac{\partial f}{\partial y} \frac{\partial}{\partial x} - \frac{\partial f}{\partial
x} \frac{\partial}{\partial y}
\]
Indeed, the $1$-form associated to $X$ is $\omega= \frac{\partial
f}{\partial x} dx + \frac{\partial f}{\partial y} dy$ and, since
$\omega\wedge df=0$, we have that $f$ is a first integral of $\fol$.

Let us analyze the blown-up foliation on a neighborhood of its
singularity given by the intersection of the exceptional divisor
$\pi^{-1}(0,0)$ with the proper transform of the separatrix
$\{y=0\}$. In coordinates $(x,t)$, where $\pi(x,t)=(x,tx) = (x,y)$
the vector field is given by $x^{(n_1+n_2+n_3-2)}t^{(n_2-1)}(1-t)^{(n_3-1)} \widetilde{X}$
where
\begin{eqnarray*}
\widetilde{X} &=& (n_2x(1-t)- n_3xt) \frac{\partial}{\partial x} +
t(-n_1-n_2-n_3 + t(n_1+n_2+n_3))\frac{\partial}{\partial t}\\
&=& n_2x [1+ h.o.t] \frac{\partial}{\partial x} -
(n_1+n_2+n_3)t[1+h.o.t]\frac{\partial}{\partial t}\,.
\end{eqnarray*}

Thus the eigenvalues associated to $\widetilde{X}$ on a
neighborhood of the singularity $p_2$ at the separatrix $\{y=0\}$
are $\lambda_1=-(n_1+n_2+n_3)$ and $\lambda_2=n_2$. Recalling the
earlier discussion, we conclude that the local holonomy is given
by $h_2(z)=e^{-2\pi i n_2/(n_1+n_2+n_3)}z$. Analyzing the other
two singularities $p_1$ at the separatrix $\{x=0\}$ and $p_3$ at
$\{x=y\}$, by analogous calculations we see that their holonomies
are respectively given by $h_1(z)=e^{-2\pi i n_1/(n_1+n_2+n_3)}z$
and $h_3(z)=e^{-2\pi i n_3/(n_1+n_2+n_3)}z$.

The general case of a polynomial first integral given by $k$
separatrizes
\[P(x,y)=x^{n_1}y^{n_2}(x-y)^{n_3}(x-\alpha_4 y)^{n_4}\ldots (x-\alpha_k y)^{n_k} \]

is analogous. At each $p_i$ the holonomy is given by
$h_i(z)=e^{-2\pi i n_i/(n_1+\cdots+n_k)}z$, where $n_i$ is the
multiplicity of the separatrix that contains $p_i$. Note that we
have fixed $3$ separatrizes, i.e., $\{x=0\}$, $\{y=0\}$ and
$\{x-y=0\}$. This can always be done without loss of generality,
i.e., we may suppose that $3$ of the singularities of the blown-up
foliation are produced by these separatrizes. In fact, since ${\rm
PSL}(2, \C)$ acts transitively on triples of points in
$\C\mathbb{P}(1)$, i.e., there always exists a transformation
taking $p_1$, $p_2$ and $p_3$ on any other $3$ distinct points of
$\C\mathbb{P}(1)$. Hence the desired normalization can be obtained
by a linear change of coordinates.

Naturally, the vector field $X$ having $P$ as its first integral
is a homogeneous polynomial of degree $n_1+\cdots+n_k-1$. And so,
the foliation associated to $X$ is preserved by homotheties.
Indeed, one has $X(\lambda x, \lambda y)=
\lambda^{(n_1+\cdots+n_k-1)}X(x,y)$, for $\lambda\in \C^\ast$, so
that the ``direction'' determined by $X$ at the points  $(x,y)$
and $(\lambda x, \lambda y)$ is the same. This implies that the
holonomy maps $h$ associated to the leaf $E\setminus
\{p_1\ldots,p_k\}$ of $\widetilde{\fol}$ (i.e., the elements in
the projective holonomy group of $\fol$, $\widetilde{\fol}$)
commutes with homotheties. This can easily be established in
details as follows.

\begin{lema}
If $X(x,y)=F(x,y)\partial/\partial x+G(x,y)\partial/\partial y$ is
homogeneous, then the projective holonomy group $h$ is linear. In
other words, the elements of the group
$\rho(\Pi_1(\C\mathbb{P}(1)\setminus \{p_1, \ldots, p_k\}))$ are
linear applications of ${\rm Diff}(\C,0)$. In particular the whole
group is Abelian. \end{lema}

{\it Proof}\,. let $\fol$ be the foliation associated to $X$.
Since $X$ is homogeneous, the separatrizes of $\fol$ are radial
lines. In the complement of the separatrizes, the leaves of $\fol$
are transverse to the complex vertical lines of $\C$. The blown-up
vector field $\widetilde{X}$ is of the form

\begin{equation}
\label{eq41} \widetilde{X}= x F(1,t)\frac{\partial}{\partial x}+
(G(1,t)-t F(1,t)\frac{\partial}{\partial t}\,,
\end{equation}

We can parameterize the leaves of $\widetilde{\fol}$ by
$t\rightarrow (\varphi(t),t)$. Indeed, away from the proper
transforms of the separatrizes of $\fol$, $\widetilde{\fol}$ is
transverse to the Hopf fibration given in these coordinates by
$(x,t)\mapsto t$, where $t$ is the natural coordinate along the
exceptional divisor. Hence the fact that there is a
parametrization for the leaves of $\fol$, gives us the
differential equation:
\begin{equation}
\label{eq37} \varphi'(t)=A(t)\varphi(t)\, ,
\end{equation}

where $A(t)=\frac{F(1,t)}{G(1,t)-tF(1,t)}$. Note that the mapping
$\varphi_{t_0}:x\mapsto \varphi(t_0,x)$ where $t_0$ is fixed and
$x$ is taken as the initial condition is {\it linear}. In fact, if
$\varphi_1$, $\varphi_2$ are solutions of (\ref{eq37}), their sum
$\varphi_1+\varphi_2$ as well as the product $c\varphi_1$, $c\in
\C$, are also solutions of (\ref{eq37}). Hence,
$\varphi_{t_0}(x_1,x_2)=\varphi_{t_0}(x_1)+\varphi_{t_0}(x_2)$ and
$\varphi_{t_0}(ct_0)=c\varphi_{t_0}(x)$. It follows that
$\varphi_{t_0}(x)$ has the form $\lambda(t_0).x$. In particular,
the holonomy maps have to be linear. Indeed, the coordinate
$\partial/\partial t$ in (\ref{eq41}) depends only on $t$, so that
holonomy maps actually agree with ``time $t_0$'' flows.\qed

Now going back to our example, we have in addition that the
generators $h_1,\ldots, h_k$ are finite. Therefore, the projective
holonomy corresponds to the group of rotations of order
$n_1+\cdots+n_k$. Indeed, they can be chosen as the local
holonomies around the singularities of $\widetilde{\fol}$.

Let us consider our initial foliation $\fol$ satisfying Conditions
$1$ and $2$ of the Mattei-Moussu Theorem. Recall that we are
assuming that the singularities of the Seidenberg tree of $\fol$
can be reduced by performing a single blow-up. To fix notations,
suppose that $\widetilde{\fol}$ has $k$ reduced singularities,
$p_1,\cdots, p_k$.

Now consider the foliation $\fol_P$ that admits a polynomial first
integral $P$, such that $\widetilde{\fol}_P$ has precisely the
same singularities $p_1 ,\cdots, p_k$ as $\widetilde{\fol}$.
Recall that the foliation $\widetilde{\fol}_P$ admits a fibration
transverse to the exceptional divisor and such that the
separatrizes are fibers, which is simply the Hopf fibration,
realizing $\widetilde{\C}^2$ as a line bundle over
$\C\mathbb{P}(1)$ (projection along the proper transforms of the
radial lines of $\C^2$).

Naturally we can wonder if $\fol$ is actually conjugate to our
``model'' foliation $\fol_P$. If this were the true, then the
existence of a first integral for $\fol$ would follow as a
corollary. To obtain some evidence that this might be the case,
let us take the obvious identification of the exceptional divisor
$E$ (associated to the blown-up of $\fol$) with the exceptional
divisor of the blown-up of $\fol_P$, which will also be denoted by
$E$. Naturally, $E\setminus\{p_1,\cdots, p_k\}$ is a leaf of the
foliation $\widetilde{\fol}$. Moreover, recall that the projective
holonomy in both cases is cyclic, so that they actually coincide.

The problem of deciding whether or not $\fol$, $\fol_P$ are
conjugate is a special case of the study of moduli spaces for
holomorphic foliations. The next theorem is a small piece in this
direction.

\begin{teo}
\label{t13} Suppose that there is a fibration $\Pi$, transverse to
the exceptional divisor, and such that the proper transforms of
$\fol$ are fibers. Then the foliations $\fol$ and $\fol_P$ are
holomorphically conjugate.
\end{teo}

{\it Proof}\,. The assumptions on the existence a transverse
fibration on a neighborhood of the origin, allows us to lift paths
on $E\setminus\{p_1,\cdots, p_k\}$ to paths in the leaves of
$\widetilde{\fol}$. The method we use here is essentially the same
as the one used in the proof of Theorem~\ref{TMM}.

Let $h$ be a conjugacy between the projective holonomy groups of
$\widetilde{\fol}$ and $\widetilde{\fol}_P$ represented in
$\Sigma$. Given a point $z\in \Sigma$ and a path $\gamma\subseteq
E\setminus\{p_1,\cdots, p_k\}$ we can lift $\gamma$ with respect
to $\Pi$ to a path $\tilde{\gamma}$ contained in the leaf of
$\widetilde{\fol}$ passing by $z\in\Sigma$. Similarly $\gamma$ can
also be lifted with respect to the Hopf fibration into a path
contained in a leaf of $\widetilde{\fol}_P$.

To construct the conjugacy between $\widetilde{\fol}$ and
$\widetilde{\fol}_P$ we proceed as follows. Take $z\in \Sigma$ and
a path $\gamma\subseteq E\setminus\{p_1,\cdots, p_k\}$ as above.
Let $\tilde{\gamma}_1$ be the lift of $\gamma$ with respect to
$\Pi$ in the leaf of $\widetilde{\fol}$ passing by $z$. Similarly,
let $\tilde{\gamma}_2$ be the lift of $\gamma$ with respect to the
Hopf fibration in the leaf of $\widetilde{\fol}_P$ passing by
$h(z)$. We then impose that $H$ to take the endpoint of
$\tilde{\gamma}_1$ to the end of $\tilde{\gamma}_2$. This
application is well-defined because the projective holonomies are
conjugate. Moreover, thanks to Mattei's estimate
(Proposition~\ref{p7}), $H$ is defined on the neighborhood of each
singularity $p_i$, except for the separatrizes not contained in
the exceptional divisor. As before, we use the Riemann Extension
Theorem to define $H$ on the separatrizes.\qed

Taking this into account, the problem of finding a conjugacy as
desired, relies only on the existence of a fibration over
$E\setminus\{p_1,\cdots, p_k\}$ having the mentioned separatrizes
as fibers. However, the existence of this fibration can only be
guaranteed if the number of singularities is at most $4$. A
counter-example for the case of $5$ singularities can be found in
\cite{B-M-S}. This prevents us to deduce Theorem~\ref{t12} from a
strong result related to the moduli space of foliations. So, in
order to treat the general case we cannot rely on the conjugation
of $\fol$ and $\fol_P$. Fortunately, the existence the existence
of the fibration is an assumption much stronger than what is
really needed to solve the problem of obtaining a first integral.

Indeed, we do not really need an application that conjugates the
leaves of $\fol$ and $\fol_P$, rather, it suffices to find $\phi$
that is {\it constant} along the leaves of $\fol$ and defined on a
neighborhood of $E$. This is what we shall do in what follows.

{\it Proof of Theorem~\ref{t12} in the case where the foliation is
resolved with a single blow-up}\,. From now on we consider the
foliation $\fol$ along with its blow-up $\widetilde{\fol}$.
Consider again the transverse section $\Sigma$ for
$\widetilde{\fol}$. As it was seen, the projective holonomy of
$\widetilde{\fol}$ is finite and cyclic. Let $z$ be a local
coordinate in $\Sigma$ in which the generator of the holonomy of
$\widetilde{\fol}$ is given by
\[z\mapsto e^{2\pi i/m}z\,.\]
Next we consider the function $h:\Sigma\rightarrow \C$ given in
the above coordinates by $h(z)=z^m$. Clearly $h$ is invariant by
the holonomy group of the leaf $E\setminus\{p_1,\cdots, p_k\}$.
Now we extend $h$ to a function $H$ by imposing the following,
\begin{enumerate}
\item $H$ coincides with $h$ in $\Sigma$.

\item $H$ is constant over the leaves of $\widetilde{\fol}$
intersecting $\Sigma$.
\end{enumerate}

The invariance of $h$ by the projective holonomy group implies
that $H$ is well-defined. Now Mattei's estimate ensures that $H$
is defined on a neighborhood of $E$ minus the separatrizes. By
Riemann Extension Theorem a continuous extension of $H$ to the
separatrizes is automatically holomorphic. Finally we just need to
check that $H$ can be continuously extended to the separatrizes
setting $H=0$ over them. This follows immediately from observing
that a sequence of leaves $L_i$ of $\widetilde{\fol}$ accumulating
on the separatrix must accumulate on the exceptional divisor as
well. Obviously $h(0)=0$ so that $H=0$ over $E$ and the theorem
follows.\qed

Now we are ready to prove the general case of the Theorem of
Mattei and Moussu.

{\it Proof of Theorem~\ref{t12}}\,. We will show this for the case
where $2$ blow-ups are sufficient to reduce the singularity. It
will become clear that the procedure is the same for any number of
blow-ups. Consider a foliation $\fol$ such that its Seidenberg
tree only contains reduced singularities after two blow-ups. More
precisely, suppose that the exceptional divisor $E_1$ of the
blow-up $\widetilde{\fol}_1$ by $\pi_1$ at $(0,0)$ contains
singularities $p_1,\ldots, p_k$, such that the ratio of their
eigenvalues is a negative rational number, along with a single
degenerated singularity $q$.

Now we consider the blow-up $\pi_2$ of $q\in E_1$, resulting the
foliation $\widetilde{\fol}_2$. We must show that there is a
holomorphic application defined on a neighborhood of the
exceptional divisor $E_2=(\pi_1\circ\pi_2)^{-1}(0,0)$ of
$\widetilde{\fol}_2$, which is constant along its leaves. Denote
by $E$ the set $\pi_2^{-1}(p)$ (isomorphic to $\C\mathbb{P}(1)$),
which only contains reduced singularities, say $q_1,\ldots, q_l$.
Consider a local section $\Sigma_1$ transverse to $E_1$ close to
$p$, and the section $\Sigma_2$ on $E$. Consider the ``local'' (in
the sense that it is comes from the singularity $p$) holonomy $h$
associated to the leaf $E_2\setminus\{p_1,\cdots, p_k,q_1,\ldots,
q_l\}$ and a path with initial and final points, $c(0)$ and $c(1)$
at the intersection of $\Sigma_1$ with $E_1$, and ``encircling''
$E$. This group of holonomies is cyclic and finite, since there is
a natural correspondence between itself and the projective
holonomy group associated to $\Sigma_2$ and
$E\setminus\{q_1,\ldots, q_l\}$. The last is the cyclic group of
rotations as was seen earlier. The other local holonomies
associated to loop around each singularity $p_i$ are also finite
and cyclic. Thus the ``global'' holonomy, i.e., the group of
holonomies associated to the leaf $E_2\setminus\{p_1,\cdots,
p_k,q_1,\ldots, q_l\}$ of $\widetilde{\fol}_2$ and {\it any} path
on this leaf is a cyclic group of rotations. And so we may define
the first integral as was done previously and extend it to the
neighborhoods of the singularities by Mattei's estimate and
Riemann Extension Theorem.\qed



\section{Singularities in the Siegel domain}

Let us consider two holomorphic foliations $\fol_1 ,\fol_2$ on
$(\C^2, 0)$ with common non-vanishing eigenvalues $\dl_1, \, \dl_2$
in the Siegel, i.e. satisfying $\lambda_1 /\lambda_2 \in \R_-$.
There are local coordinates where the foliations have the form given
by Equation~(\ref{eq32}). Denote by $h_1$ (resp. $h_2$) the holonomy of
$\fol_1$ (resp. $\fol_2$) relative to the axis $\{y=0\}$. According to
Theorem~\ref{TMM} the foliations $\fol_1 ,\fol_2$ are analytically
equivalent if and only if the correspondent holonomies are analytically
conjugated. This section is devoted to prove a generalization of this
ressult for $1$-dimensional foliations in the Siegel domain satisfying
an open condition and defined on a higher dimensional ambient space.

Let $\fol$ be a foliation on $(\C^n, 0)$ and let $X$ be a representative
of $\fol$ ($X$ is a vector field whose singular set has codimension at
least equal to~$2$). Suppose that the origin is a singular point of $\fol$
and denote by $\dl_1, \ldots, \dl_n$ the eigenvalues of $DX$ at the correspondent
point. Assume that
\begin{enumerate}
\item $\fol$ has an isolated singularity the origin
\item the singularity of $\fol$ is of Siegel type
\item the eigenvalues $\dl_1, \ldots, \dl_n$ are all distinct from zero and
there exists a straight line through the origin, in the complex plane, separating
$\dl_1$ from the other eigenvalues
\item up to a change of coordinates, $X = \sum_{i=1}^n \dl_ix_i(1+f_i(x))
\partial /\partial x_i$, where $x=(x_1,\ldots,x_n)$ and $f_i(0)=0$ for all $i$
\end{enumerate}
On $4.$ we are assuming the existence of $n$~invariant hyperplanes
through the origin.

The analogous of Mattei and Moussu's result for 1-dimensional foliation on higher
dimensional spaces can be stated in the following way:

\begin{teo}\label{TMMhigher}
Let $X$ and $Y$ be two vector fields verifying conditions $1., 2., 3.$ and $4.$.
Denote by $h_1^X$ (resp. $h_1^Y$) the holonomy of $X$ (resp. $Y$) relatively to
the separatrix of $X$ (resp. $Y$) tangent to the eigenspace associated to the
first eigenvalue. Then $h_1^X$ and $h_1^Y$ are analytically conjugated if and only
if the foliations associated to $X$ and $Y$ are analytically equivalent.
\end{teo}

Note that if $\fol$ is a foliation on a 3-dimensional space and with an
isolated singularity at the origin of strict Siegel type (i.e. such that $0 \in \C$ is
contained in the interior of the convex hull of $\{\dl_1,\ldots,\dl_n\}$) then conditions
$3.$ and $4.$ are immediately satisfied (cf. \cite{C}).

The rest of this section is devoted to the proof of this result. The proof of
this theorem can be found in either in \cite{EI} as also in \cite{R2}. Here we
follow the approach presented in \cite{R2}. This proof is more detailed than the
one presented in the correspondent paper.

The proof of Theorem~\ref{TMM} is based on the following fact. Let $\Sigma$ be a transverse
section to one of the separatrizes of $\fol$ (as in Theorem~\ref{TMM}). Then the union of
the saturated of $\Sigma$ by $\fol$ with the other separatrix contains a neighborhood of the
singular point (cf. Proposition~\ref{p7}). In order to prove Theorem~\ref{TMMhigher} we shall
obtain the correspondent to Proposition~\ref{p7} for foliations on a higher dimensional space:

\begin{prop}\label{p7higher}
Let  $X$ be a holomorphic vector field verifying conditions $1.$, $2.$, $3.$ and
$4.$. Denote by $S$ the separatrix tangent to the eigenspace associated to
$\dl_1$. Then the union of the saturated of a transverse section to $S$ by $\fol$ $($the
foliation induced by $X)$ with the invariant manifold transverse to $S$
contains a neighborhood of the origin.
\end{prop}

Note that the existence of an invariant manifold transverse to the mentioned separatrix
is a simple consequence condition $4.$.

\begin{proof}
Assume that $X$ is written in coordinates $(x_1, \dots, x_n)$ in the form
\begin{equation}\label{normalform}
X=\sum_{i=1}^n \dl_ix_i(1+f_i(x)) \dfrac{\partial }{\partial x_i}
\end{equation}
Let us normalize the vector fields assuming that $\dl_1 = 1$.

Note that the $x_1$-axis corresponds to the separatrix associated to the eigenvalue
that can be separated from the others by a straight line through the origin.
So fix a positive real constant $\de$ sufficiently close to zero and let
$\Sigma = \{(\de, x_2, \ldots, x_n) :|x_i| \leq \de, \forall i = 2, \ldots, n\}$
be a transverse section to the $x_1$-axis at the point $c(0)$, where $c:[0,2\pi]
\rightarrow \C^n$ is the curve defined by $c(\theta)=(\de e^{i\theta}, 0, \ldots, 0)$.
Since the domain of $c$ is compact, there exists $0< \delta< \de$ such that
$\Sigma_{\theta} = \{(\de e^{i\theta}, x_2, \ldots, x_n) :|x_i| \leq \delta,
\forall i=2, \ldots, n\}$ is contained in the saturated of $\Sigma$.

So let $l$ be, in the complex plane, a straight line through the origin
separating $\dl_1$ from the other eigenvalues. Consider the straight line
through the origin that is orthogonal to $l$ and denote be $L$ the part of this
straight line that is contained in the left half-plane. Finally, denote by $\bar{L}$
the complex conjugate of $L$ (cf. figure \ref{fig1}). Suppose that $v = \da + i \db$,
for a certain $\da > 0$, is a directional vector of $L$. Then let
\[
T=\{z \in \C: z = x + iy, x \in \bar{L}, -\pi < y \leq \pi\} \, .
\]
It is easy to verify that the image of T by the application map $\phi(z)=\de
e^{z}$ covers $\{z: |z| \leq \de\} \setminus \{0\}$. This map is moreover
one-to-one.

Fix $z \in T$ and consider the straight line parallel to $\bar{L}$ passing through
the fixed point. More specifically we consider the line segment contained in
this straight line between its intersection with the imaginary axis and $z$.
This line segment can be parametrized by:
\[
c_z(t) = z + \frac{1}{\da+i\db}t \, \, \, \left(= z + \frac{1}{v}t \right)
\]
The domain of $c_z$ is the interval $[0,t_z]$ were $t_z$ denotes the instant for
which the image of $c_z$ intersects the imaginary axis.

Fix $x_1 \in \C$ such that $|x_1|\leq \de$ and denote by $z = z(x_1) \in T$ the
value for which $\de e^z=x_1$. Consider the logarithmic spiral curve contained in
the $x_1$-axis given by $r_{x_1}(t)=(\de e^{c_z(t)},0,\ldots,0)$ for $t \in [0,t_z[$.
Note that this curve passes through the point $(x_1, 0, \ldots, 0)$. Denote now by
$r_{x}$ the lift of $r_{x_1}$ to the leaf through the point $x=(x_1,\ldots,x_n)$.

The spiral curve $r_{x_1}$ satisfies $r_{x_1}(0)=(x_1,0,\ldots,0)$ and also
$|r_{x_1}(t_z)| =\de$. Consequently, $r_x$ is such that $r_x(0)=x$ and
$|p_1(r_x(t_z))|=\de $, where $p_i$ denoted the projection on the $i^{th}$
component, i.e. $p_{i}(x)=x_i$. In order to simplify the notation we simply
denote by $k_x$ the value of $t_{z(x)}$.

\begin{figure}[ht!]
\centering
\includegraphics[scale=0.7]{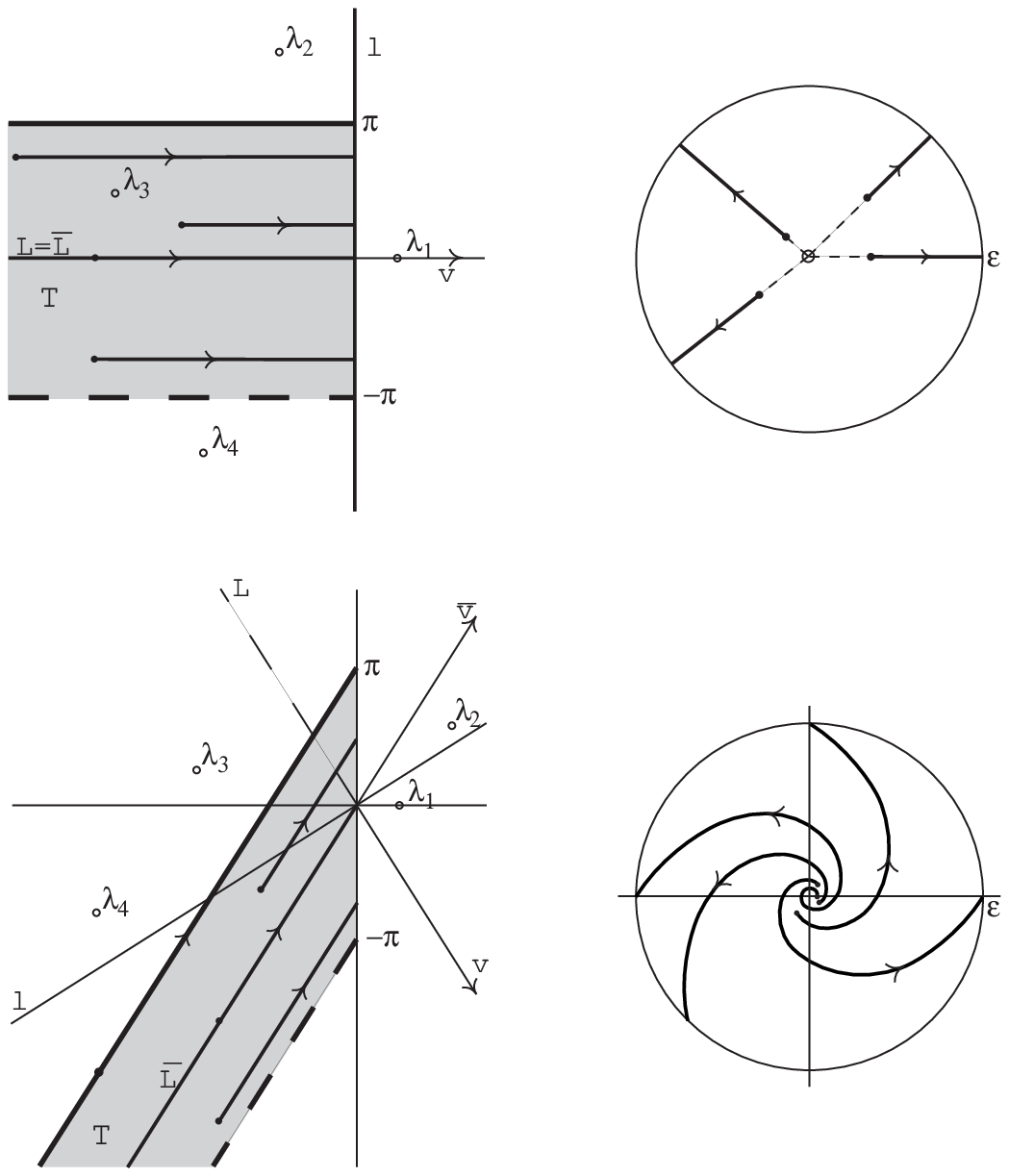}
\caption{}\label{fig1}
\end{figure}

Denote by $\mathfrak{S_n}$ the set of vector fields of type (\ref{normalform})
verifying the conditions $1.$ to $4.$. We have:

\begin{lema}\label{extensao}
Fix $X \in \mathfrak{S}_n$ and let $V$ denote the set of points $x$ in a small
neighborhood of the origin for which $r_x$ intersects $\Sigma_{\theta}$, for some
$\theta \in [0, 2\pi]$. Then $V \cup \{x_1=0\}$ contains a neighbourhood of the origin.
\end{lema}

\begin{proof}[Proof of Lemma~\ref{extensao}]
Naturally there exists a positive real number $\de < 1$ for which the projection
$p_1$ is transverse to the leaves of $\fol$ on a neighborhood of the polydisk
\[
P_{\de,\delta}=\{x \in \C^n: |x_1| \leq \de, |x_i| \leq \delta,
i \geq 2\}
\]
for $\delta$ described as above.

Fix $x_1^0 \ne 0$ such that $|x_1^0| \leq \de$ and denote by $z$ the complex number
for which $\de e^z = x_1^0$. The differential equation associated to $X$
restricted to $x_1=\de e^{c_z(t)}=x_1^0e^{\frac{t}{v}}$ is
equivalent to the system of differential equations:
\begin{equation}\label{eqdif}
\begin{cases}
\dfrac{dx_2}{dt}=\dfrac{\dl_2}{v} x_2(1+A_2(x_1^0
e^{\frac{t}{v}},x_2,\ldots,x_n))\\
\quad \vdots\\
\dfrac{dx_n}{dt}=\dfrac{\dl_n}{v} x_n(1+A_n(x_1^0
e^{\frac{t}{v}},x_2,\ldots,x_n))
\end{cases}
\end{equation}
where $A_i$ are holomorphic functions satisfying $A_i(0)=0$, for all
$i \geq 2$. The constant $\de$ can also be chosen in such a manner that
\begin{equation}\label{ab}
| A_i(x)| \leq
\dfrac{\left|\Re
\left(\dfrac{\dl_i}{v}\right)\right|}{2\left|\dfrac{\dl_i}{v}
\right|}, \, \, \, \forall i\geq 2
\end{equation}
in $P_{\de,\delta}$.

\begin{obs}
Since the eigenvalues $\dl_2,\ldots,\dl_n$ are all contained in the same half plane
defined by the straight line $l$ and not containing the direction of $v$ it follows
that the angle between each one of this eigenvalues and $v$ is greater than $\pi/2$.
This implies that $\dl_i/v$ has negative real part.
\end{obs}

Fix $x^0 \in P_{\de,\delta}$ such that all of its components $x_i^0$ are distinct from
zero. It is going to be proved that the solution of the differential equation~(\ref{eqdif})
satisfies $x(t) \in P_{\de,\delta}$ for all $t \in [0,k_{x_1^0}]$.

First of all denote by $p$ the projection on the $n-1$ last components, i.e. in coordinates
$x = (x_1, x_2, \ldots, x_n)$ we have $p(x) = (x_2, \ldots, x_n)$. Note that $p(r_{x^0}(t))$
corresponds to the solution of the differential equation~(\ref{eqdif}) with initial data
$(x_2^0,\ldots,x_n^0)$. Moreover it satisfies $|p_1(r_{x^0}(t))|=\de$ if and only if
$t = k_{x_1^0}$. This is still valid if we restrict ourselves to points $x$ such that
$x_i = 0$ for some $i \geq 2$. Note that $\{x_i=0\}$ is invariant for the foliation.

Integrating the system~(\ref{eqdif}) we obtain
\[
{\rm Log}\left(\dfrac{x_i(t)}{x_i^0}\right)=
\left(\dfrac{\dl_i}{v}t+ \int_0^t \dfrac{\dl_i}{v}A_i(x_1^0
e^{\frac{s}{v}},x_2(s)\ldots,x_n(s))ds)\right)
\]
for all $i \geq 2$. Since ${\rm Log}(w)=\log{|w|}+i \arg{w}$ we have that
\[
\log{\left| \dfrac{x_i(t)}{x_i^0}\right|}=\Re \left(\dfrac{\dl_i}{v}t+
\int_0^t \dfrac{\dl_i}{v} A_i(x_1^0
e^{\frac{s}{v}},x_2(s),\ldots,x_n(s))ds
\right)
\]
and consequently
\[
|x_i(t)| = |x_i^0|e^{\Re (\frac{\dl_i}{v})t +\Re( \frac{\dl_i}{v}
\int_0^t A_i(x_1^0 e^{\frac{s}{v}},x_2(s),\ldots,x_n(s))ds)}, \quad
\forall i \geq 2
\]

However
\begin{align*}
&\left|\Re \left(\dfrac{\dl_i}{v} \int_0^t A_i(x_1^0
e^{\frac{s}{v}},x_2(s),\ldots,x_n(s))ds
\right)\right|\\
&\leq \left|\dfrac{\dl_i}{v} \right| \int_0^t
|A_i(x_1^0 e^{\frac{s}{v}},x_2(s),\ldots,x_n(s))|ds\\
&\leq \left|\dfrac{\dl_i}{v} \right| \dfrac{\left|
\Re \left(\dfrac{\dl_i}{v} \right)
\right|}{2\left|\dfrac{\dl_i}{v} \right|}t=-\Re \left(\dfrac{\dl_i}{v}
\right)\dfrac{t}{2}
\end{align*}
and consequently we conclude that
\begin{align*}
|x_i(t)| &\leq |x_i^0|e^{\Re (\frac{\dl_i}{v})t
+|\Re(\frac{\dl_i}{v}
\int_0^t A(x_1^0 e^{\frac{s}{v}},x_2(s),\ldots,x_n(s))ds)|}\\
&\leq |x_i^0|e^{\Re (\frac{\dl_i}{v})t -\Re
(\frac{\dl_i}{v})\frac{t}{2}}\\ &\leq |x_i^0|
\quad (\leq
\delta)
\end{align*}
for all $t \in [0,k_{x_1^0}]$, since $\Re (\frac{\dl_i}{v})$ is a negative
real number.
\end{proof}

Condition $4.$ has been used in the proof of this lemma. Although the
condition $4.$ is always verified by vector fields on $\C^3$ of strict Siegel
type, it is not the case for vector fields defined on higher dimensional manifolds
(cf. \cite{C}). The condition is also pertinent for vector fields on $\C^3$ whose
singular point belongs to the boundary of the Siegel domain. In fact there exist
vector fields of that type, verifying $1.$, $2.$ and $3.$ but not admitting any
holomorphic manifold tangent to the formal invariant manifold $\{z=0\}$ (cf.
\cite{can}).

\begin{obs}
{\rm Note that in the case that the angles between $\dl_1$ and the remaining eigenvalues
are all greater than $\pi/2$ we have that the spiral curves $r_{x_1}$ are in fact
radial.}
\end{obs}

Let us now finish the proof of Proposition~\ref{p7higher}.
Assume that in coordinates $(x_1, \ldots, x_n)$ the vector field $X$ is written
in the form~(\ref{normalform}). Let us consider the transverse sections to $S$
defined above and given by $\Sigma_{\theta}=\{(\de e^{i\theta}, x_2, \ldots, x_n):
|x_i| \leq \delta, \forall i\geq 2\}$. We have that the saturates of $\Sigma = \Sigma_0$
contains $\Sigma_{\theta}$ for all $\theta \in [0, 2\pi]$.

Fix a point $x \in P_{\de,\delta}$ whose first component does not vanish, i.e. such that
$x_1 \ne 0$. By Lemma \ref{extensao} the lift of $r_x$ belongs to $P_{\de,\delta}$ and is
such that $r_x(k_{x_1}) \in \Sigma_{\theta}$, for some $\theta$. So $x$ belongs to the
saturated of $\Sigma$ by $\fol$
\end{proof}

We are now able to prove the main result of this section.

\begin{proof}[Proof of Theorem~\ref{TMMhigher}]
Let $X, \, Y$ be two holomorphic vector fields as stated above and denote by $\fol_X$ (resp.
$\fol_Y$) the foliation associated to $X$ (resp. $Y$). We can assume that $X$
and $Y$ are written in its normal form (\ref{normalform}). Denote by $(\dl_1, \ldots,\dl_n)$
(resp. $(\db_1,\ldots,\db_n)$) the eigenvalues of $DX(0)$ (resp. $DY(0)$). Note that we can
assume $\dl_1 = 1 = \db_1$ simultaneously. Denote by $h^X$ (resp. $h^Y$) the holonomy of $X$
(resp. $Y$) relative to the $x_1$-axis. It is obvious that if $\fol_X, \, \fol_Y$ are
analytically equivalent then $h_X, \, h_Y$ are analytically conjugated. So let us assume that
$h_X, \, h_Y$ are analytically conjugated. It will be constructed a holomorphic diffeomorphism
defined on a neighborhood of the origin and taking the leaves of $\fol_X$ into the leaves of
$\fol_Y$.

Let $P_{\de,\delta}$, $l$ and $L$ be as in the proof of Lemma~\ref{extensao}. We assume that
$\de (< 1)$ is such that Inequality~(\ref{ab}) is satisfied in $P_{\de,\delta}$ and
also that the projection $p_1$ is transverse to the leaves of $\fol_X \cap P_{\de,\delta}$
and of $\fol_Y \cap P_{\de,\delta}$ with exception to the ones contained in the invariant
manifold $\{x_1=0\}$.

Let $c$ be the loop around the singular point defined above ($c(\theta) = (\de e^{i\theta},0,\ldots,0)$)
and let $\Sigma_{\theta}^X, \, \Sigma_{\theta}^Y$ be the ``vertical" transverse sections to the
common separatrix $\{x_i = 0, i \geq 2\}$ of $\fol_X, \, \fol_Y$ at the point $c(\theta)$.

Let $h_0: \Sigma_0^X \rightarrow \Sigma_0^Y$ be the analytic diffeomorphism conjugating $h_1^X$
and $h_1^Y$. Denote by $l_{\theta}: \Sigma_0^X \rightarrow \Sigma_{\theta}^X$ (resp.
$\bar{l_{\theta}}: \Sigma_0^Y \rightarrow \Sigma_{\theta}^Y$) the applications obtained by
lifting the curve $c$ to the leaves of $\fol_X$ (resp. $\fol_Y$). Consider the holomorphic
diffeomorphism $h_t = \bar{l}_t \circ h_0 \circ l_t^{-1}$. Note that this family of
diffeomorphisms is well defined in the sense that $h_0$ coincides with $h_{2\pi}$.
In fact the conjugacy relation $\bar{l}_{2\pi} \circ h_0 = h_0 \circ l_{2\pi}$ (note
that $l_{2\pi}, \, \bar{l}_{2\pi}$ represent the holomomy map of $\fol_X, \, \fol_Y$,
respectively) together with the fact that $h_{2\pi}=\bar{l}_{2\pi} \circ h_0 \circ l_{2\pi}^{-1}$
implies that $h_{2\pi}=h_0$. Then it has been constructed a diffeomorphism defined on a transverse
section to the $x_1$-axis along $c$ taking the leaves of $\fol_X$ into the leaves of $\fol_Y$.
It remains to extend this diffeomorphism to a neighborhood of the origin. This will be done
showing that this diffeomorphism can be transported along the spiral curves $r_x$.

Denote by $r^X_{x}$ (resp. $r^Y_{x}$) the lift of $r_{x_1}$ to the leaf of $\fol_X$ (resp.
$\fol_Y$) through the point $x = (x_1, \ldots, x_n)$. Lemma~\ref{extensao} ensures that
for all $x$ in $P_{\de, \delta}$ there exists $\theta \in [0, 2\pi]$ such that
$r^X_{x}(k_{x_1}) \in \Sigma_{\theta}^X$ as also $r^Y_{x}(k_{x_1}) \in \Sigma_{\theta}^Y$.
This is equivalent to saying that the diffeomorphism constructed above can also be transported
over the spiral curves $r_x^X, \, r_x^Y$. Denote by $\Phi$ the resulting diffeomorphism.

An analytic equivalence between $\fol_X \setminus \{x_1=0\}$ and $\fol_Y \setminus \{x_1=0\}$
has been constructed in a neighborhood of the origin. Let us now prove that this equivalence
can be extended to the invariant manifold $\{x_1 = 0\}$. Since $\Phi$ is holomorphic in
$P_{\de,\delta} \setminus \{x_1 = 0\}$ and since $\{x_1 = 0\}$ is a thin set, it is sufficient
to check that $\Phi$ is bounded (cf. \cite{G}). In order to do this we need to revisit the
construction of $\Phi$.

From now on we use the variables $x$ to $\fol_X$ and $y$ to $\fol_Y$. Fix a point $x^0 \in P_{\de,\delta}$
and let $x_1(t) = x_1^0 e^{\frac{t}{v}}$ be a spiral curve in the complex plane identified with the $x_1$-axis.
Let $r^X_{x^0}(t)$ (resp. $r^Y_{x^0}(t)$) be the lift of the spiral curve to the leaf of $\fol_X$ (resp. $\fol_Y$)
through the point $x^0$. This lift is such that its components are given by
\begin{align*}
|x_i(t)| &= x_i^0 e^{\int_0^t \frac{\dl_i}{v} (1 + A_i(x_1^0e^{\frac{t}{v}},x_2(t),\ldots,x_n(t)))}\\
|y_i(t)| &= y_i^0 e^{\int_0^t \frac{\dl_i}{v} (1 + B_i(x_1^0e^{\frac{t}{v}},y_2(t),\ldots,y_n(t)))}
\end{align*}
for all $2 \leq i \leq n$ and holomorphic functions $A_i, \, B_i$ such that $A_i(0) = B_i(0) = 0$. In this
way we have that
\[
\frac{y_i(t)}{x_i(t)} = \frac{y_i^0}{x_i^0} e^{\int_0^t \frac{\dl_i}{v}(B_i - A_i)}
\]
It is now sufficient to prove that $\int_0^t \frac{\dl_i}{v}(B_i - A_i)$ converges.

On $B_i - A_i$ the function $B_i$ is evaluated at the point $(x, y_2, \ldots, y_n)$ while $A_i$ is evaluated
at $(x, x_2, \ldots, x_n)$. Let us sum and subtract $A_i(x, y_2, \ldots, y_n)$ on the expression of $B_i - A_i$.
The new function has a Lipshitz component relative to the variables $x_i, \, y_i$ for $i \geq 2$ (the component
$A_i(x, y_2, \ldots, y_n) - A_i(x, x_2, \ldots, x_n)$). The remaining expression can be written as the sum of a
holomorphic function $F$ that only depends on $x$ and another one whose Taylor expansion does not have monomials
only depending on $x$. We have to estimate the integral above on each case separately.

Let us begin with the Lipshitz component. We have that
\[
|A_i(x_1,y_2,\ldots, y_n) - A_i(x_1,x_2, \ldots,x_n)| \leq \sum_{i=2}^n k_i |y_i - x_i| \, .
\]
We first estimate the value of $|y_i(t) - x_i(t)|$ for all $i \geq 2$. Taking account that
$|y_i(t) - x_i(t)| \leq |y_i(t)| + |x_i(t)|$ we estimate the absolute value of each variable
separately. We have
\begin{align*}
\dfrac{d}{dt}|x_i(t)|^2 &= 2\Re (\overline{x_i(t)}x_i^{\prime}(t))\\
&=2 |x_i(t)|^2 \Re \left( \frac{\dl_i}{v} (1+A_i(x_1^0e^{\frac{t}{v}},y(t),z(t))) \right)
\end{align*}
Inequality~\ref{ab} implies that $| \Re \left( \frac{\dl_1}{v} A_i \right)|$ is less than or equal to
$-\frac{\Re \left( \frac{\dl_1}{v} \right)}{2}$. In this way
\[
\dfrac{d}{dt}\log |x_i(t)|^2 \leq \Re \left(\frac{\dl_i}{v}\right) \, \,
\Longrightarrow \, \,
|x_i(k_{x_1^0})| \leq |x_i^0|e^{\frac{1}{2} \Re (\frac{\dl_i}{v}) t}
\]
The same inequality can be deduced for $|y_i|$. We conclude therefore that
\[
|A_i(x_1,y_2,\ldots, y_n) - A_i(x_1,x_2, \ldots,x_n)| \leq K e^{-\frac{1}{2} \da t}
\]
for some positive constants $K$ and $\da$. Let us estimate now the value of the integral of the Lipshitz component.
We have
\begin{align*}
&\left| \int_0^t \frac{\dl_i}{v} [A_i(x_1^0e^{\frac{s}{v}},y_2(s),\ldots, y_n(s)) - A_i(x_1^0e^{\frac{s}{v}},x_2(s), \ldots,x_n(s))] ds \right| \\
&\leq \int_0^t \left|\frac{\dl_i}{v} [A_i(x_1^0e^{\frac{s}{v}},y_2(s),\ldots, y_n(s)) - A_i(x_1^0e^{\frac{s}{v}},x_2(s), \ldots,x_n(s))] \right| ds \\
&\leq \int_0^t K \left| \frac{\dl_i}{v} \right| e^{-\frac{1}{2} \da s} ds\\
&= -\left| \frac{\dl_i}{v} \right| \frac{2K}{\da} \left( e^{-\frac{1}{2}\da t} - 1 \right)
\end{align*}
and, consequently, the integral converges.

Let us consider now the holomorphic component $F$ depending only on $x_1$. Suppose that the Taylor's expansion of $F$
around the origin is given by $F(x) = \sum a_j x^j$. Then
\begin{align*}
\int_0^t \frac{\dl_i}{v} F(x_1^0e^{\frac{s}{v}}) ds &= \frac{\dl_i}{v} \int_0^t \sum a_j (x_1^0)^j e^{\frac{js}{v}} ds\\
&= \dl_i \left( G(x_1^0 e^{\frac{t}{v}}) - G(x_1^0) \right)
\end{align*}
where the Taylor's expansion of $G$ is given by $G(x) = \sum \frac{a_j}{j} x^j$. Since $F$ is holomorphic, so is $G$.
In particular the above integral converges.

Finally, in the remaining case we have that the holomorphic function is bounded by a linear combination of the absolute
value of $y_i(t)$, i.e. it is bounded by $\sum_{i=2}^n c_i |y_i|$. This is a simple consequence of the fact that the last
component is holomorphic on a neighborhood of the origin and the estimate of the variables $y_i$. The convergence of
the integral in that case follows as in the Lipshitz case.

We have just proved that both $\Phi^{-1}$ and $\Phi$ are bounded. Thus
$\Phi$ admits a holomorphic extension $\tilde{\Phi}$ to the invariant
manifold $\{x_1=0\}$. As $\tilde{\Phi}$ has a holomorphic inverse map
(the inverse map is constructed taking now the leaves of $Y$ into the
leaves of $X$ in a similar way), $\tilde{\Phi}$ is a diffeomorphism in a
neighbourhood of the origin.

It remains be proved that $\tilde{\Phi}$ takes the leaves of
$\fol_X|_{\{x_1=0\}}$ into the leaves of $\fol_Y|_{\{x_1=0\}}$.
Let $Z=D\tilde{\Phi} (X \circ \tilde{\Phi})$. As
$\fol_Y|_{P_{\de,\delta} \setminus \{x_1=0\}}$ coincides with
$\fol_Z|_{P_{\de,\delta} \setminus \{x_1=0\}}$, there exists a
holomorphic function $f$, defined on $P_{\de,\delta} \setminus
\{x_1=0\}$, such that $fY=Z$. In particular $f=\frac{Z_2}{Y_2}$, where
$Y_2$ ($Z_2$) is the second component of $Y$ ($Z$). As
$Y_2=x_2(\db_2+\ldots)$, where dots means terms of order greater than or
equal to $1$, and both $Y_2$ and $Z_2$ are holomorphic, $f$ can be
holomorphically extended to $U \setminus \{x_1=0,x_2=0\}$. Finally, as
$\{x_1=0,x_2=0\}$ is a set of complex codimension greater than $1$, $f$
admits a holomorphic extension, $\tilde{f}$, to $\{x_1=0,x_2=0\}$
\cite[pag. 31]{G} and this extension verifies
$\tilde{f}Y=Z$ in $U$. Thus $X$ and $Y$ are analytically equivalent.
\end{proof}


\section{Saddle-Node in Higher Dimension}

The classification of saddle-nodes in dimension~$2$ has been discussed in Section 2.3.
We shall now approach the case of a saddle-node singularity in higher dimension. In
dimension~$2$ a singularity of saddle-node type has exactly one eigenvalue equal to
zero. In higher dimension we can have more eigenvalues equal to zero. We will consider
here the case of the so called codimension~$1$ saddle-nodes, i.e. singularities with
exactly one eigenvalue equal to zero.

More specifically, we will consider saddle-node foliations whose representatives belongs
up to a linear change of coordinates to $\mathfrak{X}$, the subset of holomorphic vector
fields on $(\C^n,0)$ with an isolated singularity at the origin whose linear part at
the singular point has eigenvalues $\dl_1, \ldots, \dl_n$ such that $\dl_1 = 0$ and
$0 \not \in \mathcal{H}(\dl_2,\ldots,\dl_n)$. By $\mathcal{H}(\dl_2,\ldots,\dl_n)$ we
mean the convex hull of $\{\dl_i: i=2,\ldots,n\}$. Furthermore there are no resonance
relations between the non-vanishing eigenvalues.

Note that for $n=3$ the considered vector fields are generic between the vector fields with an
isolated singular point that is a codimension~$1$ saddle-node. It is not difficult to verify
that the elements of $\mathfrak{X}$ are analytically equivalent to a vector field of the form
\begin{equation}\label{campo}
Y_p:
\begin{cases}
\dot{x_1}=x_1^{p+1}\\
\dot{x_i}=\dl_i x_i+x_1a_i(x), \quad i=2,\ldots,n
\end{cases}
\end{equation}
where $x=(x_1,\ldots, x_n)$ and where $a_i$ are holomorphic functions such
that $a_i(0)=0$, $\forall i=2,\ldots,n$. This corresponds to the Dulac's normal
form for a codimension~$1$ saddle-node in $\C^n$.

Let us assume that $p=1$. The case $p>1$ can be treated in a similar way but some comments
will be made at the end of this section. For simplicity we denote by $Y_{1,\da}$ a vector
field of the type
\[
Y_{1,\da}:
\begin{cases}
\dot{x_1}=x_1^2\\
\dot{x_i}=x_i(\gamma_i +\da_i x_1)+x_1h_i(x), \quad i=2, \ldots,n
\end{cases}
\]
where $\da=(\da_2,\ldots,\da_n) \in \R^{n-1}$ and $h_i$ are holomorphic functions such that
$\dfrac{\partial h_i}{\partial x_i}|_{0}=0$, $\forall i=2,\ldots,n$. We assume further
$\gamma_2 = 1$.

\begin{lema}
The vector field $Y_{1,\da}$ is formally conjugated to
\begin{equation}\label{zda}
Z_{\da}:
\begin{cases}
\dot{x_1}=x_1^2\\
\dot{x_i}=x_i(\gamma_i +\da_i x_1), \quad i=2, \ldots, n
\end{cases}
\end{equation}
\end{lema}


Summarizing, there exists a formal change of coordinates of the form
\begin{equation}\label{mudformal}
\hat{H}(x)=(x_1,x_2+\sum_{k=1}^\infty f_{2k}(\bar{x})x_1^k,\ldots,
x_n+\sum_{k=1}^\infty f_{nk}(\bar{x})x_1^k)
\end{equation}
conjugating $Y_{1,\da}$ and $Z_{\da}$. The functions $f_{ik}$ on (\ref{mudformal}) are
holomorphic functions on a neighborhood of $0 \in \C^{n-1}$ verifying $f_{i1}(0)=0$
for all $i \in \{2,\ldots,n\}$. Let us denote by $\hat{G}_0$ the set of formal maps of type
above.

Although formally conjugated $Y_{1,\da}$ and $Z_{\da}$ are not, in general, analytically
conjugated. In fact, the formal changes of coordinates presented above are, in general,
divergent. Nonetheless $Y_{1,\da}$ and $Z_{\da}$ are analytically conjugated by sectors
as the next result states. We note that the union of the sectors were $Y_{1,\da}$ and
$Z_{\da}$ are analytically conjugated constitutes a neighborhood of the origin.

\begin{teo}[Theorem of Malmquist]\cite{Malm,tesecanille}
Let $\hat{H}$ be the unique formal transformation of the form (\ref{mudformal})
conjugating $Y_{1,\da}$ and $Z_{\da}$. Then there exists a holomorphic transformation $H$
defined in $S \times (\C^{n-1},0)$, where $S$ a sector with vertex at the origin of $\C$
and angle less then $2 \pi$, such that
\begin{itemize}
\item[a)] $dH(Y_{1,\da})=Z_{\da}(H)$, in $S \times
(\C^{n-1},0)$
\item[b)] $H \tilde{\rightarrow} \hat{H}$ in $S$, as $x_1
\rightarrow 0$
\end{itemize}
The union of the sectors satisfying the conditions above constitutes a neighborhood of the
origin $0 \in \C$.
\end{teo}

A holomorphic map $H$ as above is called a normalizing application. Denoting by $S_i$ the
different sectors covering the neighborhood of $0 \in \C$ nd by $H_i$ the correspondent
normalizing application we note that $Y_{1,\da}$ and $Z_{\da}$ are analytically conjugated
if and only if $H_i=H_j$ in $S_i \cap S_j$, $\forall i \ne j$. The description of $H_i$ on
the corresponding sector goes as in Section~$2.3$.

\subsection{Sectorial Isotropy of the formal normal form}

The study of the Sectorial Isotropy for saddle-node foliations as above was considered in
\cite{tesecanille}. In that work the properties of $(H_i \circ H_j^{-1})$ on the overlaps
$(S_i \cap S_j) \times (\C^{n-1},0)$ are deduced, leading us towards a classifications of
codimension~$1$ saddle-nodes as above.

Let us begin this section with a description of the solutions of the formal normal form. The
solutions of (\ref{zda}) out of $\{0\} \times (\C^{n-1},0)$ are given by
\[
x_j(x_1)=c_jx_1^{\da_j} e^{-\frac{\gamma_j}{x}}, \quad j \geq 2
\]
where $(c_2,\ldots,c_n) \in \C^{n-1}$. Basically they are parameterized by the constants
$(c_2,\ldots,c_n)$. The main objective in this subsection is to relate the solutions of $Z_{\da}$
with the solutions of $Y_{1,\da}$ on each sector given by the Theorem of Malmquist.

Denote by $\varphi_i$, for $i=2,\ldots,n$, the argument of the eigenvalue $\gamma_i$ and let
$x_1=re^{i\theta}$. Since
\[
x_j(re^{i\theta})=c_jr^{\da_j}e^{i\theta \da_j}e^{-\frac{|\gamma_j|}{r}
(\cos(\varphi_j-\theta) +i\sin(\varphi_j-\theta))}
\]
for a fixed $\theta$, the behavior of $x_j(x_1)$ along the curve $x_1=re^{i\theta}$ as
$r \rightarrow 0$ is given by the term $\frac{|\gamma_j|}{r}\cos(\varphi_j-\theta)$.
More specifically, if $\cos(\varphi_j-\theta) > 0$ (resp. $\cos(\varphi_j-\theta) < 0$)
then $x_j(re^{i\theta})$ goes to zero (resp. infinity) as $r$ goes to zero.

A sector where $(x_2(x_1),\ldots,x_n(x_1))\rightarrow (0,\ldots,0)$ as $r \rightarrow 0$
is called an attractor sector. This kind of sector is constituted by the directions $\theta$
for which $\cos(\varphi_j-\theta)>0$, $\forall j=2,\ldots,n$. A sector where $\cos(\varphi_j-\theta)<0$,
$\forall j=2,\ldots,n$, is called a saddle sector (in this case $|x_j(x_1)| \rightarrow \infty$,
$\forall j=2,\ldots,n$). Contrary to the case of saddle-nodes in $\C^2$ there exists, in general,
sectors that are neither attractors nor saddles. Those sectors are called mixed and are characterized
by the condition $\cos(\varphi_i-\theta)\cos(\varphi_j-\theta)<0$, for some $i \ne j$. In fact mixed
sectors does not exists only in the case that $\gamma_i \in \R$ (or $\R^+$ since we are assuming
$\gamma_2, \ldots, \gamma_n$ in the Poincar\'e domain).

The directions for which there exists $j$ such that $\cos(\varphi_j-\theta)=0$ are called singular
directions of the solution. They are given by $\theta=\varphi_j \pm \frac{\pi}{2}$, $j=2,\ldots,n$.
For simplicity we sometimes say that $\theta \in S$ instead of that $x=re^{i\theta} \in S$.

\subsection{The sectors where the Theorem of Malmquist is valid}

Let $S$ be a sector of $\C$ with vertex at the origin and angle less then $2\pi$.
Denote by $\Lambda_{Z_{\da}}(S)$ the group of holomorphic transformations $H: S
\times (\C^{n-1},0) \rightarrow S \times (\C^{n-1},0)$ verifying:
\begin{itemize}
\item[a)] $dH(Z_{\da})=Z_{\da}(H)$
\item[b)] $H$ is asymptotic to the identity $Id$ as $x_1 \rightarrow 0$, $x_1 \in S$
\end{itemize}
$\Lambda_{Z_{\da}}(S)$ is a presheaf and we denote by $\Lambda_{Z_{\da}}$ the sheaf
associated to the given presheaf.

An element $H$ tangent to the identity and preserving the first component can be written
in the form
\begin{equation}\label{formafeixe}
\begin{aligned}
H(x)=(x_1,x_2+a_{20}(x_1)+&\sum_{|Q|\geq 1}a_{2Q}(x_1)\bar{x}^Q,
\ldots,\\
&x_n+a_{n0}(x_1)+ \sum_{|Q|\geq 1}a_{nQ}(x_1)\bar{x}^Q)
\end{aligned}
\end{equation}
where $Q=(q_2,\ldots,q_n)$, $|Q|=q_2+\ldots+q_n$ and $\bar{x}^Q=x_2^{q_2}\ldots x_n^{q_n}$.
We look for conditions under which $H \in \Lambda_{Z_{\da}}(S)$. In order to do that we need
to interpret the conditions a) and b) above.

Condition a) expresses that the leaves of $Z_{\da}|_{S \times (\C^{n-1},0)}$ are preserved by
the elements of $\Lambda_{Z_{\da}}(S)$. This implies that
\begin{equation}\label{expa}
a_{jQ}(x_1)=a_{jQ}x_1^{-((Q,\da)-\da_j)}e^{\frac{(Q,\gamma)-\gamma_j}{x_1}}
\end{equation}
where $\gamma=(\gamma_2,\ldots,\gamma_n)$.

Condition b) says that $H \tilde{\rightarrow} Id$ as $x_1$ goes to~$0$, with $x_1 \in S$. This is
simply equivalent to say that $a_{jQ}(x_1)\tilde{\rightarrow} 0$ as $x_1$ goes to~$0$, with $x_1 \in S$,
for all $j \in \{2,\ldots,n\}$ and for all $Q\in \N_0^{n-1}$.

Denote by $\varphi_{jQ}$ the argument of the complex number $(Q,\gamma)-\gamma_j$ and let $x_1=re^{i\theta}$.
For a fixed $\theta$ the behavior of $a_{jQ}(x_1)$ as $r \rightarrow 0$, is given by $\cos (\varphi_{jQ}-\theta)$.
The directions $\theta$ for which $\cos (\varphi_{jQ} -\theta)=0$ for some $j$ and $Q$, are called singular
directions of the sheaf $\Lambda_{Z_{\da}}$ and are given by $\theta_{jQ}^{\pm} =\varphi_{jQ} \pm \frac{\pi}{2}$,
$j=2,\ldots,n$. We note that the singular directions of the solution are singular directions of the sheaf either.
In fact the eigenvalue $\gamma_j$ corresponds to the complex number $\varphi_{j0}$.

In order to study the behavior of the arguments of $(Q,\gamma)-\gamma_j$, for $Q \in \N_0^{n-1}$, we shall
represent all these numbers in the complex plane (see figure \ref{novo}). We note that the singular
directions of the sheaf are dense in the mixed sectors, while they are discrete in the attractor and saddle
ones. Although discrete in the attractor and saddle sectors they accumulate on the singular directions of
the solution.

\begin{figure}[ht!]
\centering
\includegraphics{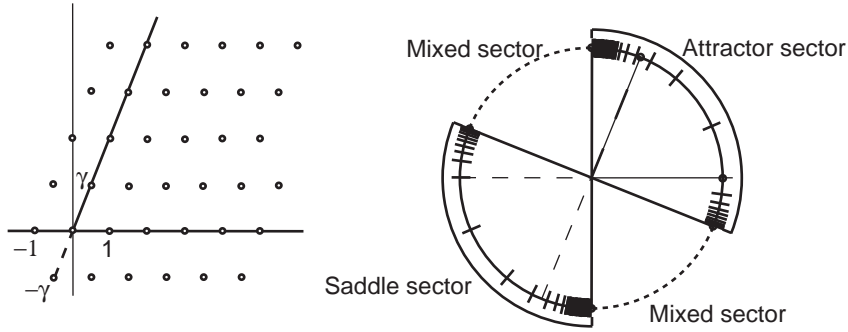}
\caption{On the left the complex numbers $(Q,\gamma)-\gamma_j$. On the right the singular
directions of the sheaf $\Lambda_{Z_{\da}}$ which are dense in the mixed sectors.}\label{novo}
\end{figure}

Let us now describe how to construct the sectors where the Theorem of Malmquist is valid. First of all
let us consider a direction $\varphi_0$ in the attractor sector that is not a singular direction of the
sheaf $\Lambda_{Z_{\da}}$. The sectors where the Theorem of Malmquist is valid are the sectors obtained
by extending the sectors between the angles $\varphi_0$ and $\varphi_0 \pm \pi$ till reach a singular
direction of the sheaf $\Lambda_{Z_{\da}}$ (see (figure \ref{escolhasector})). Since singular direction
are discrete on the attracting and saddle sectors, both sectors are well defined and have amplitude greater
than $\pi$. Denote each one of this sectors by $S_1$ and $S_2$.

By definition $S_1 \cap S_2$ is the union of two open sets, $S_+$ and $S_-$, contained in the attractor
sector and in the saddle sector respectively. In particular $S_+ \cap S_- = \emptyset$. The saddle and
attractor sectors, denoted by $S_s$ and $A_s$ respectively, are antipodes, i.e. $S_s = A_s + \pi =
\{e^{\pi i}a : a \in A_s\}$. The sets $S_+$ and $S_-$ are also antipodes.

\begin{figure}[ht!]
\centering
\includegraphics{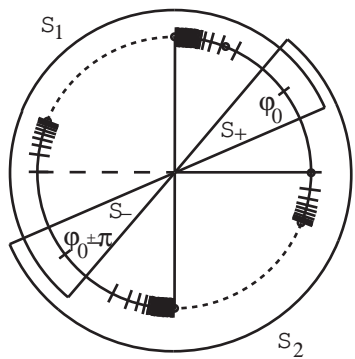}
\caption{How to construct sectors where the Theorem of Malmquist is valid}\label{escolhasector}
\end{figure}

\subsection{The importance of the pre-sheaves $\Lambda_{Z_{\da}}(S_+)$ and $\Lambda_{Z_{\da}}(S_-)$}

Let $\hat{H}$ be the unique formal diffeomorphism of $\C \{\bar{x}\}[[x_1]]$ conjugating $Y_{1,\da}$
and $Z_{\da}$.

\begin{prop}
Let $S_1$ and $S_2$ be the sectors where the Theorem of Malmquist is valid. Denote by $H_1$ (resp. $H_2$)
the normalizing application defined on $S_1$ (resp. $S_2$). Then, $H_j \circ H_i^{-1}|_{S_+}$ (resp.
$H_j \circ H_i^{-1}|_{S_-}$) belongs to $\Lambda_{Z_{\da}}(S_+)$ (resp. $\Lambda_{Z_{\da}}(S_-)$).
\end{prop}

Let $g_{+}$ (resp. $g_{-}$) be the restriction of $H_2 \circ H_1^{-1}$ to $S_{+}$ (resp. $S_{-}$). We have that
$Y_{1,\da}$ is analytically conjugated to its formal normal form $Z_{\da}$ if and only if $H_1$ coincides
with $H_2$ on the overlaps, i.e. if and only if both $g_+$ and $g_-$ reduces to the identity. The diffeomorphisms
$g_+, \, g_-$ represents the gluing of the leaves on the overlaps. As already mentioned they take the
form~\ref{formafeixe} with $a_{jQ}(x_1)=a_{jQ}x_1^{-((Q,\da)-\da_j)}e^{\frac{(Q,\gamma)-\gamma_j}{x_1}}$ and
$a_{jQ}(x_1)$ asymptotic to zero as $x_1$ goes to zero.

\subsection{Gluing of the leaves}

In order to describe the gluing of the leaves on the overlaps, it is important to obtain a description of the
maps of the form (\ref{formafeixe}) on $S_+$ and $S_-$. In other words, we shall obtain a description of the
sheaves $\Lambda_{Z_{\da}}(S_+)$ and $\Lambda_{Z_{\da}}(S_-)$. The next two propositions describe their
fundamental properties:

\begin{prop}\label{condsuf}
Fix a sector $S$ and assume that $a_{jQ} \ne 0$ on $S$. Then $\cos(\varphi_{jQ}-\theta)<0$, for all
$\theta$ such that $re^{i\theta} \in S$.
\end{prop}

\begin{proof}
Since the elements of $\Lambda_{Z_{\da}}(S)$ are asymptotic to the Identity it follows that
$a_{jQ}(x_1) \tilde{\longrightarrow} 0$ as $x_1$ goes to $0$. Let $x_1 = re^{i\theta}$. For a
fixed $\theta $ the behavior of $a_{jQ}(x_1)$, as $r$ goes to~$0$, is given by the real part of
$\frac{(Q,\gamma)-\gamma_j}{x_1}$, i.e. by $\dfrac{|(Q,\gamma)-\gamma_j|}{r} \cos(\varphi_{jQ}-\theta)$.

Suppose that $a_{jQ} \ne 0$. Assume for a contradiction that there exists $\theta \in S$ such that $\cos(\varphi_{jQ}-\theta)>0$.
Then
\[
\dfrac{|(Q,\gamma)-\gamma_j|}{r}\cos(\varphi_{jQ}-\theta)
\stackrel{r\rightarrow 0} \longrightarrow +\infty \,
\]
which means that $a_{jQ}(x_1)$ is not asymptotic to the zero function, contradicting our assumption.
\end{proof}

\begin{prop}
There exists a duality between $\Lambda_{Z_{\da}}(S)$ and $\Lambda_{Z_{\da}}(S+\pi)$ in the following sense:
if $a_{jQ} \ne 0$ in $S$ then $a_{jQ}=0$ in $S+\pi$. Note that by $S + \pi$ we mean the antipode of $S$.
\end{prop}

In particular there exists a duality between $\Lambda_{Z_{\da}}(S_+)$ and $\Lambda_{Z_{\da}}(S_-)$. Therefore
we just need to know the constants that can be non zero in $\Lambda_{Z_{\da}}(S_+)$, i.e. we need to know for
which pair $(j,Q)$ we have $\cos(\varphi_{jQ}-\theta)<0$, for all $\theta$ such that $re^{ i\theta} \in S_+$
($r \in \R^+$).

The next  result expresses how the gluing of the leaves is done.

\begin{prop}\label{transform}
Fix a sector $S$ and take an element $H$ of the presheaf $\Lambda_{Z_{\da}}(S)$. Thus $H$ transforms the
solution of the differential equation associated to the formal normal form given by:
\[
x_j(x_1)=c_jx_1^{\da_j}e^{-\frac{\gamma_j }{x_1}}, \quad j \geq 2
\]
into the solution of the same equation given by:
\[
x_j(x_1)=(c_j+a_{j0}+\sum_{|Q|\geq 1}a_{jQ}c^Q)x_1^{\da_j}e^{-\frac{\gamma_j }{x_1}}, \quad j\geq 2
\]
where $c^Q=c_2^{q_2}\ldots c_n^{q_n}$.
\end{prop}

Fixed a sector $S$, each $c=(c_2,\ldots,c_n) \in (\C^{n-1},0)$ represents a leaf of the foliation
associated to the restriction of $Z_{\da}$ to $S \times (\C^{n-1},0)$. More specifically, $c$ works
like a parametrization of the leaves in $S$. If $S$ does not contain singular directions of the sheaf,
we can identify $\Lambda_{Z_{\da}}(S)$ with the set of transformations in the space of the leaves
given by
\[
\{c \mapsto (c_2+a_{20}+\sum_{|Q|\geq 1}a_{2Q}c^Q,\ldots,
c_n+a_{n0}+\sum_{|Q|\geq 1}a_{nQ}c^Q)\}
\]
We also denote this set by $\Lambda_{Z_{\da}}(S)$. The presheaf $\Lambda_{Z_{\da}}(S)$ expresses then
that the leaf of $Z_{\da}|_{S \times (\C^{n-1},0)}$ parameterized by $(c_2,\ldots,c_n)$ is taken into the
leaf parameterized by $(c_2+a_{20}+\sum_{|Q|\geq 1}a_{2Q}c^Q, \ldots, c_n+a_{n0}+\sum_{|Q|\geq 1}a_{nQ}c^Q)$.
Since $S_+$ (resp. $S_-$) does not contain singular direction of the sheaf the interpretation above can be given
to the presheaf $\Lambda_{Z_{\da}}(S_+)$ (resp. $\Lambda_{Z_{\da}}(S_-)$).

Let us now explain how $\Lambda_{Z_{\da}}(S_+)$ can be determined for a given sector $S_+$. Note that
$S_1, \, S_2$ are not uniquely determined and so do $S_+, \, S_-$. The first step is to choose sectors
$S_1$ and $S_2$ or, equivalently, sectors $S_+$ and $S_-$ we are going to work on. Then we consider the
set of complex numbers $C = \{(Q,\gamma)-\gamma_i:i=2,\ldots,n, Q\in \N_0^{n-1}\}$. We can assume without
 loss of generality that $0 = \arg(\gamma_2) \leq \ldots  \leq \arg(\gamma_n)<\pi$.

Denote by $K$ the sector, with vertex at the origin, whose elements have arguments between $0 = \arg(\gamma_2)$
and $\arg(\gamma_n)$. Then we choose two directions not contained in $K$ such that the two of the four sectors
defined by those directions do not contain any element of $C$ in its interior (figure \ref{geometria1}).
Fix one of those two sectors and denote it by $S$. If $S + \frac{\pi}{2}$ is contained in the attractor
sector we take $S_+ = S + \frac{\pi}{2}$ otherwise we take $S_+ = S - \frac{\pi}{2}$. Note that $S$ can be
chosen so close to the real axis as we want since we can chose the two directions above so close to $\pi$ as
we want.

\begin{figure}[ht!]
\centering
\includegraphics{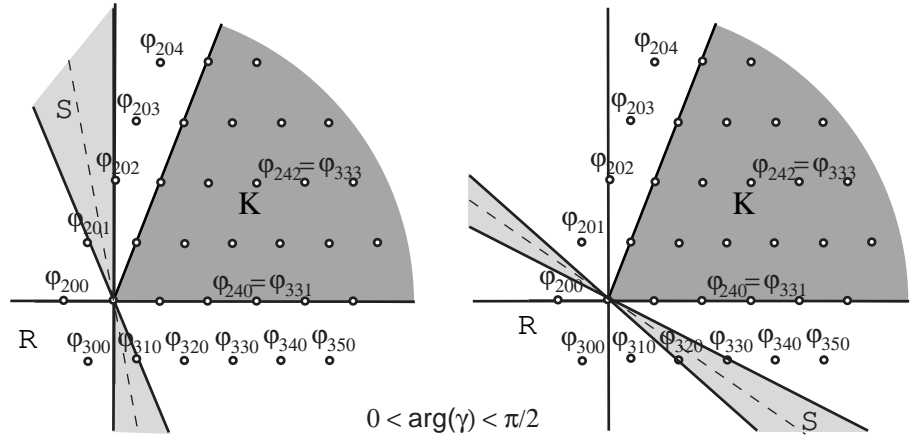}
\caption{Each point represents an element of the form $(Q,\gamma)-\gamma_i$ for some $i=2,\ldots,n$ and $Q \in
\N^{n-1}$}\label{geometria1}
\end{figure}

The constants $a_{jQ}$ that can be non zero in $\Lambda_{Z_\da}(S_+)$ are those for which $(Q,\gamma)-\gamma_j$ is
in the half plane, defined by the bisectrix of $S$, not containing $S_+$. Denote this region by $R$. For example,
if we look at figure \ref{geometria1}, on the left case $\Lambda_{Z_{\da}}(S_+)$ is given by
\[
\{(y,z) \mapsto (y+a_{200}+a_{201}z,z+a_{300})\}
\]
while on the right one $\Lambda_{Z_{\da}}(S_+)$ is given by
\[
\{(y,z) \mapsto (y+a_{200},z+a_{300}+a_{310}y+a_{320}y^2)\}.
\]
More specifically, the monomial coefficient of $c^Q$ on the $(i-1)^{th}$-component of $g_+$ can be non-zero if and only
if $(Q,\gamma)-\gamma_i \in R$. We should note that the elements of $\Lambda_{Z_{\da}}(S_+)$ are always polynomial while
the elements $\Lambda_{Z_{\da}}(S_-)$ are always tangent to the identity, i.e, $a_{i0}$ must be equal to $0$ in $S_-$ for
all $i=2,\ldots,n$.

In the case that $0 = \arg(\gamma_2) = \ldots = \arg(\gamma_n)$ the constants $a_{jQ}$ that can be non zero in
$\Lambda_{Z_\da}(S_+)$ are those such that $\arg((Q,\gamma)-\gamma_j) = \pi$. In particular, we have only one
element $Q$ such that $\arg((Q,\gamma)-\gamma_j) = \pi$ in the case of a saddle-node in dimension~$2$. Therefore
the element on $S_+$ is just a translation as already mentioned.

\end{document}